\documentclass{amsart}
\usepackage[utf8x]{inputenc}
\usepackage{amsmath, amsthm, amssymb}
\usepackage[usenames,dvipsnames,svgnames,table]{xcolor}
\usepackage[margin=1.3in]{geometry}
\usepackage{cleveref}

\newtheorem{theorem}{Theorem}[section]
\newtheorem{proposition}[theorem]{Proposition}
\newtheorem{lemma}[theorem]{Lemma}

\newcommand{\N}{\mathbb N}

\newcommand{\R}{\mathbb R}

\newcommand{\T}{\mathbb T}
\newcommand{\eps}{\varepsilon}

\newcommand{\dd}{\, \mathrm{d}}

\newcommand{\tr}{\mbox{tr}}

\newcommand{\vv}{\langle v \rangle}

\usepackage{tikz}
\usetikzlibrary{arrows}

\numberwithin{equation}{section}

\title{Global existence for an isotropic modification of the Boltzmann equation}

\author{Stanley Snelson}
\address{Department of Mathematical Sciences, Florida Institute of Technology, Melbourne, FL 32901}
\email{ssnelson@fit.edu}

\thanks{The author was partially supported by NSF grant DMS-2213407.}

\begin{document}

\maketitle

\begin{abstract}
Motivated by the open problem of large-data global existence for the non-cutoff Boltzmann equation, we introduce a model equation that in some sense disregards the anisotropy of the Boltzmann collision kernel. We refer to this model equation as {\it isotropic Boltzmann} by analogy with the isotropic Landau equation introduced by Krieger and Strain [Comm. Partial Differential Equations 37(4), 2012, 647--689]. The collision operator of our isotropic Boltzmann model converges to the isotropic Landau collision operator under a scaling limit that is analogous to the grazing collisions limit connecting (true) Boltzmann with (true) Landau. 

Our main result is global existence for the isotropic Boltzmann equation in the space homogeneous case, for certain parts of the ``very soft potentials'' regime in which global existence is unknown for the space homogeneous Boltzmann equation. The proof strategy is inspired by the work of Gualdani-Guillen [J. Funct. Anal. 283(6), 2022, Paper No. 109559] on isotropic Landau, and makes use of recent progress on weighted fractional Hardy inequalities.\\
\noindent \textbf{MSC(2010)}: 35Q20, 35B60, 46N20.\\
\noindent \textbf{Keywords}: Boltzmann equation, global existence, fractional Hardy inequality
\end{abstract}

\section{Introduction}

\subsection{Motivation: the non-cutoff Boltzmann equation}
The Boltzmann equation is a fundamental kinetic model in statistical physics. The unknown funtion $f(t,x,v)\geq 0$ represents the particle density of a diffuse gas as it evolves in phase space. The equation reads
\begin{equation}\label{e:boltzmann}
\partial_t f + v\cdot \nabla_x f = Q_B(f,f),
\end{equation}
where $Q_B(f,f)$ is Boltzmann's collision operator, a nonlinear, nonlocal operator that acts in the velocity variable. For any two functions $f,g$, the operator is defined by
\[
Q_B(f,g) = \int_{\R^d} \int_{\mathbb S^{d-1}} B(v,v_*,\sigma) [f(v_*') g(v') - f(v_*) g(v)] \dd \sigma \dd v_*.
\]
Because collisions are assumed to be elastic, momentum and energy are conserved, which implies the four pre- and post-collisional velocities $v$, $v_*$, $v'$, $v_*'$ lie on a sphere of radius $|v-v_*|/2$. (See Figure \ref{f:boltz}.) Parameterizing this sphere by $\sigma \in \mathbb S^{d-1}$, one has the formulas
\begin{equation}\label{e:pre-post}
\begin{split}
v' &= \frac{v+v_*}2 + \frac{|v-v_*|} 2 \sigma,\\
v_*' &= \frac{v+v_*} 2 - \frac{|v-v_*|} 2 \sigma.
\end{split}
\end{equation}
The {\it non-cutoff} collision kernel $B(v,v_*,\sigma)$ is not integrable over $\mathbb S^{d-1}$. For constants $\gamma> -d$ and $s\in (0,1)$, it takes the form
\begin{equation}\label{e:kernel}
B(v,v_*,\sigma) = |v-v_*|^\gamma b(\cos\theta), \quad b(\cos\theta) \approx \theta^{-(d-1)-2s} \text{ as } \theta \to 0,
\end{equation}
where $\theta$ is the deviation angle between pre- and post-collisional velocities:
\[
\cos\theta = \frac{v-v_*}{|v-v_*|} \cdot \sigma.
\]
This paper is (mostly) focused on the space homogeneous regime, where the system is constant in $x$. In this case, the Boltzmann equation reduces to 
\[
\partial_t f = Q_B(f,f). 
\]
Global existence of classical smooth solutions is known for this homogeneous equation only when $\gamma + 2s \geq 0$: see e.g. \cite{desvillettes2004homogeneous, desvillettes2009stability, he2012homogeneous-boltzmann}. Regarding other notions of solution, measure-valued solutions are known to exist globally when $\gamma \geq -2$ \cite{lu2012measure, morimoto2016measure}, and H-solutions exist for any $\gamma>-d$ and $s\in (0,1)$ \cite{villani1998weak}. Thanks to the recent result of \cite{chaker2022entropy}, H-solutions are in fact weak solutions in the usual sense (integration against smooth test functions) when $\gamma + 2s > -2$. For more on the existence and regularity theory of the spatially homogeneous Boltzmann equation, see \cite{DM2005Lp, fournier2008uniqueness, morimoto2009homogeneous, chen2011smoothing, glangetas2016sharp, barbaroux2017gevrey, alonso2019homogeneous} and the references therein.

We are interested in the problem of constructing global smooth solutions, so we restrict our attention to the range
\[
\gamma+2s< 0,
\]
sometimes referred to as ``very soft potentials.''\footnote{If the interaction potential between particles is taken to be an inverse power law $\phi(r) = r^{1-p}$, then the assumption $\gamma+2s<0$ corresponds to $p<3$.} The regularity of the generalized solutions discussed above is not understood in this parameter regime.

The main extra difficulty in the case of very soft potentials comes from the singularity $|v-v_*|^\gamma$ in \eqref{e:kernel}, which becomes more severe when $\gamma$ is more negative. To control this singularity, one would need higher integrability estimates for the solution $f$, but it is not clear how to obtain this higher integrability unconditionally, when $\gamma+2s<0$.  For the model equation we introduce below, the desired higher integrability estimates are available. 

%
%
%
%
%

For the full inhomogeneous equation \eqref{e:boltzmann}, global existence remains unknown in the case of general (far from equilibrium) initial data, regardless of $\gamma$ and $s$. 
A more attainable, but still difficult, goal is to prove regularity conditional to uniform control on the mass, energy, and entropy densities of $f$, i.e. show that the solution is smooth and can be continued past any given time, as long as these densities satisfy uniform bounds. This has been accomplished for the case $\gamma+2s \in [0,2]$ in the breakthrough result \cite{imbert2020smooth}, but is still open in the case $\gamma+2s<0$. This open problem is discussed in the review article \cite{imbert2020review}, Section 12.2. The extra obstacle is, once again, the more severe singularity in the collision operator as $v_*\sim v$. We believe that the model equation introduced in the present article may provide a useful stepping stone for this problem as well.


\begin{figure}
\begin{tikzpicture}[>=triangle 45]
\draw[thick] (-2.5,0) node[anchor=east] {\Large $v'$}-- (2.5,0) node[anchor=west] {\Large $v_*'$} ;
\draw[thick] (-1.915,-1.607) node[anchor=north east] {\Large $v$} -- (1.915,1.607) node[anchor=south west] {\Large $v_*$};
\draw[thick] (0,0) circle (2.5cm);
\draw[blue, thick] (-1.2,0)    arc  (180:220:1.2) ; 
\draw[red, thick,->] (0,0)  --   (-0.9,0);
\node at (-0.4, 0.25) {\color{red}\Large $\sigma$};
\node at (-1.4,-0.6) {\color{blue}\Large $\theta$};
\end{tikzpicture}
\caption{The collision geometry corresponding to the Boltzmann collision operator.}\label{f:boltz}
\end{figure}
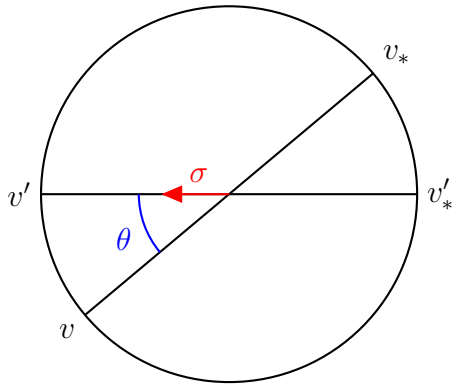

\subsection{An isotropic Boltzmann model}

We now introduce a model equation that---although not necessarily relevant from the point of view of physics---is more tractable than the Boltzmann equation while still encapsulating many of the key mathematical difficulties. 


To derive our equation, we begin by relaxing the requirement that the four pre- and post-collisional velocities lie on a sphere, i.e. we replace $\sigma$ in \eqref{e:pre-post} with any $z\in \R^d$, yielding
\begin{equation}\label{e:vprime-iso}
\begin{split}
v' &= \frac{v+v_*}2 +\frac{|v-v_*|} 2 z,\\
v_*' &= \frac{v+v_*} 2 -\frac{|v-v_*|} 2 z.
\end{split}
\end{equation}
(See Figure \ref{f:iso}.) This amounts to discarding energy conservation while still requiring momentum conservation (since $v+v_* = v'+v_*'$ is still true in \eqref{e:vprime-iso}). 

Next, we look at the Boltzmann collision kernel \eqref{e:kernel}, which in light of Figure \ref{f:boltz} and $\theta \approx \sin\theta$ can be rewritten as
\begin{equation}\label{e:B-change}
B(v,v_*,\sigma) 
\approx |v-v_*|^\gamma\left( \frac{|v'-v|}{|v-v_*|}\right)^{-(d-1)-2s}.
\end{equation}
Since we will integrate over $z\in\R^d$ rather than $\sigma \in \mathbb S^{d-1}$, dimensional analysis tells us that we should replace the exponent $-(d-1)-2s$ with $-d-2s$. Next, we make the change of variables
\[
z \mapsto w = v'-v = \frac{v_* - v} 2 + \frac {|v-v_*|} 2 z,
\]
which gives
\begin{equation}\label{e:vpw}
\begin{split}
v' &= v+w,\\
v_*' &= v_*-w,
\end{split}
\end{equation}
and noting that $\dd z = 2^{d} |v-v_*|^{-d} \dd w$, we absorb this extra factor of $\sim |v-v_*|^{-d}$ into the collision kernel, giving $B\approx |v-v_*|^{\gamma+2s}|v'-v|^{-d-2s}$. Finally, we need to change the factor $|v-v_*|$ in order to recover micro-reversibility, i.e. the property that a collision and its reverse are weighed the same amount by the collision kernel. It would be difficult to obtain a well-behaved operator without micro-reversibility. Since $|v'-v_*'| \neq |v-v_*|$ in our collision geometry, we replace $|v-v_*|$ in the collision kernel with $|v-v_*'|$, or equivalently, $|v' - v_*|$. This is the length of the dashed lines in Figure \ref{f:iso}.
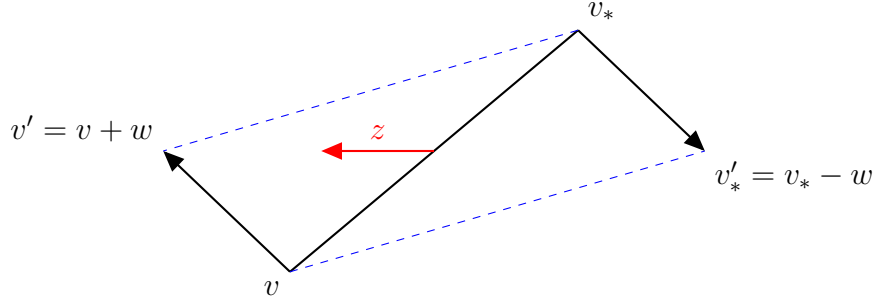
\begin{figure}
\begin{tikzpicture}[>=triangle 45]
\draw[thick] (-1.915,-1.607) node[anchor=north east] {\Large $v$} -- (1.915,1.607) node[anchor=south west] {\Large $v_*$};
\draw[thick,->] (-1.915,-1.607) -- (-3.6,0) node[anchor=south east] {\Large $v' = v+w$};
\draw[thick,->] (1.915,1.607) -- (3.6,0) node[anchor=north west] {\Large $v_*'=v_*-w$};
\draw[red, thick,->] (0,0)  --   (-1.5,0);
\node at (-0.75, 0.25) {\color{red}\Large $z$};
\draw[blue, dashed] (-1.915, -1.607) -- (3.6,0);
\draw[blue, dashed] (1.915, 1.607) -- (-3.6,0);
\end{tikzpicture}
\caption{The collision geometry corresponding to the isotropic collision operator. Momentum is conserved by the collisions, but energy is not.}\label{f:iso}
\end{figure}

For simplicity, we choose our kernel to be a product of power laws:
\[
B(v,v_*,w) = c_{d,\gamma,s} |v-v_*'|^{\gamma+2s+d} |v'-v|^{-d-2s},
\]
and our collision operator finally becomes
\begin{equation}\label{e:Qiso}
Q(f,g)  =c_{d,\gamma,s} \int_{\R^d} \int_{\R^d} |v-v_*+w|^{\gamma+2s} |w|^{-d-2s} [g(v+w)f(v_*-w) - g(v)f(v_*)] \dd w \dd v_* ,
\end{equation}
with $v'$, $v_*'$ defined in \eqref{e:vpw}, and
\begin{equation}\label{e:cdgs}
c_{d,\gamma,s} =  
 \frac{(1-s)\Gamma\left(\frac {d+2s} 2\right)\Gamma\left(-\frac {\gamma+2s} 2\right)}{\pi^d 2^{d+\gamma} \Gamma\left(\frac{d+\gamma+2s}{2}\right)}.
\end{equation}
The value of the constant $c_{d,\gamma,s}$ does not play an important role, as long as one studies the equation for fixed $\gamma$ and $s$. The choice we make for $c_{d,\gamma,s}$ ensures that $Q(f,g)$ has a well-defined limit as $s\to 1$ (see \eqref{e:sto1} below).


The following alternate form of the operator $Q$ is justified in Section \ref{s:carleman}: 
\begin{equation}\label{e:carleman}
\begin{split}
Q(f,g) &= c_1 [f\ast|\cdot|^{\gamma+2s}](-\Delta)^s  g +  c_2 [f\ast |\cdot|^\gamma] g,
\end{split}
\end{equation}
where throughout this paper, we use the notation $[g \ast |\cdot|^\mu]$ or $[g\ast |\cdot|^\mu](v)$ for convolutions over $\R^d_v$ with power functions $|v|^\mu$, 
and the constants $c_1, c_2$ are defined by 
\begin{equation}\label{e:c1c2}
\begin{split}
c_1 &= 
\frac{(1-s)|\Gamma(-s)| \Gamma\left(- \frac{\gamma+2s} 2\right)}{\pi^{d/2} 2^{d+\gamma+2s} \Gamma\left(\frac{d+\gamma+2s} 2\right)},\\
c_2 &= 
\frac{(1-s)|\Gamma(-s)| \Gamma\left(-\frac\gamma 2\right)}{ \pi^{d/2} 2^{d+\gamma} \Gamma\left(\frac{d+\gamma} 2\right)}.
\end{split}
\end{equation}
Both representations \eqref{e:Qiso} and \eqref{e:carleman} of $Q(f,g)$ are useful in our analysis.  



The goal of this paper is to study the equation
\begin{equation}\label{e:main}
\partial_t f + v\cdot\nabla_x f = Q(f,f),
\end{equation}
with $Q(f,f)$ defined as in \eqref{e:Qiso} or equivalently, \eqref{e:carleman}. We call \eqref{e:main} the {\it isotropic Boltzmann equation}.

To justify the use of the term ``isotropic,'' recall that the true Boltzmann collision operator can be written in the Carleman representation (see, e.g. \cite{silvestre2016boltzmann}) as follows:
\begin{align}
Q_B(f,g) &= \int_{\R^d} K_B(v,w) \frac{ g(v+w) - g(v)}{|w|^{d+2s}} \dd w + c [f\ast |\cdot|^\gamma] g,\label{e:boltz-carl}\\
K_B(v,w) &\approx \left(\int_{\{z \perp w\}} f(v+z) |z|^{\gamma+2s+1} \dd z\right) |w|^{-d-2s}.
\end{align}
This integral kernel $K_B(v,w)$ consists of the standard singularity of order $2s$, weighted by a function depending nonlocally on $f$. The kernel is anisotropic because the weight depends on the direction of $w$ relative to $v$, and this leads to nontrivial complications in the analysis of $Q_B$. On the other hand, the isotropic collision operator \eqref{e:carleman} can be written in the same form \eqref{e:boltz-carl}, with $K_B(v,w)$ replaced by
\[
 K_f(v,w) : = c_{d,\gamma,s}\left(\int_{\R^d} f(v+z)|z|^{\gamma+2s} \dd z\right) |w|^{-d-2s}.
\]
Note that this  weight function $c_{d,\gamma,s}[f\ast|\cdot|^{\gamma+2s}](v)$ is independent of $w$.

\subsection{Main result}

Recall the notation $\vv = \sqrt{1+|v|^2}$. Our main result is the existence of global solutions in the homogeneous case, for part of the very soft potentials regime:

\begin{theorem}\label{t:main}
Let $d\geq 3$ be an integer, and let $\gamma\in (-d,0)$ and $s\in (0,1)$ satisfy
\[
\max\left\{ -\frac{d+4s} 3 ,  -2s - \frac{4s}d\right\} \leq \gamma < -2.
\]
Then, for any $f_{\rm in}: \R^d \to [0,\infty)$ with $\vv^q f_{\rm in} \in C^\infty(\R^d)$ for all $q \geq 0$, there is a unique smooth solution $f\geq 0$ of the homogeneous isotropic Boltzmann equation
\begin{equation}\label{e:homogeneous}
\partial_t f = Q(f,f),
\end{equation}
on $[0,\infty)\times \R^d$ with $f(0,v) = f_{\rm in}(v)$, where $Q(f,f)$ is defined in \eqref{e:Qiso}. This solution has nonincreasing $L^2(\R^d)$ norm:
\[
\int_{\R^d} f^2(t,v) \dd v \leq \int_{\R^d} f_{\rm in}^2(v) \dd v, \quad t\geq 0.
\]
\end{theorem}

We emphasize that global existence of classical solutions is not known for the homogeneous Boltzmann equation in the range of $\gamma$ and $s$ covered by Theorem \ref{t:main}. For the most important case $d=3$, our condition on $\gamma$ and $s$ simplifies to
\[
-1 - \frac 4 3 s \leq \gamma < -2.
\]
Some of our intermediate results apply also to the case $d=2$, which is ruled out in our main theorem by the condition $\gamma < -2$.

The requirement that $f_{\rm in}$ is smooth with rapid decay is certainly not sharp.  We take very smooth initial data to avoid some technical difficulties, and to obtain a clean statement of our result.

The key step in the proof of Theorem \ref{t:main} is an estimate in $L^2(\R^d_v)$ that holds for all $t\geq 0$ (Theorem \ref{t:L2} below). This estimate and its proof are inspired by a recent result for the isotropic Landau equation by Gualdani-Guillen \cite{gualdani2022isotropic}. The proof in \cite{gualdani2022isotropic} is based on a clever application of a weighted Hardy inequality, and our proof of Theorem \ref{t:L2} applies an analogous strategy, using a weighted fractional Hardy inequality that was recently proven in \cite{dyda2022hardy}.

After proving the $L^2$ estimate, we show that the energy $\int_{\R^d} |v|^2 f(t,v) \dd v$ (which is not conserved by the flow) remains bounded on finite time intervals. Next, we adapt techniques from Boltzmann theory to prove global boundedness and polynomial decay. Finally, by bootstrapping $L^2$-based energy estimates, we show that these global estimates imply continuation of solutions for large times.

%

\subsection{Local existence}

A local existence theorem is naturally needed as part of the proof of our main result. For the sake of potential future applications, we prove local existence under much weaker hypotheses than our global existence theorem. In particular, we construct a solution for both the homogeneous and inhomogeneous equations, and for all $d\geq 2$ and $(\gamma,s)\in (-d,0)\times (0,1)$ such that $\gamma+2s< 0$, i.e. all relevant values of $d$, $\gamma$, and $s$.

Let $\T^d$ denote the $d$-dimensional torus of side length one, and define the weighted $L^p$ and Sobolev spaces
\begin{equation*}
\begin{split}
L^p_q(\T^d\times\R^d) &= \{ f:\T^d_x\times\R^d_v \to \R : \|\vv^q f(x,v)\|_{L^p(\T^d\times\R^d)}\},\\
H^k_q(\T^d\times\R^d) &= \{ f:\T^d_x\times\R^d_v\to \R: \partial^\beta f \in L^p_q(\T^d\times\R^d) \text{ for all $\beta \in \mathbb N^{2d}$ with $|\beta|\leq k$}\}.
\end{split}
\end{equation*}
For functions of $v$ only, let $L^p_q(\R^d)$ and $H^k_q(\R^d)$ be defined in the analogous way, with norms over $\R^d_v$ and with $\beta \in \mathbb N^d$ instead of $\mathbb N^{2d}$. 
For the inhomogeneous equation, we have
\begin{theorem}\label{t:short-time-existence-inhom}
Let $d\geq 2$ and $(\gamma,s)\in (-d,0)\times(0,1)$ be such that $\gamma+2s<0$. For any $q> \max\{\gamma+2s+d,1\}$, and any nonnegative initial condition $f_{\rm in}\in H^{2d+2}_q(\T^d\times\R^d)$, there exists a $T>0$ depending only on $\|f_{\rm in}\|_{H^{2d+2}_q(\T^d\times\R^d)}$, 
and a unique nonnegative solution 
\[
f\in L^\infty([0,T],H^k_q(\T^d\times\R^d)) \cap W^{1,\infty}([0,T], H^{d+1}_{q-1}(\T^d\times\R^d))
\] 
to 
\[
\partial_t f + v\cdot \nabla_x f = Q(f,f).
\]
In addition, if $f_{\rm in} \in H^k_{q'}(\T^d\times\R^d)$ for any $k\geq 2d+2$ and $q' \geq q$, then $f(t)\in H^k_{q'}(\T^d\times\R^d)$ for all $t\in [0,T]$, and 
\begin{equation*}
\|f(t)\|_{H^k_{q'}(\T^d\times\R^d)} \leq \|f_{\rm in}\|_{H^k_{q'}(\T^d\times\R^d)} \exp(C t (1+\|f\|_{L^\infty([0,t],H^{2d+2}_q(\T^d\times\R^d))})) , \quad t\in [0,T].
\end{equation*}
In particular, if $\vv^{q'} f_{\rm in} \in C^\infty(\T^d\times\R^d)$ for all $q'>0$, then $\vv^{q'} f\in C^\infty([0,T]\times\T^d\times\R^d)$ for all $q'>0$, with the same $T$ as above. 
\end{theorem}
For the homogeneous equation, we have essentially the same result, with weaker regularity assumptions on $f_{\rm in}$ needed because the domain is of smaller dimension.  Letting $\lceil x \rceil$ denote the smallest integer $\geq x$, we have:
\begin{theorem}\label{t:short-time-existence}
Let $d\geq 2$ and $(\gamma,s)\in (-d,0)\times(0,1)$ be such that $\gamma+2s<0$. For any $q> \gamma+2s+d$, and any nonnegative initial condition $f_{\rm in}\in H^{d+3}_q(\R^d)$, there exists a $T>0$ depending only on $\|f_{\rm in}\|_{H^{d+3}_q(\R^d)}$, 
and a unique nonnegative solution $f\in L^\infty([0,T],H^{d+3}_q(\R^d)) \cap W^{1,\infty}([0,T], H^{\lceil (d+1)/2\rceil}_{q}(\R^d))$ to 
\[
\partial_t f = Q(f,f).
\]
In addition, if $f_{\rm in} \in H^k_{q'}(\R^d)$ for any $k\geq d+3$ and $q' \geq q$, then $f(t)\in H^k_{q'}(\R^d)$ for all $t\in [0,T]$, and 
\begin{equation}\label{e:energy-est-local-thm}
\|f(t)\|_{H^k_{q'}(\R^d)} \leq \|f_{\rm in}\|_{H^k_{q'}(\R^d)} \exp(C t (1+\|f\|_{L^\infty([0,t],H^{d+3}_q(\R^d))})) , \quad t\in [0,T].
\end{equation}
In particular, if $\vv^{q'} f_{\rm in} \in C^\infty(\R^d)$ for all $q'>0$, then $\vv^{q'} f\in C^\infty([0,T]\times\R^d)$ for all $q'>0$, with the same $T$ as above. 
\end{theorem}

In these results, we have made no attempt to optimize the regularity requirements on $f_{\rm in}$, since our main focus in this paper is the continuation of smooth solutions for large times. 
%
%
%

For the inhomogeneous Boltzmann equation \eqref{e:boltzmann}, short-time existence has been proven under much weaker regularity assumptions ($f_{\rm in} \in L^\infty_q$ for some $q>3+2s$, and $f_{\rm in}$ is uniformly positive in some small ball in $\T^3_x\times\R^3_v$), see \cite{HST2022irregular}. A similar existence result could be derived for our isotropic model, thereby extending the existence part of Theorem \ref{t:short-time-existence-inhom} to irregular initial data, by adapting the proofs in \cite{HST2022irregular} (and the companion paper \cite{HST2020lowerbounds} about pointwise lower bounds, which are used in \cite{HST2022irregular}), replacing the Boltzmann operator $Q_B(f,f)$ with our isotropic $Q(f,f)$. As in \cite{HST2022irregular}, a uniqueness result in relatively low-regularity spaces would require H\"older continuity and uniform-in-$x$ positivity for $f_{\rm in}$, at least with current techniques.  We do not explore the details here. See also  \cite{desvillettes2004homogeneous, desvillettes2009stability} for existence results on homogeneous Boltzmann with irregular initial data. 

These local results should not be considered as novel as our Theorem \ref{t:main}, since they are similar in spirit to known results for the true Boltzmann equation, e.g. \cite{morimoto2015polynomial, HST2020boltzmann, henderson2022existence}. (However, these cited results specialize to $d=3$, while we give a proof for any $d\geq 2$.) We have highlighted these local existence results in the introduction because they may be useful for future research.

\subsection{Comparison with Landau and isotropic Landau equations}\label{s:landau}

The Landau equation is the plasma physics counterpart of the Boltzmann equation. It reads 
\begin{equation}\label{e:landau}
\partial_t f + v\cdot \nabla_x f = Q_L(f,f),
\end{equation}
where for some $\gamma \geq -d$, the Landau collision operator can be written 
\[
Q_L(f,g) = a_{d,\gamma}\tr\left( [\Pi(\cdot)|\cdot|^{\gamma+2} \ast f] D_v^2 g\right) + c_{d,\gamma}  [f\ast |\cdot|^\gamma] g,
\]
%
where $\Pi(z)$ is the projection matrix onto $z^\perp$, and
\[
\begin{split}
a_{d,\gamma} &= \frac{\pi^{-d/2} \Gamma\left(-\frac\gamma 2\right)}{(-\gamma-2)2^{d+\gamma+1}  \Gamma\left(\frac{d+\gamma+2} 2\right)},\\
c_{d,\gamma} &= (-\gamma-2) (d+\gamma) a_{d,\gamma}.
\end{split}
\]
When $\gamma = -d$, the second term in $Q_L(f,f)$ must be replaced by $f^2$. 




In 2012, Krieger and Strain \cite{krieger2012isotropic} introduced a spatially homogeneous model equation that has since come to be known as {\it isotropic Landau}. It is obtained by removing the projection matrix from $Q_L$:
\begin{equation}\label{e:isotropic-landau}
\begin{split}
\partial_t f &= Q_{IL}(f,f),\\
Q_{IL}(f,g) &= a_{d,\gamma}[f\ast |\cdot|^{\gamma+2}] \Delta g + c_{d,\gamma} [f\ast |\cdot|^{\gamma}]g,
\end{split}
\end{equation}
As above, in the case $\gamma = -d$, one replaces $c_{d,\gamma}[f \ast |\cdot|^\gamma] f$ with $f^2$. Note that \cite{krieger2012isotropic} only considered the case $d=3, \gamma = -3$. 

Some regularity and existence results for \eqref{e:isotropic-landau} were later obtained in \cite{gressman2012isotropic, gualdani2016radial, gualdani2018isotropic, gualdani2022isotropic}. See also the review \cite{gualdani2018review}. For Coulomb potentials ($d=3, \gamma= -3$, the most physically relevant case), global existence has been proven only in the case of radially decreasing initial data \cite{gualdani2016radial}. Global existence was shown in \cite{gualdani2022isotropic} for part of the parameter range $\gamma \in [-d, -2)$, which is striking because  at that time, global existence was not known for homogeneous Landau in this parameter range. Very recently, global existence for homogeneous Landau was established in the breakthrough result \cite{guillen2023landau} for any $\gamma \in [-3,1]$. 

The Landau collision term $Q_L(f,f)$ can be seen as the limit of the Boltzmann operator $Q_B(f,f)$ when {\it grazing collisions} (collisions with $\theta \approx 0$ in \eqref{e:kernel}) predominate \cite{desvillettes1992grazing, alexandre2004landau}.  For our isotropic collision operator \eqref{e:Qiso}, a true grazing collisions limit does not make sense because the angle between velocities plays no role. The correct analogy to the grazing collisions limit is a rescaling that focuses on the singularity at $w=0$ (equivalently, $z=0$), which becomes a delta function in the limit. With our choice of normalization constant $c_{d,\gamma,s}$, this is exactly the limit $s\to 1$ in \eqref{e:Qiso}. Then one has $(-\Delta^s) f \to -\Delta f$ and $[f\ast |\cdot|^{\gamma+2s}]\to [f\ast|\cdot|^{\gamma+2}]$, and taking the limit in $c_1$ and $c_2$ in \eqref{e:c1c2}, we obtain the following important fact:
\begin{equation}\label{e:sto1}
Q(f,g) \to Q_{IL}(f,g), \quad \text{ as } s\to 1,
\end{equation}
for any sufficiently smooth functions $f$ and $g$.

In light of this convergence, one may compare our results to the main result of \cite{gualdani2022isotropic}. Sending $s\to 1$ in the condition $\gamma \geq -\frac {d+4s} 3$ in our Theorem \ref{t:main}, this would converge to $\gamma \geq -\frac{d+4}3$, which is slightly worse than the condition obtained  in \cite{gualdani2022isotropic} for the isotropic Landau equation. 
This gap exists because we need to use the $L^2$ norm in our proof (see Theorem \ref{t:L2}), whereas in \cite{gualdani2022isotropic} they manage to bound any $L^p$ norm with $1\leq p\leq \frac{d+\gamma} {-2-\gamma}$. By tracing the proof in \cite{gualdani2022isotropic} and insisting on using only the $L^2$ norm, one would get the condition $2 \leq \frac{d+\gamma} {-2-\gamma}$, or $\gamma\geq -\frac{d+4}3$, as expected.

It is interesting to note that if the Boltzmann collision operator is written in the Carleman form \eqref{e:boltz-carl}, all four equations (Boltzmann, Landau, isotropic Boltzmann, and isotropic Landau) have the same reaction term $f[f\ast |\cdot|^\gamma]$ up to a constant, at least when $\gamma > -d$.

\subsection{Open problems}

\subsubsection{Expanding the allowable range of $\gamma$ and $s$ in our main result.} 

The restrictions on $\gamma$ and $s$ in Theorem \ref{t:main} come from the following obstructions:

\begin{itemize}

\item The condition $\gamma \geq -\frac{d+4s}3$ arises in our proof that the $L^2$ norm is nonincreasing (Theorem \ref{t:L2}). Extending our $L^2$ bound to higher $L^p$ norms, as is done for isotropic Landau in \cite{gualdani2022isotropic}, would improve the allowable range of $\gamma$ and $s$. To do this, one would likely need to understand a ``divergence form'' type of structure of $Q(f,f)$, by analogy with the form $Q_{IL}(f,f) = c\nabla_v \cdot ([f\ast |\cdot|^{\gamma+2}] \nabla f - f [\nabla f\ast |\cdot|^{\gamma+2}])$  used in \cite{gualdani2022isotropic}.

\item The condition $\gamma <-2$ arises from our estimate on the energy $\int_{\R^d} |v|^2 f(t,v) \dd v$, Lemma \ref{l:energy}. It is not clear whether this estimate would be true for $\gamma > -2$. 
Without a bound on the energy, it is difficult to control the collision kernel $K_f(v,w) \approx [f\ast|\cdot|^{\gamma+2s}](v) |w|^{-d-2s}$ either from above or below. 
 Note that this issue does not arise for the homogeneous Boltzmann equation, which conserves energy.

\item The condition $\gamma \geq -2s - 4s/d$ is needed for our proof that $f$ is globally bounded, Proposition \ref{p:Linfty}. This proof uses a barrier argument. A different proof of this $L^\infty$ estimate (for example, using De Giorgi or Moser iteration) could possibly allow one to weaken or remove this condition.

\end{itemize}

We should note that the constant in the fractional Hardy inequality degenerates as $\gamma \to -d$ (i.e. the constant $C_H$ in Theorem \ref{t:hardy} approaches 0), which suggests our approach cannot be extended to very negative values of $\gamma$.

\subsubsection{The radially symmetric case} 

The first global existence result for the isotropic Landau equation \eqref{e:isotropic-landau} was obtained in  \cite{gualdani2016radial} under the assumption of initial data that is radially symmetric and monotone decreasing in $|v|$. See also the earlier works \cite{krieger2012isotropic, gressman2012isotropic} which proved a similar result for modified versions of \eqref{e:isotropic-landau}. 
These results all apply to the Coulomb case $d=3, \gamma = -3$.

Based on these results, it is natural to conjecture that global existence for the (homogeneous) isotropic Boltzmann equation for very negative values of $\gamma$ may be more tractable in the case of radially decreasing initial data. 

\subsubsection{The spatially inhomogeneous case}

As mentioned above, the inhomogeneous case of the isotropic Boltzmann equation may provide a useful model for the problem of conditional regularity in the very soft potentials regime.  We have constructed a local classical solution for the inhomogeneous equation in Theorem \ref{t:short-time-existence-inhom} with the goal of inspiring future work on this question.

\subsubsection{Results for the true Boltzmann equation} 

As with any model problem, the eventual goal is to gain insights that can be applied to the actual Boltzmann equation. One interesting, but speculative, direction is to seek an ``anisotropic fractional Hardy inequality'' adapted to the specific structure of the Boltzmann collision operator, which could play a role similar to the weighted fractional Hardy inequality of \cite{dyda2022hardy} in the present work.

\subsection{Notation} We sometimes use the abbreviations $\varphi' = \varphi(v')$, $\varphi_*' = \varphi(v_*')$, $\varphi_* = \varphi(v_*)$, and $\varphi =\varphi(v)$, where $v'$ and $v_*'$ are defined by \eqref{e:vpw}.

For $q\in \R$, we let $L^p_q(\R^d)$ denote the polynomially-weighted $L^p$ spaces with norm
\[
\|h\|_{L^p_q(\R^d)} = \|\vv^q h\|_{L^p(\R^d)},
\]
and simlarly, for integer $k$ we define the weighted Sobolev norms
\[
\|h\|_{H^k_q(\R^d)} = \|\vv^q h\|_{H^k(\R^d)}.
\]
For $s\in (0,1)$, we use the standard $H^s(\R^d)$ seminorm
\[
[h]_{H^s(\R^d)}^2 = \int_{\R^d} \int_{\R^d} \frac{|h(v+w) - h(v)|^2}{|w|^{d+2s}} \dd w \dd v,
\]
and define
\[
[h]_{H^s_q(\R^d)} = [\vv^q h]_{H^s(\R^d)}
\]
as well as
\[
\|h\|_{H^{k+s}_q(\R^d)}^2 = \|h\|_{H^k_q(\R^d)}^2 + \sum_{|\beta|=k} [\partial^\beta h]_{H^s_q(\R^d)}^2,
\]
where $|\beta|$ denotes the order of the multi-index $\beta = (\beta_1,\ldots,\beta_d)$. 

We often use the notation $A\lesssim B$ to indicate $A\leq CB$ for a constant depending on $d$, $\gamma$, $s$, and sometimes additional quantities in the statement of a given theorem or lemma. We also use $A\approx B$ to mean $A\lesssim B$ and $B\lesssim A$.

\subsection{Organization of the paper}

In Section \ref{s:prelim}, we derive some general properties of our equation, and discuss the weighted fractional Hardy inequality that we need in our proof. Section \ref{s:L2} proves  that the $L^2(\R^d)$ norm of our solution is nonincreasing, and establishes a bound on the energy (second moment) of the solution. In Section \ref{s:Linfty}, we prove an $L^\infty$ estimate and propagation of polynomial decay estimates. Section \ref{s:global} contains the proof of global existence (Theorem \ref{t:main}), and Section \ref{s:short} contains the proof of short-time existence (Theorems \ref{t:short-time-existence-inhom} and \ref{t:short-time-existence}), which is given outside its proper logical order because of its length. Appendix \ref{s:a} contains some technical lemmas.

\section{Fundamental properties and tools}\label{s:prelim}

\subsection{Weak formulation and conservation laws}\label{s:conservation}

It is well known that the Boltzmann collision operator $Q_B(f,f)$ has a useful weak formulation (originally written down by Maxwell \cite{maxwell1867}) that allows one to make sense of $\int_{\R^d} \varphi Q(f,f) \dd v$ without using any regularity of $f$. Our collision operator satisfies an analogous property. To see this, we use the pre-post-collisional change of variables $(v,v_*,w) \mapsto (v',v_*',-w)$, with unit Jacobian, to write
%
\begin{equation}\label{e:weak1}
\int_{\R^d} \varphi Q(f,f) \dd v 
= \int_{\R^d}\int_{\R^d} \int_{\R^d} B(v,v_*,w) ff_*[\varphi' - \varphi] \dd w \dd v_* \dd v.
\end{equation}
Symmetrizing this expression again with the change of variables $(v,v') \leftrightarrow (v_*,v_*')$, which also has unit Jacobian, gives
\begin{equation}\label{e:weak2}
\int_{\R^d} \varphi Q(f,f) \dd v= \frac 1 2 \int_{\R^d}\int_{\R^d} \int_{\R^d} B(v,v_*,w) ff_*[\varphi' +\varphi_*' - \varphi - \varphi_*] \dd w \dd v_* \dd v.
\end{equation}

 From \eqref{e:weak2} with the choices $\varphi(v) \equiv 1$ and $\varphi(v) = v$, we see that $\int_{\R^d} Q(f,f) \dd v  = 0$ and $\int_{\R^d} v Q(f,f) \dd v = 0$. As a result, integrating the equation \eqref{e:homogeneous} against $1$ and $v$ shows that the evolution formally conserves mass and momentum.   Note, however, that $|v'|^2 + |v_*'|^2 - |v|^2 - |v_*|^2 \neq 0$, so the energy is not conserved.

We can show that the evolution of \eqref{e:homogeneous} dissipates entropy by the same formal argument as in the Boltzmann equation: symmetrizing \eqref{e:weak2} again with the pre-post collsional change of variables, one has
\[
\int_{\R^d} \varphi Q(f,f) \dd v = -\frac 1 4 \int_{\R^d}\int_{\R^d} \int_{\R^d} B(v,v_*,w) (f'f_*' - f f_*)[\varphi' +\varphi_*' - \varphi - \varphi_*] \dd w \dd v_* \dd v.
\]
The choice $\varphi(v) = \log f(v)$ then gives
 \[
 \begin{split}
 \int_{\R^d} \log f  Q(f,f) \dd v &= -\frac 1 4 \int_{\R^d} \int_{\R^d} \int_{\R^d} B(v,v_*,w) (f'f_*' - f f_*) \log\left( \frac {f' f_*'}{f f_*}\right) \dd w \dd v_* \dd v \leq 0,
 \end{split}
 \]
 since $(x-y) (\log x - \log y) \geq 0$ for any $x,y>0$. As a result, the entropy $\int_{\R^d} f \log f \dd v$ is non-increasing in $t$ for solutions of \eqref{e:homogeneous}.

\subsection{Integro-differential form}\label{s:carleman}

In this section, we justify the formula \eqref{e:carleman} for $Q(f,f)$, that is analogous to the Carleman representation for the Boltzmann collision operator. Starting with \eqref{e:Qiso}, we add and subtract $f(v_*- w)g(v)$ inside the integral to write $Q(f,g) = Q_1(f,g)+ Q_2(f,g)$, with
\[
\begin{split}
 Q_1(f,g) &= c_{d,\gamma,s} \int_{\R^d} \int_{\R^d}  |v-v_*+w|^{\gamma+2s} |w|^{-d-2s} f(v_*-w)[g(v + w) - g(v)] \dd w \dd v_*,\\
Q_2(f,g) &=  c_{d,\gamma,s} g(v)\int_{\R^d} \int_{\R^d}  |v-v_*+w|^{\gamma+2s} |w|^{-d-2s}[ f(v_*-w)- f(v_*)] \dd w \dd v_*.
\end{split}
\]
In $Q_1$, we reverse the order of integration and observe that $\int_{\R^d} |v-v_*+w|^{\gamma+2s} f(v_*-w) \dd v_* = [f\ast |\cdot|^{\gamma+2s}](v)$, yielding
\begin{equation}\label{e:Q1}
\begin{split}
Q_1(f,g) &=  c_{d,\gamma,s}\int_{\R^d} [f\ast |\cdot|^{\gamma+2s}](v) \frac{g(v+w)- g(v)}{|w|^{d+2s}} \dd w\\ 
&=  c_{d,\gamma,s}\frac{\pi^{d/2} |\Gamma(-s)|} {4^s \Gamma(\frac{d+2s}{2})} [f\ast|\cdot|^{\gamma+2s}](v) (-\Delta)^s g.
\end{split}
\end{equation}
Alternatively, one could write
\[
\begin{split}
Q_1(f,g) &= \int_{\R^d} K_f(v,w) [g(v+w) - g(v)] \dd w,
\end{split}
\]
with 
\[ 
K_f(v,w) := 
 c_{d,\gamma,s} |w|^{-d-2s} \left[f\ast |\cdot|^{\gamma+2s}\right](v).
\]

For $Q_2$, we separate terms and make the change of variables $v_*\mapsto v_*-w$ in the first term. To avoid divergent integrals, we write the $w$ integral as a limit:
\begin{equation}\label{e:Q2der}
\begin{split}
Q_2(f,g) &= c_{d,\gamma,s}g(v) \lim_{\eps\to 0}\left(\int_{\R^d} \int_{\{|w|\geq \eps\}} |v-v_*+w|^{\gamma+2s} |w|^{-d-2s} f(v_*-w) \dd w \dd v_*\right.\\
&\left. \quad - \int_{\R^d} \int_{\{|w|\geq \eps\}}|v-v_*+w|^{\gamma+2s} |w|^{-d-2s} f(v_*) \dd w \dd v_*\right),\\
&= c_{d,\gamma,s}g(v)\lim_{\eps\to 0} \int_{\R^d} \int_{\{|w|\geq \eps\}} [|v-v_*|^{\gamma+2s} - |v-v_*+w|^{\gamma+2s}] |w|^{-d-2s}  f(v_*) \dd w\dd v_*.
\end{split}
\end{equation}
Next, we use the known formula for $(-\Delta)^s |v|^{\gamma+2s}$ \cite[Table 1]{kwasnicki2019fractional} to write
\[
\begin{split}
\lim_{\eps\to 0} \int_{\{|w|\geq \eps\}} \frac{|v|^{\gamma+2s} - |v+w|^{\gamma+2s}} {|w|^{d+2s}} \dd w &=  \frac{\pi^{d/2} |\Gamma(-s)|} {4^s \Gamma(\frac{d+2s}{2})} (-\Delta)^s |v|^{\gamma+2s}\\
&=  c_{R} |v|^\gamma.
\end{split}
\]
with $c_R$, the constant corresponding to the reaction term, given by
\begin{equation}\label{e:cR}
c_R = \frac{\pi^{d/2} |\Gamma(-s)| \Gamma(\frac{\gamma+2s+d} 2)\Gamma(\frac{-\gamma} 2)} {\Gamma(\frac{d+2s}2)\Gamma(\frac{\gamma+d} 2)\Gamma(-\frac {\gamma+2s} 2)}
\end{equation}
Evaluating this at $v-v_*$, we now have
\[ 
\begin{split}
Q_2(f,g) 
&=  c_{d,\gamma,s} c_R g(v)  \int_{\R^d} |v-v_*|^\gamma f(v_*) \dd v_*.
\end{split}
\]
The preceding (relatively simple) calculation to write $Q_2(f,g)$ in a convenient form, plays the same role as the Cancellation Lemma \cite{alexandre2000entropy} for the Boltzmann collision operator.

Combining the formulas for $Q_1$ and $Q_2$, we have, as claimed above in \eqref{e:carleman},
\[
Q(f,g)(v) = c_1[f\ast |\cdot|^{\gamma+2s}](v) (-\Delta)^s g +c_2 [f\ast |\cdot|^{\gamma}](v) g(v),
\]
with $c_1$ and $c_2$ as in \eqref{e:c1c2}. It is convenient to write $Q(f,g)$ without the normalization constant of $(-\Delta)^s$, i.e.
\[
Q(f,g) = c_{d,\gamma,s} \left( [f\ast |\cdot|^{\gamma+2s}](v) \int_{\R^d} \frac{g(v+w) - g(v)}{|w|^{d+2s}} \dd w  + c_R [f\ast |\cdot|^\gamma](v) g(v)\right),
\]
where $c_{d,\gamma,s}$ is given by \eqref{e:cdgs} and $c_R$ is given by \eqref{e:cR}. 
This form is useful because the value of $c_{d,\gamma,s}$ does not affect the structure of the equation, but the balance of constants between the terms $Q_1$ and $Q_2$ plays an important role. 

\subsection{Fractional Hardy inequality}

A functional inequality of Hardy type plays a crucial role in our proof of Theorem \ref{t:L2} below. Let us quote from the recent work of Dyda-Kijaczko \cite{dyda2022hardy}, specialized to the case $p=2$, $\alpha = 0$, $\beta = -(\gamma+2s)$:

\begin{theorem}\cite[Theorem 5]{dyda2022hardy}\label{t:hardy}
Let $0< s < 1$. For all $u\in C_c(\R^d)$, the following inequality holds:
\[
C_H \int_{\R^d} \frac{|u(v)|^2}{|v|^{-\gamma}} \dd v \leq \int_{\R^d} \int_{\R^d} \frac{|u(v+w) - u(v)|^2}{|w|^{d+2s}}  |v|^{\gamma+2s} \dd w \dd v,
\]
where 
\[
C_H  = 
\frac{\pi^{d/2}|\Gamma(-s)|}{\Gamma(\frac{d+2s} 2)} \left[ \frac{ 2 \Gamma(\frac {d-\gamma} 4) \Gamma(\frac {d+\gamma+4s} 4)} { \Gamma(\frac{d+\gamma} 4) \Gamma(\frac{d-\gamma-4s} 4)} - \frac{ \Gamma(\frac{-\gamma} 2) \Gamma(\frac{d+\gamma+2s} 2)} {\Gamma(\frac{d+\gamma} 2)\Gamma(\frac{-\gamma-2s} 2)} \right].
\]
Furthermore, the constant $C_H$ is optimal.
\end{theorem}

This result can easily be extended to non-compactly supported $u$ by density, as long as both sides of the inequality are finite. 

Fractional inequalities of this general type have a long history. See e.g. the result of Stein-Weiss \cite{stein1958fractional} from 1958, as well as some more recent works such as \cite{frank2008hardy, loss2010fractional, dyda2012fractional, abdellaoui2017fractional}. For our purposes in this article, we need an inequality featuring weights of a general enough form, and we need an optimal constant. It seems that \cite{dyda2022hardy} is the first result satisfying both these requirements.

To derive the form of the inequality that we use below, for each $v_*\in \R^d$, we apply Theorem \ref{t:hardy} to $u(v-v_*)$ and change variables to obtain
\[
C_H \int_{\R^d} |u(v)|^2|v-v_*|^{\gamma} \dd v \leq \int_{\R^d} \int_{\R^d} \frac{|u(v+w) - u(v)|^2}{|w|^{d+2s}}  |v-v_*|^{\gamma+2s} \dd w \dd v.
\]
Multiplying this inequality by $f(v_*)$ and integrating in $v_*$, we obtain
\begin{equation}\label{e:hardy} 
C_H\int_{\R^d} |u(v)|^2 [f\ast|\cdot|^\gamma](v)  \dd v \leq  \int_{\R^d}\int_{\R^d} \frac{|u(v+w) - u(v)|^2} {|w|^{d+2s}} [f\ast |\cdot|^{\gamma+2s}](v) \dd w \dd v.
\end{equation}

%

\section{Time-independent $L^2$ bound}\label{s:L2}

This section is devoted to the $L^2$ estimate for solutions of \eqref{e:homogeneous}, which is the key step in proving global existence. 

In the following proof, we need to use an alternate form of identity \eqref{e:weak1}:
\begin{equation}\label{e:identity}
\begin{split}
\int_{\R^d} \varphi Q(f,f) \dd v &= c_{d,\gamma,s} \int_{\R^d} \int_{\R^d} \int_{\R^d} f(v) f(v_*)\frac{\varphi(v+w) - \varphi(v)}{|w|^{d+2s}} |v-v_*+w|^{\gamma+2s} \dd w \dd v_* \dd v\\
&= c_{d,\gamma,s}\int_{\R^d} \int_{\R^d}  f(v) [f \ast |\cdot|^{\gamma+2s}](v+w) \frac{\varphi(v+w) - \varphi(v)}{|w|^{d+2s}} \dd w \dd v.
\end{split}
\end{equation}

\begin{theorem}\label{t:L2}
Let $\gamma$ and $s$ satisfy
\[
\gamma \geq -\frac {d+4s} 3,
\]
and let $f\geq 0$ be a classical solution to the (homogeneous) isotropic Boltzmann equation \eqref{e:homogeneous} in $C^1([0,T), C^2(\R^d))$, with $T\leq +\infty$, such that $\int_{\R^d} f^2(0,v) \dd v < +\infty$.  

Then the $L^2$ norm of $f$ is non-increasing in time, and
\[
\int_{\R^d} f^2(t,v) \dd v \leq \int_{\R^d} f^2(0,v) \dd v, \quad t\in [0,T).
\]
\end{theorem}

\begin{proof}
Integrate the equation against $f$:
\begin{equation}\label{e:energy-est-L2}
 \frac 1 2 \frac d {dt} \int_{\R^d} f^2 \dd v = \int_{\R^d} f Q(f,f) \dd v.
 \end{equation}
 With the $Q_1+Q_2$ decomposition discussed in Section \ref{s:carleman}, this right-hand side equals
 \begin{equation}
 \begin{split}
\int_{\R^d} f Q(f,f) \dd v &= \int_{\R^d} f  (Q_1(f,f) + Q_2(f,f)) \dd v\\
 &=  c_{d,\gamma,s}\left(\int_{\R^d} \int_{\R^d} f(v) [f\ast|\cdot|^{\gamma+2s}](v) \frac{f(v+w) - f(v)}{|w|^{d+2s}} \dd w \dd v\right.\\
 &\quad \left.+ c_R\int_{\R^d} f^{2}(v) [f\ast|\cdot|^\gamma](v) \dd v\right).
\end{split}
\end{equation}
Applying the Hardy inequality \eqref{e:hardy} to the second term on the right, with $u = f$, we obtain
\[
\begin{split}
c_R \int_{\R^d} f^2(v) [f\ast |\cdot|^\gamma](v) \dd v &\leq \frac {c_R} {C_H} \int_{\R^d}\int_{\R^d} \frac{|f(v+w) - f(v)|^2}{|w|^{d+2s}} [f\ast |\cdot|^{\gamma+2s}](v) \dd w \dd v\\
&\leq \frac {c_R} {C_H} \left( \int_{\R^d}\int_{\R^d}  f(v+w)\frac{f(v+w) - f(v)}{|w|^{d+2s}} [f\ast |\cdot|^{\gamma+2s}](v) \dd w \dd v\right.\\
&\quad \left. + \int_{\R^d} \int_{\R^d} f(v) \frac{f(v) - f(v+w)}{|w|^{d+2s}} [f\ast |\cdot|^{\gamma+2s}](v) \dd w \dd v\right).
\end{split}
\]
We notice that $c_{d,\gamma,s}$ times the second integral on the right is equal to $- \int_{\R^d} f Q_1(f,f) \dd v$.  For the first integral on the right, we change variables to exchange $v$ and $v+w$ and use formula \eqref{e:identity} with $\varphi = f$ to write
\[
\begin{split}
c_{d,\gamma,s}\int_{\R^d}\int_{\R^d}  f(v+w) & \frac{f(v+w) - f(v)}{|w|^{d+2s}} [f\ast |\cdot|^{\gamma+2s}](v) \dd w \dd v\\
 &= c_{d,\gamma,s}\int_{\R^d}\int_{\R^d}  f(v)\frac{f(v) - f(v+w)}{|w|^{d+2s}} [f\ast |\cdot|^{\gamma+2s}](v+w) \dd w \dd v\\
&= -\int_{\R^d} f Q(f,f) \dd v.
\end{split}
\]
We therefore have
\[
\int_{\R^d} f Q(f,f) \dd v \leq \left( 1 - \frac {c_R} {C_H}\right) \int_{\R^d} f Q_1(f,f) \dd v - \frac {c_R} {C_H} \int_{\R^d} f Q(f,f) \dd v,
\]
or
\begin{equation}\label{e:cC}
\int_{\R^d} f Q(f,f) \dd v \leq \left( 1 + \frac {c_R} {C_H}\right)^{-1} \left( 1 - \frac {c_R} {C_H}\right) \int_{\R^d} f Q_1(f,f) \dd v.
\end{equation}
Next, we claim that $\int_{\R^d} f Q_1(f,f) \dd v \leq 0$. Indeed, using \eqref{e:identity} with the change of variables $v\leftrightarrow v+w$ again,
\[
\begin{split}
\int_{\R^d} f Q_1(f,f) \dd v &= c_{d,\gamma,s} \int_{\R^d}\int_{\R^d} f(v) \frac{f(v+w) - f(v)}{|w|^{d+2s}} [f\ast |\cdot|^{\gamma+2s}](v) \dd w \dd v\\
&= - c_{d,\gamma,s}\int_{\R^d}\int_{\R^d} \frac{|f(v+w) - f(v)|^2}{|w|^{d+2s}} [f\ast |\cdot|^{\gamma+2s}](v) \dd w \dd v\\
&\quad + c_{d,\gamma,s} \int_{\R^d} \int_{\R^d} f(v+w) \frac{f(v+w) - f(v)}{|w|^{d+2s}} [f\ast |\cdot|^{\gamma+2s}](v) \dd v\\
&\leq  - \int_{\R^d} f Q(f,f) \dd v\\
&= -\int_{\R^d} f Q_1(f,f) \dd v - \int_{\R^d} f Q_2(f,f)\dd v,
\end{split}
\]
which implies $2\int_{\R^d} f Q_1(f,f)\dd v \leq -\int_{\R^d} f Q_2(f,f) \dd v = -\int_{\R^d} f^2 [f\ast|\cdot|^\gamma] \dd v \leq 0$.

We conclude $\int_{\R^d} f Q(f,f) \dd v \leq 0$ whenever $c_R/{C_H} \leq 1$. Using this in \eqref{e:energy-est-L2}, we see that
\[
\frac 1 2 \frac d {dt}\int_{\R^d} f^2 \dd v \leq 0,
\]
as desired. From Lemma \ref{l:constants} below, the condition $c_R/{C_H}\leq 1$ is true exactly when $\gamma \geq -\frac {d+4s} 3$.
\end{proof}

\begin{lemma}\label{l:constants}
For $d$, $s$, and $\gamma$ such that 
\[
-2s> \gamma \geq -\frac {d+4s} 3,
\]
one has $c_R\leq C_H$, where $c_R$ is defined in \eqref{e:cR} and $C_H$ is the constant from the weighted fractional Hardy inequality, Theorem \ref{t:hardy}.
\end{lemma}

Numerical computations show that this lemma is sharp, i.e. $c_R>C_H$ whenever $\gamma \in (-d, -\frac {d+4s}3)$.

\begin{proof}
Recalling \eqref{e:cR} and Theorem \ref{t:hardy}, we have
\[ 
\begin{split}
\frac {c_R} {C_H} &=  \frac{\pi^{d/2} |\Gamma(-s)| \Gamma(\frac{\gamma+2s+d} 2)\Gamma(\frac{-\gamma} 2)} {\Gamma(\frac{d+2s}{2})\Gamma(\frac{\gamma+d} 2)\Gamma(-\frac {\gamma+2s} 2)} \left( \frac{\pi^{d/2}|\Gamma(-s)|}{\Gamma(\frac{d+2s} 2)} \left[ \frac{ 2 \Gamma(\frac {d-\gamma} 4) \Gamma(\frac {d+\gamma+4s} 4)} { \Gamma(\frac{d+\gamma} 4) \Gamma(\frac{d-\gamma-4s} 4)} - \frac{ \Gamma(\frac{-\gamma} 2) \Gamma(\frac{d+\gamma+2s} 2)} {\Gamma(\frac{d+\gamma} 2)\Gamma(\frac{-\gamma-2s} 2)} \right]\right)^{-1}\\
&= \frac{ \Gamma(\frac{\gamma+2s+d} 2)\Gamma(\frac{-\gamma} 2)} {\Gamma(\frac{\gamma+d} 2)\Gamma(-\frac {\gamma+2s} 2)} \left(   \frac{ 2 \Gamma(\frac {d-\gamma} 4) \Gamma(\frac {d+\gamma+4s} 4)} { \Gamma(\frac{d+\gamma} 4) \Gamma(\frac{d-\gamma-4s} 4)} - \frac{ \Gamma(\frac{-\gamma} 2) \Gamma(\frac{d+\gamma+2s} 2)} {\Gamma(\frac{d+\gamma} 2)\Gamma(\frac{-\gamma-2s} 2)} \right)^{-1}\\
&= \left(   \frac{ 2 \Gamma(\frac {d-\gamma} 4) \Gamma(\frac {d+\gamma+4s} 4) \Gamma(\frac{d+\gamma} 2)\Gamma(\frac{-\gamma-2s} 2)   } { \Gamma(\frac{d+\gamma} 4) \Gamma(\frac{d-\gamma-4s} 4) \Gamma(\frac{-\gamma} 2) \Gamma(\frac{d+\gamma+2s} 2) } - 1  \right)^{-1}.
\end{split}
\]
This expression is $\leq 1$ whenever
\begin{equation}\label{e:Psi-def}
\Phi(d,s,\gamma) :=  \frac{  \Gamma(\frac {d-\gamma} 4) \Gamma(\frac {d+\gamma+4s} 4) \Gamma(\frac{d+\gamma} 2)\Gamma(\frac{-\gamma-2s} 2)   } { \Gamma(\frac{d+\gamma} 4) \Gamma(\frac{d-\gamma-4s} 4) \Gamma(\frac{-\gamma} 2) \Gamma(\frac{d+\gamma+2s} 2) } \geq 1.
\end{equation}
For $\gamma \in (-\frac{d+4s}{3}, -2s)$, all of the evaluations of the $\Gamma$ function in \eqref{e:Psi-def} are well-defined and nonzero. We also observe directly that 
\[
\Phi\left(d,s,- \frac{d+4s} 3\right) = \frac{  \Gamma(\frac {d+s}3) \Gamma(\frac {d+4s} 6) \Gamma(\frac{d-2s} 3)\Gamma(\frac{d-2s} 6)   } { \Gamma(\frac{d-2s} 6) \Gamma(\frac{d-2s} 3) \Gamma(\frac{d+4s} 6) \Gamma(\frac{d+s} 3) } = 1.
\] 
Therefore, the lemma will follow from showing $\Phi(d,s,\gamma)$ is increasing in $\gamma \in (-\frac{d+4s} 3, -2s)$, for each fixed $d$ and $s$. Taking derivatives, and letting $\psi$ denote the digamma function, $\psi = \Gamma'/\Gamma$, we have
\begin{equation}\label{e:Psi-d}
\begin{split}
\partial_\gamma \Phi(d,s,\gamma) &= \Phi(d,s,\gamma)\left[ -\frac 1 4 \psi\left(\frac{d-\gamma} 4 \right) -\frac 1 4 \psi\left( \frac{d+\gamma} 4\right) +\frac 1 4\psi\left(\frac{d+\gamma+4s}4\right) + \frac 1 4 \psi\left( \frac{d-\gamma-4s} 4\right)  \right.\\
&\qquad \qquad \left.+ \frac 1 2 \psi\left( \frac{d+\gamma} 2\right) + \frac 1 2 \psi\left( \frac{-\gamma} 2 \right) - \frac 1 2 \psi\left(\frac{-\gamma-2s} 2 \right) - \frac 1 2 \psi\left( \frac{d+\gamma+2s} 2\right) \right].
\end{split}
\end{equation}
It is well-known that the digamma function $\psi$ is strictly concave on $(0,\infty)$, so $\psi'$ is decreasing. To use this fact, we consider the first four terms on the right in \eqref{e:Psi-d}, in two cases. If $\gamma < -4s$, then for some $z_1, z_2$ with 
\[
\frac{d+\gamma} 4 < z_1 < \frac{d+\gamma+4s} 4 < \frac {d-\gamma-4s} 4 < z_2 < \frac{d-\gamma}4 ,
\]
there holds
\[
\begin{split}
-\frac 1 4  \psi\left(\frac{d-\gamma} 4 \right) -\frac 1 4 \psi\left( \frac{d+\gamma} 4\right) &+\frac 1 4\psi\left(\frac{d+\gamma+4s}4\right) + \frac 1 4 \psi\left( \frac{d-\gamma-4s} 4\right)\\
&= \frac s 4 [\psi'(z_1) - \psi'(z_2)] > 0,
\end{split}
\]
since $\psi'$ is decreasing. On the other hand, if $\gamma \geq -4s$, we pair the four terms differently and find $z_3,z_4$ with
\[
\frac{d+\gamma} 4 < z_3 < \frac{d-\gamma-4s} 4 \leq \frac {d+\gamma+4s} 4 < z_4 < \frac{d-\gamma}4 ,
\]
such that 
\[
\begin{split}
-\frac 1 4  \psi\left(\frac{d-\gamma} 4 \right) -\frac 1 4 \psi\left( \frac{d+\gamma} 4\right)& +\frac 1 4\psi\left(\frac{d+\gamma+4s}4\right) + \frac 1 4 \psi\left( \frac{d-\gamma-4s} 4\right)\\
&= \frac 1 4 \left( \frac {-\gamma} 2 -s\right)[\psi'(z_3) - \psi'(z_4)] > 0,
\end{split}
\]
since $\frac {-\gamma} 2 -s > 0$. We analyze the last four terms in \eqref{e:Psi-d} in a similar manner: if $-\frac d 2 \leq \gamma < -2s$, then there exist $z_5, z_6$ such that
\[
\frac {-\gamma-2s} 2 < z_5 <    \frac{-\gamma} 2 \leq \frac{d+\gamma} 2 < z_6 < \frac{d+\gamma+2s} 2,
\]
and 
\[
\begin{split}
 \frac 1 2 \psi\left( \frac{d+\gamma} 2\right) + \frac 1 2 \psi\left( \frac{-\gamma} 2 \right)& - \frac 1 2 \psi\left(\frac{-\gamma-2s} 2 \right) - \frac 1 2 \psi\left( \frac{d+\gamma+2s} 2\right)\\
 &= \frac s 2 [\psi'(z_5) -   \psi'(z_6)] >0.
\end{split}
\] 
If $-\frac d 2 - s < \gamma < -\frac  d 2$, then there exist $z_7,z_8$ with
\[
\frac {-\gamma-2s} 2 < z_7 <  \frac{d+\gamma} 2 <  \frac{-\gamma} 2   < z_8 < \frac{d+\gamma+2s} 2,
\]
and
\[
\begin{split}
 \frac 1 2 \psi\left( \frac{d+\gamma} 2\right) + \frac 1 2 \psi\left( \frac{-\gamma} 2 \right)& - \frac 1 2 \psi\left(\frac{-\gamma-2s} 2 \right) - \frac 1 2 \psi\left( \frac{d+\gamma+2s} 2\right)\\
 &= \frac s 2 [\psi'(z_7) -   \psi'(z_8)] >0.
\end{split}
\] 
 We have shown $\partial_\gamma \Phi(d,s,\gamma)>0$ whenever $\gamma > -\frac d 2 -s$. Since $-\frac{d+4s} 3 > -\frac d 2 -s$, the proof is complete.
\end{proof}

Next, we show that our $L^2$ bound implies a bound on the energy $\int_{\R^d} |v|^2 f(t,v) \dd v$ on any finite time interval:

\begin{lemma}\label{l:energy}
Assume 
\[
\max\{-\frac d 2 - 1, -d-1-2s\}< \gamma < -2,
\]
and let $f\geq 0$ be a classical solution of \eqref{e:homogeneous} in $C^1([0,T), C^2(\R^d))$ with finite mass and energy at $t=0$, i.e.
\[
\int_{\R^d} (1+|v|^2) f(0,v) \dd v < +\infty,
\]
and whose $L^2$ norm is nonincreasing, $\int_{\R^d} f^2(t,v) \dd v \leq \int_{\R^d} f^2(0,v)\dd v$. 

Then the energy of $f$ satisfies the inequality
\[
\begin{split}
\int_{\R^d} |v|^2 f \dd v &\leq \exp\left( C t \left(\|f_{\rm in}\|_{L^1(\R^d)} + \|f_{\rm in}\|_{L^2(\R^d)}\right)\right)\\
&\quad \times \left( \int_{\R^d} |v|^2 f(0,v) \dd v + C \left(\|f_{\rm in}\|_{L^1(\R^d)} + \|f_{\rm in}\|_{L^2(\R^d)}\right)\right),
\end{split}
\]
for a constant $C>0$ depending only on $d$, $\gamma$, and $s$. 
\end{lemma}

Note that this lower bound on $\gamma$ follows from our assumption $\gamma \geq -\frac{d+4s}3$ in Theorem \ref{t:main}.

Also, note that the condition $\gamma< -2$ is really needed for our proof. Indeed, the right-hand side of \eqref{e:integrate-against-v2} can in general be infinite if $\gamma \geq -2$. 

\begin{proof}
Integrating equation \eqref{e:homogeneous} against $|v|^2$, and using \eqref{e:weak1}, gives 
\begin{equation}\label{e:integrate-against-v2}
\begin{split}
\frac d {dt} \int_{\R^d} |v|^2  f\dd v &= \int_{\R^d} |v|^2  Q(f,f)\dd v \\
&= c_{d,\gamma,s}\int_{\R^d} \int_{\R^d}\int_{\R^d}|v-v_*+w|^{\gamma+2s} |w|^{-d-2s} f(v) f(v_*)\\
&\qquad\qquad \times [|v+w|^2 - |v|^2] \dd w \dd v_* \dd v.
\end{split}
\end{equation}
Let us consider the inner $w$ integral, which we denote
\[
I := \int_{\R^d} |v-v_*+w|^{\gamma+2s} |w|^{-d-2s} [|v+w|^2 - |v|^2] \dd w.
\]
With $r = 2|v-v_*|$, we divide $I$ as follows:
\[
I = I_1 + I_2+ I_3,
\]
with
\[
\begin{split}
I_1 &=   \int_{B_r} |v-v_*+w|^{\gamma+2s} |w|^{2-d-2s} \dd w,\\
I_2 &=  2\int_{B_r}|v-v_*+w|^{\gamma+2s} |w|^{-d-2s} v\cdot w \dd w,\\
I_3 &= \int_{\R^d\setminus B_r} |v-v-*+w|^{\gamma+2s} [|w|^2 + 2v\cdot w]\dd w.
\end{split}
\]

For $I_1$, we subdivide $B_r= B_r(0)$ into $B_r(0) \setminus B_{r/4}(v-v_*)$ and $B_{r/4}(v-v_*)$. The first part is bounded using $|v-v_*+w| \geq r/4 = |v-v_*|/2$:
\[
\int_{B_r\setminus B_{r/4}(v-v_*)} |v-v_*+w|^{\gamma+2s} |w|^{2+d-2s} \dd w \lesssim r^{\gamma+2s} \int_{B_r} |w|^{2-d-2s} \dd w \lesssim |v-v_*|^{\gamma+2}.
\]
Next, for $w\in B_{r/4}(v-v_*)$, we must have $|w|\geq |v-v_*|/2$, so 
\[
\int_{B_{r/4}(v-v_*)} |v-v_*+w|^{\gamma+2s} |w|^{2+d-2s} \dd w \lesssim |v-v_*|^{2+d-2s} \int_{B_{r/4}(0)} |u|^{\gamma+2s} \dd u \lesssim |v-v_*|^{\gamma+2}.
\]
We conclude
\[
I_1 \lesssim |v-v_*|^{\gamma+2}.
\]

For $I_2$, which is more singular as $w\to 0$, we use the fact that
\[
\mathrm{p.v.} \int_{B_r} |w|^{-d-2s} w |v-v_*|^{\gamma+2s}  \dd w = 0,
\]
to write
\[
\begin{split}
I_2 = 2\int_{B_r}& |w|^{-d-2s} v\cdot w \left[|v-v_*+w|^{\gamma+2s} - |v-v_*|^{\gamma+2s}\right]\dd w.
\end{split}
\]
Now we make the same splitting into $B_{r/4}(v-v_*)$ and $B_r(0) \setminus B_{r/4}(v-v_*)$ that was used to estimate $I_1$. In $B_r(0) \setminus B_{r/4}(v-v_*)$, letting $F(z) = |z|^{\gamma+2s}$, we can differentiate $F$ because $\gamma+2s-1 > - d$, by assumption. We therefore have
\[
\begin{split}
|v-v_*+w|^{\gamma+2s} - |v-v_*|^{\gamma+2s}  \leq & |w|\max_{\sigma\in [0,1]} |D F(v-v_*+\sigma w)|\\
&\leq C|w| |v-v_*|^{\gamma+2s-1},
\end{split}
\]
since $|v-v_*+\sigma w| \geq \frac 1 2 |v-v_*|$ for all $\sigma \in [0,1]$. This implies
\[
\begin{split}
&\int_{B_r\setminus B_{r/4}(v-v_*)} |w|^{-d-2s} v\cdot w \left[|v-v_*+w|^{\gamma+2s}  - |v-v_*|^{\gamma+2s} \right] \dd w \\
&\qquad\leq C|v-v_*|^{\gamma+2s-1} |v|\int_{B_r\setminus B_{r/4}(v-v_*)} |w|^{-d-2s+2} \dd w  \\
&\qquad \leq C|v-v_*|^{\gamma +1}|v|. 
\end{split}
\]
Next, in $B_{r/4}(v-v_*)$, we use once again that $|w|\geq \frac 1 2 |v-v_*|$:
\[
\begin{split}
&\int_{B_{r/4}(v-v_*)} |w|^{-d-2s} v\cdot w \left[|v-v_*+w|^{\gamma+2s}  - |v-v_*|^{\gamma+2s} \right] \dd w \\
 &\qquad\leq C|v-v_*|^{-d-2s+1}|v|\left(\int_{B_{r/4}(v-v_*)} |v-v_*+w|^{\gamma+2s} \dd w+ |v-v_*|^{\gamma+2s} \int_{B_{r/4}(v-v_*)} \dd w\right)\\
&\qquad\leq C |v-v_*|^{\gamma+1}|v|.
\end{split}
\]
We have shown
\[
I_2 \lesssim  |v-v_*|^{\gamma+1} |v|.
\]

Next, we estimate $I_3$. In this domain, we have $|w|\geq r = 2|v-v_*|$, which implies $|v-v_*+w| \geq |w|- |v-v_*| \geq \frac 1 2 |w|$. We also have $|v-v_*+w| \leq \frac 3 2 |w|$, so in fact $|v-v_*+w| \approx |w|$ in this region. This gives
\[
\begin{split}
I_3 
&= \int_{\R^d\setminus B_r} |w|^{-d-2s} |v-v_*+w|^{\gamma+2s}  (|w|^2 + 2v\cdot w) \dd w\\ &\lesssim \int_{\R^d\setminus B_r} |w|^{-d+\gamma+2} \dd w + |v| \int_{\R^d\setminus B_r} |w|^{-d+\gamma+1} \dd w\\
&\lesssim |v-v_*|^{\gamma+2} + |v| |v-v_*|^{\gamma+1}.
\end{split}
\]
since $r\approx |v-v_*|$ and $\gamma+2 < 0$.

Returning to \eqref{e:integrate-against-v2} and collecting our upper bounds for $I_1$, $I_2$, and $I_3$, we have
\begin{equation}\label{e:energy-step}
\begin{split}
\frac d {dt} \int_{\R^d} |v|^2  f \dd v &\leq C\int_{\R^d} \int_{\R^d} f(v) f(v_*) \left( |v-v_*|^{\gamma+2} + |v| |v-v_*|^{\gamma+1} \right)\dd v_* \dd v.
\end{split}
\end{equation}
The right hand side can be bounded by standard convolution estimates. In detail, dividing the $v_*$ integral into $B_1(v)$ and $\R^d\setminus B_1(v)$, since $\gamma+2 > \gamma+1 > - \frac d 2$, we have
\[
\int_{B_1(v)} f(v_*) |v-v_*|^{\gamma+2} \dd v_* \leq \|f\|_{L^2(\R^d)} \left(\int_{B_1(v)} |v-v_*|^{2(\gamma+2)} \dd v_* \right)^{1/2} \leq C \|f\|_{L^2(\R^d)},
\]
and 
\[
\int_{\R^d \setminus B_1(v)} f(v_*)|v-v_*|^{\gamma+2} \dd v_* \leq \int_{\R^d\setminus B_1(v)} f(v_*) \dd v_* \leq \|f\|_{L^1(\R^d)}.
\]
For the second term in \eqref{e:energy-step}, we proceed similarly, since $\gamma+1> -\frac d 2$:
\[
|v|\int_{B_1(v)} f(v_*) |v-v_*|^{\gamma+1} \dd v_* \leq |v|\|f\|_{L^2(\R^d)} \left(\int_{B_1(v)} |v-v_*|^{2(\gamma+1)} \dd v_* \right)^{1/2} \leq C |v| \|f\|_{L^2(\R^d)},
\]
and 
\[
|v|\int_{\R^d \setminus B_1(v)} f(v_*)|v-v_*|^{\gamma+1} \dd v_* \leq |v| \|f\|_{L^1(\R^d)},
\]
and we finally have
\[
\begin{split}
\frac 1 2 \frac d {dt} \int_{\R^d} |v|^2 f \dd t &\leq C\left(\|f\|_{L^1(\R^d)} + \|f\|_{L^2(\R^d)}\right) \int_{\R^d} (1+|v|)f(v) \dd v \\
&\leq  C\left(\|f\|_{L^1(\R^d)} + \|f\|_{L^2(\R^d)}\right)\left( \|f\|_{L^1(\R^d)} + \int_{\R^d}|v|^2 f \dd v\right).
\end{split}
\]
Gr\"onwall's inequality implies the conclusion of the lemma, since $\|f(t)\|_{L^1(\R^d)}$ and $\|f(t)\|_{L^2(\R^d)}$ are bounded by their values at $t=0$. 
\end{proof}

%
%
%
%

\section{Global upper bounds}\label{s:Linfty}

This section establishes bounds for our solution $f(t)$ in $L^\infty(\R^d)$, using a barrier method inspired by the approach of \cite{silvestre2016boltzmann} and \cite{imbert2018decay}. We also establish estimates in polynomially-weighted $L^\infty(\R^d)$ spaces.

\begin{lemma}\label{l:coercive}
Let $f:\R^d\to \R$ be nonnegative and satisfy
\[
\begin{split}
m_0 \leq \int_{\R^d} f \dd v &\leq M_0,\\
\int_{\R^d} |v|^2 f\dd v &\leq E_0,\\
\int_{\R^d} f \log f \dd v &\leq H_0,
\end{split}
\]
for some positive constants $m_0, M_0, E_0, H_0$. Then there exists a constant $c_0>0$ depending on $m_0$, $M_0$, $E_0$, and $H_0$, such that
\[
K_f(v,w) \geq c_0 \langle v\rangle^{\gamma+2s} |w|^{-d-2s},
\]
for any $w\in \R^d$.
\end{lemma}

\begin{proof}
With $f$ as in the statement of the lemma, \cite[Lemma 4.6]{silvestre2016boltzmann} states that there exist  $r, \ell, \mu>0$ depending on $m_0, M_0, E_0$, and $H_0$, such that
\[
| \{v\in \R^d: f(v) > \ell\} \cap B_r| \geq \mu.
\]
Let $S = \{f(v)>\ell\} \cap B_r$. 
Since $f\geq \ell\chi_S$, we clearly have
\[
K_f(v,w) \geq  K_{\ell\chi_S}(v,w) = c_{d,s} \ell |w|^{-d-2s}\int_{\R^d} \chi_S(u) |v-u|^{\gamma+2s} \dd u.
\]
Since $S\subset B_r$, we have $|v-u|\leq |v|+r$, and
\[
\int_{\R^d} \chi_S(u) |v-u|^{\gamma+2s} \dd u = \int_S |v-u|^{\gamma+2s} \dd u \leq (|v|+r)^{\gamma+2s} |S| \geq \mu (|v|+r)^{\gamma+2s},
\]
since $\gamma+2s\leq 0$. The conclusion of the lemma follows.
\end{proof}

The following upper bound for $Q_2(f,g)$ will be needed in the proof of the $L^\infty$ bound, Proposition \ref{p:Linfty}. The key point is that using the $L^2$ estimate for $f$ (rather than the $L^1$ estimate provided by conservation of mass) leads to a less severe dependence on $\|f\|_{L^\infty}$.
\begin{lemma}\label{l:Q2est}
If $f\in L^2(\R^d)$ satisfies the bounds $\int_{\R^d} f\dd v \leq M_0$ and $\int_{\R^d} |v|^2 f \dd v \leq E_0$, then for any $\sigma$ satisfying
\[ 
\left(-\frac{2\gamma} d - 1\right)_+ < \sigma \leq 1, 
\]
there holds
\[
Q_2(f,g)(v) \leq C g(v) \langle v\rangle^{\max\{\gamma, -2 -4\gamma/(d(1+\sigma))\}} \left(1+\|f\|_{L^\infty(\R^d)}^{-2\sigma \gamma/(d(1+\sigma))}\right),
\]
where the constant $C$ depends on $c_{d,\gamma,s}$, $\sigma$, $M_0$, $E_0$, and $\|f\|_{L^2(\R^d)}$. 
\end{lemma}

\begin{proof}
Recall that $Q_2(f,g)(v) \approx g(v)[f\ast |\cdot|^\gamma](v)$. 

First, assume $|v|\geq 1$. Letting 
\[
 \rho = |v|^{-\frac 4 {d(1+\sigma)} }\|f\|_{L^\infty}^{-\frac{2\sigma}{d(1+\sigma)}},
\] 
we divide the integral defining $[f\ast |\cdot|^\gamma](v)$ into 
\[
\begin{split}
I_1 = \int_{B_\rho}|w|^\gamma f(v-w) \dd w, \quad I_2 = \int_{B_{|v|/2}\setminus B_\rho} |w|^\gamma f(v-w)\dd w, \quad I_3 = \int_{\R^d \setminus B_{|v|/2}} |w|^\gamma f(v-w) \dd w.
\end{split}
\]
For $I_1$, we use both the $L^2$ and $L^\infty$ bounds for $f$:
\[
\begin{split}
I_1 &\leq \|f\|_{L^\infty}^\sigma \int_{B_\rho} |w|^\gamma f^{1-\sigma}(v-w) \dd w\\
 &\leq  \|f\|_{L^\infty}^\sigma \|f\|_{L^2}^{1-\sigma} \left( \int_{B_\rho} |w|^{2\gamma/(1+\sigma)} \dd w\right)^{(1+\sigma)/2}\\
&\lesssim \rho^{\gamma + d(1+\sigma)/2} \|f\|_{L^\infty}^\sigma \|f\|_{L^2}^{1-\sigma}\\
&\lesssim |v|^{-2 - 4\gamma/(d(1+\sigma))} \|f\|_{L^\infty}^{-2\sigma\gamma/(d(1+\sigma))}.
\end{split}
\]
The condition $\sigma> -\frac {2\gamma} d - 1$ ensures that $|w|^{2\gamma/(1+\sigma)}$ is integrable near $w=0$. 

For $I_2$, we use the energy bound. Since $|v-w|\gtrsim |v|$ when $|w|\leq |v|/2$, 
\[
\begin{split}
I_2 \leq \rho^\gamma |v|^{-2} \int_{B_{|v|/2} \setminus B_\rho} |v-w|^2 f(v-w) \dd w \leq \rho^\gamma |v|^{-2} E_0 \lesssim |v|^{-2-4\gamma/(d(1+\sigma))} \|f\|_{L^\infty}^{-2\sigma\gamma/(d(1+\sigma))}.
\end{split}
\]
Finally, for $I_3$, we use the bound on the mass:
\[
I_3 \lesssim |v|^\gamma \int_{\R^d\setminus B_{|v|/2}} f(v-w) \dd w \leq M_0 |v|^\gamma.
\]

In the case $|v|< 1$, we choose $\rho' = \|f\|_{L^\infty}^{-2\sigma/(d(1+\sigma))}$, and proceed as in the estimate of $I_1$ to write
\[
\int_{B_{\rho'}} |w|^\gamma f(v-w) \dd w \lesssim (\rho')^{\gamma+d(1+\sigma)/2} \|f\|_{L^\infty}^\sigma \lesssim \|f\|_{L^\infty}^{-2\sigma \gamma/(d(1+\sigma))},
\]
as well as 
\[
\int_{\R^d \setminus B_{\rho'}} |w|^\gamma f(v-w) \dd w \lesssim (\rho')^\gamma M_0 \lesssim \|f\|_{L^\infty}^{-2\sigma \gamma/(d(1+\sigma))},
\]
as desired.
\end{proof}

Next, we prove a global upper bound that is valid away from $t=0$. We assume $f$ is a classical solution that is smooth enough that conservation of mass holds, by the argument in Section \ref{s:conservation}. 

\begin{proposition}\label{p:Linfty}
Assume 
\begin{equation}\label{e:gamma-s}
\begin{cases} -\dfrac d 2 +s < \gamma +2s < 0, &\text{if } -\dfrac d 2 +s \geq -\dfrac {4s} d,\\[7pt]
-\dfrac {4s} d \leq \gamma + 2s < 0, &\text{if } -\dfrac d 2 + s< -\dfrac{4s} d.
\end{cases}
\end{equation}
For any bounded classical solution $f$ of \eqref{e:homogeneous} on $[0,T]\times \R^d$ satisfying
\[
\int_{\R^d} f(t,v) \dd v = M_0, \quad \int_{\R^d} |v|^2 f(t,v) \dd v \leq E_0, \quad \|f\|_{L^2(\R^d)} \leq L_0,
\]
the following upper bound holds:
\[
f(t,v) \leq N \left( t^{-d/(2s)} + 1\right),
\]
for some $N>0$ depending only on $d$, $\gamma$, $s$, $M_0, E_0$, and $L_0$. 
\end{proposition}

Note that the first condition $\gamma + 2s > -\frac d 2 + s$ in \eqref{e:gamma-s} follows from our condition $\gamma \geq -\frac{d+4s}3$ in Theorem \ref{t:main}. 

\begin{proof}
First, note that the bound for $f$ in $L^2(\R^d)$ implies an upper bound on the entropy $\int_{\R^d} f \log f \dd v$, which will allow us to apply Lemma \ref{l:coercive}. 

Following \cite{silvestre2016boltzmann}, let
\[
h(t) = N \left( t^{-d/(2s)}+ 1 \right),
\]
with $N\geq 1$ to be chosen later. We claim that $f(t,v) < h(t)$ for all $(t,v) \in [0,T]\times \R^d$. Since $f$ is bounded, the inequality $f< h$ must be true on some small time interval near zero, so if our claim is false, there is a point $(t_0,v_0)$ where $f(t_0,v_0) = h(t_0)$ for the first time. At this point, we must have
\begin{equation}\label{e:dt}
\partial_t f(t_0,v_0) \geq \partial_t h(t_0) = -\frac {dN} {2s} t_0^{-1-d/(2s)}.
\end{equation}
From the equation, we have
\[
\partial_t f(t_0,v_0) = Q_1(f,f)(t_0,v_0) + Q_2(f,f)(t_0,v_0).
\]
We will find an upper bound for $Q_1+Q_2$ at the crossing point $(t_0,v_0)$ that contradicts \eqref{e:dt}. 

First, we consider the case where $|v_0|$ is large, i.e. $|v_0|\geq R$ for an $R>1$ that will be chosen below. For the term $Q_1(f,f)$, we note that $f(t_0,v) \leq h(t_0)$ for all $v\in \R^d$, which implies
\begin{equation}\label{e:Q1est}
\begin{split}
Q_1(f,f)(t_0,v_0) &\leq - \int_{\R^d} K_f(t_0,v_0,w) [h(t_0) - f(t_0,v_0+w)] \dd w\\
&\leq - c_0 | v_0|^{\gamma+2s}\int_{\R^d} |w|^{-d-2s}  [h(t_0) - f(t_0,v_0+w)] \dd w,
\end{split}
\end{equation}
with $c_0$ as in Lemma \ref{l:coercive}. This principal value integral is well-defined because $f(t_0,v_0) = h(t_0)$. To get a good (negative) upper bound for this integral, we apply an argument inspired by \cite[Proposition 3.3] {imbert2018decay}. The idea is to find a set where (i) $f(t_0,v_0+w) < h(t_0)/2$ and (ii) $|w|$ is relatively small, so that $|w|^{-d-2s}$ is large. With $r\in (0, |v_0|/4)$ to be chosen later, Chebyshev's inequality implies
\[
\begin{split}
|\{w: f(v_0+w) \geq h(t_0)/2\} \cap (B_{2r} \setminus B_r)| &\leq \frac 2 {h(t_0)} \int_{B_{2r}\setminus B_r} f(v_0+w) \dd w\\
&\leq \frac 8 {|v_0|^2 h(t_0)} \int_{B_{2r}\setminus B_r} |v_0+w|^2 f(v_0+w) \dd w\\
&\leq \frac {8 E_0}{|v_0|^2 h(t_0)},
\end{split}
\] 
where in the second line, we used that $|v_0+w| \geq |v_0|/2$ when $|w|\leq 2r \leq |v_0|/2$. Letting $c_d$ be the constant such that $|B_{2r}\setminus B_r| = c_d r^d$, we want to choose $r$ such that
\[
4 |\{w: f(v_0+w) \geq h(t_0)/2\} \cap (B_{2r} \setminus B_r)| \leq \frac {32 E_0}{|v_0|^2 h(t_0)} \leq |B_{2r}\setminus B_r| = c_d r^d.
\]
This will imply that $f(v_0+w) < h(t_0)/2$ in at least three fourths of the set $B_{2r}\setminus B_r$. The appropriate choice of $r$ is given by
\[
r = \left( \frac{32 E_0}{c_d |v_0|^2 h(t_0)}\right)^{1/d}.
\]
We also choose
\[
R = 1+ 4^{d/(d+2)} \left( \frac {32 E_0}{c_d}\right)^{1/(d+2)},
\]
so that $|v_0|\geq R$ implies $|v_0|^{d+2} \geq 4^d \cdot 32 E_0/(c_d h(t_0))$ (recall that $h(t_0)\geq 1$), which implies $r< |v_0|/4$. 

Returning to \eqref{e:Q1est}, since the integrand $h(t_0) - f(t_0,v_0+w)$ is non-negative everywhere, we can write
\[
\begin{split}
Q_1(f,f)(t_0,v_0) &\leq -c_0 | v_0|^{\gamma+2s} \frac {h(t_0)} 2 \int_{(B_{2r}\setminus B_r)\cap \{f(v_0+w) < h(t_0)/2\}} |w|^{-d-2s}\dd w \\
&\leq  -c_0 | v_0|^{\gamma+2s} \frac {h(t_0)} 2 r^{-d-2s} |(B_{2r}\setminus B_r)\cap \{f(v_0+w) < h(t_0)/2\}|\\
&\leq -c_0 | v_0|^{\gamma+2s} \frac {h(t_0)} 2 r^{-d-2s} \frac 3 4 c_d r^d\\
&\leq -c_0 | v_0|^{\gamma+2s+4s/d} h(t_0)^{1+2s/d},
\end{split}
\]
by our choice of $r$. The value of $c_0$ has changed line by line, but still depends only on $d$, $\gamma$, $s$, $M_0$, $E_0$, and $L_0$. 

For the term $Q_2(f,f) \approx f[f\ast|\cdot|^\gamma]$, Lemma \ref{l:Q2est} implies, for some $\sigma$ to be chosen below, 
\[
Q_2(f,f)(t_0,v_0) \leq C h(t_0)^{1-2\sigma\gamma/(d(1+\sigma))} |v_0|^{\max\{\gamma, - 2 - 4\gamma/(d(1+\sigma))\}},
\]
since $\|f(t_0,\cdot)\|_{L^\infty(\R^d)} \leq h(t_0)$. We want to choose $\sigma$ so that the positive term coming from $Q_2$ is smaller than the negative term from $Q_1$. Since $h(t_0) \geq 1$, this means we need
\[
\gamma+2s+\frac{4s}{d} > -2 -\frac{4\gamma}{d(1+\sigma)} \quad \text{and} \quad 1+\frac{2s}d > 1 - \frac{2\sigma \gamma}{d(1+\sigma)}.
\]
After some straightforward algebra using $\gamma+2s+2> 0$ (which follows from \eqref{e:gamma-s}), this translates to 
\begin{equation}\label{e:sigma1}
\frac s {-s-\gamma} < \sigma < \frac{-4\gamma}{4s+d(\gamma+2s+d)} - 1.
\end{equation}
A number $\sigma$ satisfying both these inequalities exists exactly when $\gamma+2s > -2d/(d+4)$, which is true by our assumption \eqref{e:gamma-s}.\footnote{Indeed, as functions of $s$, the lines $-d/2-s$ and $-4s/d$ cross at $s=-d^2/(2(d+4))$ with height exactly $-2d/(d+4)$.} In addition, we need $\sigma$ to satisfy the following condition in order to apply Lemma \ref{l:Q2est}:
\begin{equation}\label{e:sigma2}
\left(-\frac{2\gamma} d - 1\right)_+ < \sigma \leq 1. 
\end{equation}
Since $\gamma + 2s < 0$, we have $\dfrac s {-s-\gamma} < 1$. By our assumption that $\gamma +2s > -\frac d 2 + s$ (which is part of \eqref{e:gamma-s}), we also have $(-2\gamma/d - 1)_+ < \dfrac  s {-s-\gamma}$, so there is always a $\sigma$ satisfying both \eqref{e:sigma1} and \eqref{e:sigma2}.

With such a choice of $\sigma$, we combine our estimates of $Q_1(f,f)$ and $Q_2(f,f)$ and use $h(t_0) = N( t_0^{-d/(2s)}+1)$ to obtain
\[
Q(f,f)(t_0,v_0) \leq -\frac 1 2 c_0 N^{1+2s/d}\left( t_0^{-d/(2s)}+1\right)^{1+2s/d} | v_0|^{\gamma+2s+4s/d},
\]
if $N$ is chosen sufficiently large. This estimate holds whenever $|v_0| \geq R$.

When $|v_0| \leq R$, we obtain a negative upper bound for $Q_1(f,f)$ using the mass bound rather than the energy: for some $\rho>0$ to be determined, Chebyshev implies
\[
| \{ w: f(v_0+ w ) \geq h(t_0)/2\} \cap (B_{2\rho}\setminus B_\rho)| \leq \frac 2 {h(t_0)} \int_{B_{2\rho}\setminus B_\rho} f(v_0+w) \dd w \leq \frac {2 M_0}{h(t_0)}.
\]
With the choice $\rho = (8M_0/(c_d h(t_0)))^{1/d}$, with $c_d$ as above, we have
\[
4| \{ w: f(v_0+ w ) \geq h(t_0)/2\} \cap (B_{2\rho}\setminus B_\rho)| \leq \frac{8 M_0}{h(t_0)} \leq c_d \rho^d = |B_{2\rho} \setminus B_\rho|.
\]
Proceeding as above, we obtain
\[
\begin{split}
Q_1(f,f)(t_0,v_0) &\leq - c_0\langle R\rangle^{\gamma+2s} \frac {h(t_0)} 2 \int_{(B_{2\rho}\setminus B_\rho) \cap \{f(v_0+w) < h(t_0)/2\}} |w|^{-d-2s} \dd w\\
&\leq - c_0 \langle R\rangle^{\gamma+2s} \frac{ h(t_0)} 2 \rho^{-d-2s} \frac 3 4 c_d \rho^d\\
&\leq - c_0 h(t_0)^{1+2s/d}.
\end{split}
\]
Combined with Lemma \ref{l:Q2est}, we have
\[
Q(f,f)(t_0,v_0) \leq -c_0 h(t_0)^{1+2s/d} + C h(t_0)^{1-2\sigma \gamma/(d(1+\sigma))},
\]
for the same value of $\sigma$ chosen above. In particular, the exponent $1-2\sigma \gamma/(d(1+\sigma))$ is strictly less than $1+2s/d$, so that the negative term dominates for $N$ (and therefore $h(t_0)$) sufficiently large. We have shown, in both cases $|v_0|\geq R$ and $|v_0|< R$, 
\[
Q(f,f)(t_0,v_0) \leq -\frac 1 2 c_0 N^{1+2s/d}\left( t_0^{-d/(2s)}+1\right)^{1+2s/d} \langle v_0\rangle^{\gamma+2s+4s/d}.
\]
Combining this with \eqref{e:dt}, we have
\[
- \frac {dN}{2s} t_0^{-1-d/(2s)} \leq -\frac 1 2 c_0 N^{1+2s/d} \left( t_0^{-d/(2s)}+1\right)^{1+2s/d} \langle v_0\rangle^{\gamma+2s+4s/d},
\]
which is a contradiction for $N$ large enough, since $\gamma+2s+4s/d\geq 0$ by assumption \eqref{e:gamma-s}.
\end{proof}

Next, we prove polynomial decay estimates for $f$. In our proof of global existence, we will work with local solutions that satisfy a qualitative assumption of rapid decay on some time interval $[0,T_*)$, but the following proposition is necessary to ensure the decay estimates for $f$ do not degenerate as $t\to T_*$. 

This proposition is rather flexible in terms of the allowable values of $\gamma$ and $s$, but in order to be useful, it requires $f$ to already satisfy a quantitative $L^\infty$ bound.

\begin{proposition}\label{p:decay}
For any $d\geq 2$ and $(\gamma, s) \in (-d,0)\times(0,1)$ such that $\gamma+2s<0$, let $f:[0,T]\times\R^d\to [0,\infty)$ be a bounded solution of \eqref{e:homogeneous}. Assume in addition that $f_{\rm in} \in L^\infty_q(\R^d)$ for some $q>0$, and that the solution $f\in L^\infty_{q+\eps}([0,T]\times\R^d)$ for some $\eps>0$. Then
\[
\|f(t)\|_{L^\infty_q(\R^d)} \leq \|f_{\rm in}\|_{L^\infty_q(\R^d)} \exp( C_0 t), \quad t\in [0,T],
\]
where $C_0>0$ depends on $M_0 = \int_{\R^d} f(t,v) \dd v$ and $\|f\|_{L^\infty([0,T]\times\R^d)}$. In particular, $C_0$ does not depend quantitatively on the assumption $f\in L^\infty_{q+\eps}([0,T]\times\R^d)$. 
\end{proposition}
\begin{proof}
We use a barrier argument of a somewhat different style than the proof of Proposition \ref{p:Linfty}. Define
\[
h(t,v) = N e^{\kappa t} \vv^{-q}.
\]
with $N =  (1+\nu)\|f_{\rm in}\|_{L^\infty_q(\R^d)}$ with $\nu>0$ small, and $\kappa>0$ to be chosen below. Clearly, the inequality $f<h$ holds at $t=0$, so as in the proof of Proposition \ref{p:Linfty}, we assume by contradiction that $f$ and $h$ cross for the first time at some location $(t_0,v_0)$. We must have $t_0>0$ because of our assumption that $\vv^{q+\eps} f \in L^\infty([0,T]\times\R^d)$. 

At the crossing point, we have 
\begin{equation}\label{e:kappaN}
\kappa N e^{\kappa t_0} \langle v_0\rangle^{-q} = \partial_t h \leq \partial_t f = Q(f,f) \leq Q(f,h),
\end{equation}
where the last equality followed from the fact that $f(t_0,v_0) = h(t_0,v_0)$ and $f(t_0,v) \leq h(t_0,v)$ for all $v\in \R^d$. Since $(-\Delta)^s \vv^{-q} \lesssim  \vv^{-q-2s}$ (see Lemma \ref{l:q-power}), we have
\[
\begin{split}
Q(f,h)(t_0,v_0) &\lesssim N e^{\kappa t} \left([f\ast |\cdot|^{\gamma+2s}](v_0) \langle v_0\rangle^{-q-2s} + [f\ast |\cdot|^\gamma](v_0) \langle v_0\rangle^{-q}\right)\\
&\leq C_0 N  \langle v_0\rangle^{-q},
\end{split}
\]
where we have bounded $[f\ast |\cdot|^{\gamma+2s}]$ and $[f\ast |\cdot|^\gamma]$ in terms of $\|f\|_{L^\infty([0,T]\times\R^d)}$ and $M_0$ by a standard convolution estimate.  
This inequality is a contradiction with \eqref{e:kappaN} if we choose $\kappa \gtrsim C_0$. We conclude 
\[
f(t,v) < (1+\nu) \|f_{\rm in}\|_{L^\infty_q(\R^d)} e^{\kappa t}\vv^{-q}, \quad t\in [0,T],
\]
which implies the conclusion of the proposition, after sending $\nu\to 0$. 
\end{proof}

\section{Proof of global existence}\label{s:global}

This section is devoted to the proof of our main theorem. We start with a local solution on a time interval $[0,T]$, which is constructed in Theorem \ref{t:short-time-existence} (whose proof is postponed to Section \ref{s:short}). To show the solution can be extended past any given time, we need to show that the $H^{d+3}_q$ norm of $f$ remains finite, so that Theorem \ref{t:short-time-existence} can be reapplied. Therefore, one needs some regularization argument to bound $H^{d+3}_q$ norms in terms of the quantities (mass, energy, $L^\infty$ norm, etc.) that are globally bounded by our estimates from Sections \ref{s:L2} and \ref{s:Linfty}.

One option is to bootstrap regularity estimates in H\"older spaces (i.e. De Giorgi and Schauder estimates for parabolic integro-differential equations). This approach is based on the uniform ellipticity of the kernel $K_f(t,v,w)$: from Lemma \ref{l:coercive} and Proposition \ref{p:Linfty}, one has
\begin{equation*}
 c \vv^{\gamma+2s} |w|^{-d-2s} \leq K_f(t,v,w) \leq C \vv^{\gamma+2s} |w|^{-d-2s},
 \end{equation*}
Because the lower ellipticity bounds for $K_f(t,v,w)$ degenerate as $|v|\to \infty$, the regularity gained at each step is not uniform in $v$. To get around this, one would localize around a point $(t_0,v_0)$ and rescale the equation to obtain a new kernel with an ellipticity bound that holds locally uniformly.  A similar issue occurs for the true Boltzmann equation, but in a more challenging way, because the ellipticity degenerates at different rates depending on the direction of $w$ relative to $v$. The authors of \cite{imbert2020smooth} developed a change of variables to solve this issue, and iteratively applied H\"older estimates to obtain $C^\infty$ regularity in the inhomogeneous case.


A version of the approach in \cite{imbert2020smooth} would work for us, but because we are in the spatially homogeneous setting, there is a more self-contained proof that uses $L^2$-based energy estimates. The regularization is provided by the following coercivity property of the $Q_1$ term:

\begin{lemma}\label{l:good-sign}
For $f\in L^\infty_q(\R^d)$ with $q> \gamma+2s+d$, and $h\in L^2(\R^d)$, there holds
\[
\int_{\R^d} h Q_1(f,h) \dd v \leq  - N_{s,\gamma}^f(h)^2 + C \|h\|_{L^2(\R^d)}^2 \|f\|_{L^\infty_q(\R^d)},
\]
for a constant $C>0$ depending only on $d$, $\gamma$, $s$, and $q$, where
\begin{equation}\label{e:Nsgamma}
\begin{split}
N_{s,\gamma}^f(h) &:= \left( \int_{\R^d} \int_{\R^d} K_f(v,w) |h(v+w) - h(v)|^2 \dd w \dd v\right)^{1/2}\\
&= \left( c_{d,\gamma,s}\int_{\R^d} [f\ast|\cdot|^{\gamma+2s}]  \int_{\R^d}\frac{|h(v+w) - h(v)|^2}{|w|^{d+2s}} \dd w \dd v\right)^{1/2}.
\end{split}
\end{equation}
Furthermore, for any $h\in H^s(\R^d)$ and $f\in L^\infty_q(\R^d)$, there holds
\begin{equation}\label{e:Nbound}
N_{s,\gamma}^f(h) \leq \|f\|_{L^\infty_q(\R^d)}^{1/2} [h]_{H^{s}(\R^d)}, 
\end{equation}
and if $f$ additionally satisfies the hypotheses of Lemma \ref{l:coercive}, there holds
\begin{equation}\label{e:Nlower}
N_{s,\gamma}^f(h) \geq c_0 [h]_{H^s_{(\gamma+2s)/2}(\R^d)},
\end{equation}
with $c_0>0$ depending on the constants in Lemma \ref{l:coercive}.
\end{lemma}
Lemma \ref{l:good-sign} can be compared to the well-known {\it entropy dissipation} estimate satisfied by the Boltzmann collision operator \cite{alexandre2000entropy}.
\begin{proof}
From the formula \eqref{e:Q1} for $Q_1$, we have
\begin{equation*}
\begin{split}
\int_{\R^d}  h Q_1(f,h) \dd v &= c_{d,\gamma,s}\int_{\R^d}\int_{\R^d} h(v)\frac {h(v+w) -h(v)}{|w|^{d+2s}}[f\ast |\cdot|^{\gamma+2s}](v) \dd w \dd v\\
&= - c_{d,\gamma,s}\int_{\R^d}\int_{\R^d} \frac{|h(v+w) - h(v)|^2}{|w|^{d+2s}} [f\ast |\cdot|^{\gamma+2s}](v) \dd w \dd v\\
&\quad  + c_{d,\gamma,s}\int_{\R^d} \int_{\R^d} h(v+w) \frac{h(v+w) - h(v)}{|w|^{d+2s}} [f\ast |\cdot|^{\gamma+2s}](v) \dd w \dd v.
\end{split}
\end{equation*}
We observe that the first term on the right is equal to $-N_{s,\gamma}^f(h)^2$. In the second term, we change variables to swap $v \leftrightarrow v+w$, yielding
\begin{equation}\label{e:hQ1}
\begin{split}
\int_{\R^d}& h Q_1(f,h) \dd v\\
 &=  -N_{s,\gamma}^f(h)^2 + c_{d,\gamma,s}\int_{\R^d} \int_{\R^d} h(v)\frac{h(v) - h(v+w)}{|w|^{d+2s}} [f\ast |\cdot|^{\gamma+2s}](v+w) \dd w \dd v\\
&=   -N_{s,\gamma}^f(h)^2 -\int_{\R^d} h Q_1(f,h) \dd v\\
&\quad  + c_{d,\gamma,s}\int_{\R^d} \int_{\R^d} h(v)\frac{h(v) - h(v+w)}{|w|^{d+2s}} \left([f\ast |\cdot|^{\gamma+2s}](v+w) - [f \ast |\cdot|^{\gamma+2s}](v)\right) \dd w \dd v.
\end{split}
\end{equation}
The last integral on the right is equal to
\[
\begin{split}
I:=\int_{\R^d} \int_{\R^d} h(v)(h(v) - h(v+w)) \int_{\R^d} f(v-z) \frac { |z+w|^{\gamma+2s} -  |z|^{\gamma+2s}}{|w|^{d+2s}} \dd z\dd w \dd v. 
\end{split}
\]
The following analysis of the integral $I$ covers all cases $s\in (0,1)$. If we were only concerned with $s<\frac 1 2$, the proof would be easier since a first-order cancellation would be sufficient to handle the singularity as $w\to 0$. 

Using the change of variables $w\mapsto -w$, we obtain
\[
\begin{split}
2I&=  \int_{\R^d}\int_{\R^d} h(v) (2h(v) - h(v+w) - h(v-w)) \\
&\quad \quad\times\int_{\R^d} f(v-z) \frac{|z+w|^{\gamma+2s} + |z-w|^{\gamma+2s} - 2|z|^{\gamma+2s}}{|w|^{d+2s}} \dd z \dd w \dd v\\
&\leq \int_{\R^d} \left(\int_{\R^d} h^2(v)\dd v\right)^{1/2}\left(\int_{\R^d} (h(v) - h(v+w)-h(v-w))^2 \dd v\right)^{1/2}\\
&\quad \quad \times \sup_{v\in \R^d}\int_{\R^d} f(v-z)  \frac {\left| |z+w|^{\gamma+2s} +|z-w|^{\gamma+2s} -  2|z|^{\gamma+2s}\right|}{|w|^{d+2s}}\dd z \dd w\\
&\lesssim \|h\|_{L^2(\R^d)}^2 \|f\|_{L^\infty_q(\R^d)} \int_{\R^d} \sup_{v\in \R^d} \int_{\R^d} \langle v-z\rangle^{-q}\frac { \left| |z+w|^{\gamma+2s} +|z-w|^{\gamma+2s} -  2|z|^{\gamma+2s}\right|}{|w|^{d+2s}}\dd z \dd w.
\end{split}
\]
We divide this $(w,z)$ integral into three parts: $\{|w|\leq 1, |z|\geq 2|w|\}$,  $\{|w|\leq 1, |z|< 2|w|\}$, and $\{|w|>1\}$. For the first part, a Taylor expansion gives
\[
\begin{split}
\int_{\{|w|\leq 1\}} \sup_{v\in \R^d} \int_{\{|z|\geq 2|w|\}} &\langle v-z\rangle^{-q} \frac { \left| |z+w|^{\gamma+2s} +|z-w|^{\gamma+2s} -  2|z|^{\gamma+2s}\right|}{|w|^{d+2s}}\dd z \dd w \\
&\lesssim \int_{\{|w|\leq 1\}} \sup_{v\in \R^d} \int_{\R^d} \langle v-z\rangle^{-q} |z|^{\gamma+2s-2} |w|^{2-d-2s} \dd z \dd w\\
&\lesssim \int_{\{|w|\leq 1\}} |w|^{2-d-2s} \dd w \lesssim 1,
\end{split}
\]
since $|w|\leq |z|/2$ and $q> \gamma+2s+d$. Next, we have
\[
\begin{split}
\int_{\{|w|\leq 1\}}&\sup_{v\in \R^d} \int_{\{|z|< 2|w|\}} \langle v-z\rangle^{-q} \frac { \left| |z+w|^{\gamma+2s} +|z-w|^{\gamma+2s} -  2|z|^{\gamma+2s}\right|}{|w|^{d+2s}}\dd z \dd w \\
&\lesssim \int_{\{|w|\leq 1\}} |w|^{-d-2s} \int_{\{|z|< 2|w|\}} \left| |z+w|^{\gamma+2s} +|z-w|^{\gamma+2s} -  2|z|^{\gamma+2s}\right| \dd z \dd w\\
&\lesssim \int_{\{|w|\leq 1\}} |w|^{-d-2s} |w|^{\gamma+2s+d} \dd w \lesssim 1.
\end{split}
\]
For the remaining part,
\[
\begin{split}
\int_{\{|w| > 1\}} \sup_{v\in \R^d} &\int_{\R^d} \langle v-z\rangle^{-q} \frac { \left||z+w|^{\gamma+2s} + |z-w|^{\gamma+2s} - 2 |z|^{\gamma+2s}\right|}{|w|^{d+2s}}\dd z \dd w\\ 
&\leq \int_{\{|w|>1\}} |w|^{-d-2s}\sup_{v\in \R^d} \left(\int_{\R^d} \langle v-z\rangle^{-q}|z+w|^{\gamma+2s} \dd z \right.\\
&\qquad \left. + \int_{\R^d} \langle v-z\rangle^{-q} |z+w|^{\gamma+2s} \dd z + 2\int_{\R^d} \langle v-z\rangle^{-q} |z|^{\gamma+2s} \dd z\right) \dd w \\
&\lesssim \int_{\{|w|>1\}} |w|^{-d-2s}\sup_{v\in \R^d}\left( \langle v+w\rangle^{\gamma+2s} + \langle v-w\rangle^{\gamma+2s} +2 \vv^{\gamma+2s}\right) \dd w\\
&\lesssim 1.
\end{split}
\]
Returning to \eqref{e:hQ1} and collecting terms, we obtain the first statement of the lemma. 

The upper bound \eqref{e:Nbound} follows from the convolution estimate of Lemma \ref{l:convolution}(a) and the definition \eqref{e:Nsgamma} of $N_{s,\gamma}^f(h)$. The lower bound \eqref{e:Nlower} follows from Lemma \ref{l:coercive} and \eqref{e:Nsgamma}.
\end{proof}

We need the following interpolation lemma to trade decay for regularity. The proof is the same as \cite[Lemma 2.6]{HST2020boltzmann}.

\begin{lemma}\label{l:interpolation}
For any $q, m\geq 0$, suppose that $g \in L^\infty_m \cap H^k_q(\R^d)$ and $k' \in (0,k)$. Then if $\ell < (m-d/2)(1-k'/k) +q(k'/k)$, there holds
\[
\|g\|_{H^{k'}_\ell(\R^d)} \leq C \|g\|_{L^\infty_m(\R^d)}^{1-k'/k}\|g\|_{H^k_q(\R^d)}^{k'/k} \leq C\left( \|g\|_{L^\infty_m(\R^d)} + \|g\|_{H^k_q(\R^d)}\right).
 \]
The constant $C>0$ depends on $d$, $k$, $k'$, $q$, and $m$. 
%
\end{lemma}

Now we are ready to prove our main theorem:

\begin{proof}[Proof of Theorem \ref{t:main}]

By Theorem \ref{t:short-time-existence}, which we prove in Section \ref{s:short} below, a solution $f\geq 0$ exists on some time interval $[0,T]$. This solution can be extended by re-applying Theorem \ref{t:short-time-existence} for as long as the $H^{d+3}_q(\R^d)$ norm of $f(t,\cdot)$ remains finite, so we assume by contradiction that $f$ exists on some maximal time interval $[0,T_*)$ with $T_*\geq T$, and that 
\begin{equation}\label{e:contradiction-global}
\|f\|_{L^\infty([0,t],H^{d+3}_q(\R^d))} \nearrow \infty \quad \text{ as } t \to T_*.
\end{equation}
From estimate \eqref{e:energy-est-local-thm} in Theorem \ref{t:short-time-existence}, and our assumption on $f_{\rm in}$, we see that $f$ is $C^\infty$ with rapid decay on $[0,T_*) \times\R^d$ (but this smoothness could {\it a priori} degenerate as $t\to T_*$). We can therefore apply our estimates from earlier in the paper:
\begin{itemize}

\item Conservation of mass on $[0,T_*)$ follows from the formal argument in Section \ref{s:conservation}, which is rigorously valid because $f$ is smooth and rapidly decaying. 

\item Theorem \ref{t:L2} implies  the $L^2(\R^d)$ norm of $f(t)$, for any $t\in [0,T_*)$, is bounded by its value at $t=0$.

\item Lemma \ref{l:energy} provides a bound on the energy $E(t) = \int_{\R^d} |v|^2 f(t,v) \dd v$ for  $t \in [0,T_*)$, that depends on $d$, $\gamma$, $s$, $q$, the initial data, and $T_*$, and is finite since $T_*<\infty$.

\item Because of the mass, energy, and $L^2$ bounds, we can apply the $L^\infty$ estimate of Proposition \ref{p:Linfty}, which is uniform on $[T_*/2,T_*)\times\R^d$. On the remaining time interval $[0,T_*/2)$, the $H^{d+3}_q(\R^d)$ norm of $f(t,\cdot)$ is uniformly bounded as a result of \eqref{e:contradiction-global}, so by Sobolev embedding, $f$ is bounded in $L^\infty(\R^d)$ for $t\in [0,T_*/2)$ as well.

\item With the $L^\infty$ bound from the previous bullet point, the decay estimates of Proposition \ref{p:decay} apply: 
for each $q>0$, there is a constant $C_q$ with
\begin{equation}\label{e:decay-global}
\|f(t)\|_{L^\infty_q(\R^d)} \leq C_q, \quad t\in [0,T_*).
\end{equation}
with $C_q>0$ depending on $q$, $d$, $\gamma$, $s$, the initial data, and the energy bound.
\end{itemize}

For the remainder of this proof, we call a constant {\it universal} if it depends only on $d$, $\gamma$, $s$, the initial data, and the energy bound on $[0,T_*)$. Quantities that depend only on universal constants will be absorbed into inequalities via the $\lesssim$ symbol, sometimes without comment. 

Our goal is to show that all Sobolev norms of $f$ in $v$ remain finite on $[0,T_*)$. We prove this by induction. For the base case, we multiply the equation \eqref{e:homogeneous} by $f$ and integrate over $\R^d$ to obtain
\[
\frac 1 2 \frac d {dt} \int_{\R^d} f^2 \dd v = \int_{\R^d} f Q_1(f,f) \dd v + \int_{\R^d} f Q_2(f,f) \dd v.
\]
Using the coercivity estimate of Lemma \ref{l:good-sign} for the $Q_1$ term, we have, for $q>\gamma+2s+d$, 
\[ 
\begin{split}
\frac 1 2 \frac d {dt} \int_{\R^d} f^2 \dd v &\leq -N_{s,\gamma}^f(f)^2 + C\|f\|_{L^2(\R^d)}^2 \|f\|_{L^\infty_q(\R^d)} + C\int_{\R^d} f^2 [f\ast |\cdot|^\gamma] \dd v\\
&\leq - c [f]_{H^s_{(\gamma+2s)/2}(\R^d)}^2 + C\|f\|_{L^2(\R^d)}^2 \|f\|_{L^\infty_q(\R^d)} + C\|f\|_{L^2(\R^d)}^2 \|f\|_{L^\infty_m(\R^d)},
\end{split}
\]
where we used the lower bound \eqref{e:Nlower} for $N_{s,\gamma}^f$, and Lemma \ref{l:convolution}(a) with $m>\gamma+d$ in the last term. Recalling that $\frac d {dt} \int_{\R^d} f^2 \dd v = 0$, we now have, using \eqref{e:decay-global},
\[
\int_0^{T_*}[f(t)]_{H^s_{(\gamma+2s)/2}(\R^d)}^2 \dd t \leq C,
\]
for some universal constant $C$.

Now, we assume by induction that for some $k\geq 1$, 
\begin{equation}\label{e:induction}
\|f(t)\|_{H^{k-1}(\R^d)} \leq C_k, \quad t\in [0,T_*],
\end{equation}
and that
\begin{equation}\label{e:induction2}
\|F_{k-1}(t)\|_{L^2([0,T_*))}^2 := \int_0^{T_*} \|f(t)\|_{H^{k-1+s}_{(\gamma+2s)/2}(\R^d)}^2 \dd t \leq C_k,
\end{equation}
for some $C_k$ depending only on universal constants and $k$. Letting $\beta = (\beta_1,\ldots, \beta_d)$ be a multi-index in the $v$ variable of order $k$, we differentiate the equation \eqref{e:homogeneous} by $\partial^\beta$ and integrate against $\partial^\beta f$.  It is easy to see that the operator $Q$ satisfies a Liebnitz-type rule $\partial_{v_i}Q(F,G) = Q(\partial_{v_i} F, G) + Q(F, \partial_{v_i}G)$. We then obtain
\begin{equation}\label{e:energy-est-global}
\begin{split}
\frac 1 2 \frac d {dt} \int_{\R^{d}}  (\partial^\beta f)^2 \dd v
&= \int_{\R^{d}} \partial^\beta f \partial^\beta Q(f,f) \dd v \\
&=   \sum_{\beta' + \beta'' = \beta} \left( \int_{\R^d}  \partial^\beta f Q_1(\partial^{\beta'} f, \partial^{\beta''}f) \dd v + \int_{\R^d}  \partial^\beta f Q_2(\partial^{\beta'} f, \partial^{\beta''}f) \dd v\right).
\end{split}
\end{equation}
Starting with the terms involving $Q_1$, there are three cases:

If $|\beta'| = 0$ and $\beta'' = \beta$, then we use the coercivity estimate of Lemma \ref{l:good-sign} and obtain, for any $q>\gamma+2s+d$,
\begin{equation}\label{e:c0term}
\begin{split}
\int_{\R^d} \partial^\beta f Q_1(f, \partial^\beta f) \dd v &\leq -N_{s,\gamma}^f (\partial^\beta f)^2 + C \|\partial^\beta f\|_{L^2(\R^d)}^2 \|f\|_{L^\infty_q(\R^d)}\\
&\leq -c_0 [\partial^\beta f]_{H^s_{(\gamma+2s)/2}(\R^d)}^2 + C\|f\|_{H^k(\R^d)}^2,
\end{split}
\end{equation}
by \eqref{e:Nlower} and the decay estimate \eqref{e:decay-global}. This negative term will absorb the other highest-order terms that arise in the energy estimates.

If $1\leq |\beta'| \leq k-1$, then we have, using Lemma \ref{l:convolution}(b) (since $\gamma + 2s > -d/2$, by our assumption $\gamma \geq -\frac{d+4s} 3$),
\[
\begin{split}
\int_{\R^d} \partial^\beta f Q_1(\partial^{\beta'}f, \partial^{\beta''} f) \dd v &\approx \int_{\R^d} \partial^\beta f [\partial^{\beta'}f\ast |\cdot|^{\gamma+2s}] (-\Delta)^s \partial^{\beta''} f \dd v\\
&\lesssim \|\partial^\beta f \|_{L^2(\R^d)} \|\partial^{\beta'} f\ast |\cdot|^{\gamma+2s-1}\|_{L^\infty(\R^d)} \|\partial^{\beta''} f\|_{H^{2s}(\R^d)}\\
&\lesssim \|f\|_{H^k(\R^d)} \|\partial^{\beta'} f\|_{L^2_m(\R^d)} \|f\|_{H^{k-1+2s}(\R^d)},
\end{split}
\]
for some $m>d/2+\gamma+2s$. We also used $|\beta''|\leq k-1$. In the last expression,  we use the interpolation estimate of Lemma \ref{l:interpolation} for the middle factor (since $|\beta'| \leq k-1$) and obtain
\[
\begin{split}
\|\partial^{\beta'} f \|_{L^2_m(\R^d)} &= \|\vv^{m-(\gamma+2s)/2} \partial^{\beta'} f \|_{L^2_{(\gamma+2s)/2}(\R^d)}\\
&\lesssim \|f\|_{H^{k-1+s}_{(\gamma+2s)/2}(\R^d)} +  \|f\|_{L^\infty_{m+1-(\gamma+2s)/2}(\R^d)}\lesssim \|f\|_{H^{k-1+s}_{(\gamma+2s)/2}(\R^d)} + 1,
\end{split}
\]
by the decay estimate \eqref{e:decay-global}. Next, we use a similar interpolation to write $\|f\|_{H^{k-1+2s}(\R^d)}\lesssim \|f\|_{H^{k+s}_{(\gamma+2s)/2}(\R^d)} + \|f\|_{L^\infty_{1-(\gamma+2s)/2}(\R^d)} \lesssim \|f\|_{H^{k+s}_{(\gamma+2s)/2}(\R^d)} +1$, giving 
\[
\begin{split}
\int_{\R^d} \partial^\beta f Q_1(\partial^{\beta'}f, \partial^{\beta''} f) \dd v &\lesssim \|f\|_{H^k(\R^d)}\left( \|f\|_{H^{k-1+s}_{(\gamma+2s)/2}(\R^d)} +1\right) \left(\|f\|_{H^{k+s}_{(\gamma+2s)/2}(\R^d)} + 1\right)\\
&\leq c_0^{-1}C \|f\|_{H^k(\R^d)}^2 \left( \|f\|_{H^{k-1+s}_{(\gamma+2s)/2}(\R^d)} + 1\right)^2+ \frac{c_0}{4d^k}\|f\|_{H^{k+s}_{(\gamma+2s)/2}(\R^d)}^2 + C\\
&\leq c_0^{-1}C\left( 1 + F_{k-1}^2(t) \right)\|f\|_{H^k(\R^d)}^2 + \frac {c_0} {4d^k} \|f\|_{H^{k+s}_{(\gamma+2s)/2}(\R^d)}^2 + C.
\end{split}
\]
by Young's inequality. Here, $c_0$ is the constant appearing in \eqref{e:c0term}, $F_{k-1}(t)$ is defined in \eqref{e:induction2}, and $C$ is a universal constant.

If $|\beta'| = k$ and $|\beta''| = 0$, then we need to proceed in two subcases, $k=1$ and $k\geq 2$. When $k=1$, we have $\partial^\beta f= \partial_{v_i} f$ for some $i$. We transfer this derivative from $f$ to the convolution kernel $|v|^{\gamma+2s}$ via $(\partial_{v_i} f)\ast |\cdot|^{\gamma+2s} = f \ast (\partial_{v_i} |\cdot|^{\gamma+2s})$. The new kernel still has an integrable singularity because $\gamma+2s-1> -d$ by our assumption $\gamma \geq -\frac{d+4s} 3$. We then have
\[
\begin{split}
\int_{\R^d} \partial_{v_i} f Q_1(\partial_{v_i} f, f) \dd v &\approx \int_{\R^d} \partial_{v_i} f [\partial_{v_i}f\ast |\cdot|^{\gamma+2s}] (-\Delta)^s f \dd v\\
&\lesssim \int_{\R^d} \partial_{v_i}f [f \ast |\cdot|^{\gamma+2s-1}](-\Delta)^s f \dd v\\
&\lesssim \|\partial_{v_i} f\|_{L^2(\R^d)} \|f\ast |\cdot|^{\gamma+2s-1}\|_{L^\infty(\R^d)} \|f\|_{H^{2s}(\R^d)}.
\end{split}
\]
We estimate the convolution with Lemma \ref{l:convolution}(a) and \eqref{e:decay-global}, and absorb this term into the constant. We estimate the $H^{2s}$ norm with an interpolation similar to above:
\[
\|f\|_{H^{2s}(\R^d)} \lesssim \|f\|_{H^{1+s}_{(\gamma+2s)/2}(\R^d)} + 1,
\]
by \eqref{e:decay-global} again. We now have
\[
\begin{split}
\int_{\R^d} \partial_{v_i} f Q_1(\partial_{v_i} f, f) \dd v &\lesssim \|f\|_{H^1(\R^d)} \|f\|_{H^{1+s}(\R^d)} \leq c_0^{-1} C \|f\|_{H^1(\R^d)}^2 + \frac {c_0} {4 d} \|f\|_{H^{1+s}(\R^d)}^2.
\end{split}
\]
When $k\geq 2$, we use Lemma \ref{l:convolution}(b) with $m>\gamma+2s+d$ to estimate the convolution:
\[
\begin{split}
\int_{\R^d} \partial^\beta f Q_1(\partial^{\beta}f,  f) \dd v &\approx \int_{\R^d} \partial^\beta f [\partial^\beta f\ast |\cdot|^{\gamma+2s}] (-\Delta)^s f \dd v\\
&\lesssim \|\partial^\beta f \|_{L^2(\R^d)} \|\partial^\beta f\ast |\cdot|^{\gamma+2s}\|_{L^\infty(\R^d)} \|f\|_{H^{2s}(\R^d)}\\
&\lesssim \|f\|_{H^k(\R^d)} \|\partial^\beta f\|_{L^2_m(\R^d)} \|f\|_{H^{2s}(\R^d)}.
\end{split}
\]
We apply interpolation to the second and third factors, in a similar manner to above, and obtain
\[
\begin{split}
\int_{\R^d} \partial^\beta f Q_1(\partial^{\beta}f,  f) \dd v &\lesssim \|f\|_{H^k(\R^d)} \left( \|f\|_{H^{k+s}_{(\gamma+2s)/2}(\R^d)} + 1\right) \left( \|f\|_{H^{k-1+s}(\R^d)}+1\right)\\
&\leq c_0^{-1} C \|f\|_{H^k(\R^d)}^2\left( F_{k-1}^2(t) + 1\right) + \frac {c_0} {4d^k} \|f\|_{H^{k+s}_{(\gamma+2s)/2}(\R^d)}^2 + C.
\end{split}
\]

For the terms involving $Q_2$ in \eqref{e:energy-est-global}, we write
\[
\int_{\R^d}  \partial^\beta f Q_2(\partial^{\beta'} f, \partial^{\beta''}f) \dd v \approx \int_{\R^d} \partial^\beta f \partial^{\beta''} f [\partial^{\beta'} f \ast |\cdot|^\gamma] \dd v.
\]
Again, there are three cases.

If $|\beta'| = k$ and $|\beta''| = 0$, we pick a constant $\theta > d/2 + \gamma$ and use Lemma \ref{l:young} as well as the weighted $L^\infty$ bounds \eqref{e:decay-global} to obtain
\[
\begin{split}
 \int_{\R^d} \partial^\beta f  \vv^\theta f \vv^{-\theta}[\partial^{\beta} f \ast |\cdot|^\gamma] \dd v
&\leq \|\partial^\beta f\|_{L^2(\R^d)} \|f\|_{L^\infty_\theta(\R^d)} \|\partial^\beta f \ast |\cdot|^\gamma\|_{L^2_{-\theta}(\R^d)}\\
&\leq C \|\partial^\beta f\|_{L^2(\R^d)} \|\partial^\beta f\|_{L^2_{d/2+\eps}(\R^d)},
\end{split}
\]
for some small $\eps>0$. By interpolation (Lemma \ref{l:interpolation}), the decay estimate \eqref{e:decay-global}, and Young's inequality, this expression is bounded above by
\[
C \|f\|_{H^k(\R^d)}\left( \|f\|_{H^{k+s}(\R^d)}+1\right) \leq c_0^{-1} C \left(\|f\|_{H^k(\R^d)}^2 + 1\right) + \frac {c_0} {4d^k} \|f\|_{H^{k+s}(\R^d)}^2.
\]

Next, if $1\leq |\beta'| \leq k-1$, we have $|\beta''|\leq k-1$, and
\begin{equation}\label{e:Q2-global}
\begin{split}
\int_{\R^d} \partial^\beta f \partial^{\beta''} f [\partial^{\beta'} f \ast |\cdot|^\gamma] \dd v
&\leq \|\partial^\beta f\|_{L^2(\R^d)} \| f\|_{H^{k-1}(\R^d)} \|\partial^{\beta'} f \ast |\cdot|^\gamma\|_{L^\infty(\R^d)}\\
&\leq  C_q\|f\|_{H^k} \|\partial^{\beta'} f\ast |\cdot|^\gamma\|_{L^\infty(\R^d)},
\end{split}
\end{equation}
by the inductive hypothesis \eqref{e:induction}. Since $\gamma$ may be too negative to apply Lemma \ref{l:convolution}(b), we transfer a derivative from $|v|^\gamma$ to $\partial^{\beta'} f$. In more detail, using the identity
\[
|v|^\gamma = \frac 1 {d+\gamma} \nabla\cdot (v|v|^\gamma),
\]
we write
\[
\|\partial^{\beta'} f \ast |\cdot|^\gamma\|_{L^\infty(\R^d)} \lesssim  \sum_{i=1}^d \|(\partial_{v_i} \partial^{\beta'} f) \ast |\cdot|^{\gamma+1}\|_{L^\infty(\R^d)}.
\]
From our assumption that $-\frac{d+4s} 3 \leq \gamma < -2$, we have
\[
-\frac d 2 < \gamma+1 < 0,
\]
so we can apply Lemma \ref{l:convolution}(b) with $\mu = \gamma+1$ and $m>\gamma+1+\frac d 2$, and obtain
\[
\|\partial^{\beta'} f \ast |\cdot|^\gamma\|_{L^\infty(\R^d)} \lesssim \sum_{i=1}^d \|(\partial_{v_i}\partial^{\beta'} f) \ast |\cdot|^\gamma\|_{L^\infty(\R^d)} \lesssim \sum_{i=1}^d \| \partial_{v_i} \partial^{\beta'} f\|_{L^2_m(\R^d)}\lesssim \|f\|_{H^k_m(\R^d)}.
\]
As above, we apply interpolation (Lemma \ref{l:interpolation}) plus the decay estimate of \eqref{e:decay-global} to bound $\|f\|_{H^k_m(\R^d)}$ from above by $\lesssim \|f\|_{H^{k+s}_{(\gamma+2s)/2}(\R^d)} + 1$. Returning to \eqref{e:Q2-global}, we now have
\[
\begin{split}
\int_{\R^d} \partial^\beta f Q_2(\partial^{\beta'} f, \partial^{\beta''} f) \dd v &\leq C_k \|f\|_{H^k(\R^d)} \left(\|f\|_{H^{k+s}_{(\gamma+2s)/2}(\R^d)}+1\right)\\
&\leq \frac {C_k}{c_0} \left(\|f\|_{H^k(\R^d)}^2+1\right) + \frac{c_0}{4d^k}\|f\|_{H^{k+s}_{(\gamma+2s)/2}(\R^d)}^2.
\end{split}
\]

Finally, if $|\beta'| = 0$ and $\beta'' = \beta$, we have for any $q>\gamma+d$,
\[
\begin{split}
\int_{\R^d} \partial^\beta f Q_2(f, \partial^{\beta} f) \dd v &\approx \int_{\R^d} [f\ast |\cdot|^\gamma] (\partial^{\beta} f)^2 \dd v\\
& \leq \|f\ast |\cdot|^\gamma\|_{L^\infty(\R^d)} \|f\|_{H^k(\R^d)}^2\\
&\leq C\|f\|_{L^\infty_q(\R^d)} \|f\|_{H^k(\R^d)}^2\\
&\leq C\|f\|_{H^k(\R^d)}^2,
\end{split}
\]
by Lemma \ref{l:convolution}(a) and \eqref{e:decay-global}. 

Now we collect our upper bounds for the terms in \eqref{e:energy-est-global} and sum over all $d^k$ multi-indices $\beta$ with $|\beta| = k$, to obtain
\[
\begin{split}
\frac d {dt} \sum_{|\beta|=k}\|\partial^\beta f\|_{L^2(\R^d)}^2 &\leq - c_0 \sum_{|\beta| = k}  [\partial^\beta f]_{H^s_{(\gamma+2s)/2}(\R^d)}^2 + \frac {c_0} {2} \|f\|_{H^{k+s}_{(\gamma+2s)/2}(\R^d)}^2\\
&\quad  + C(1 + F_{k-1}^2(t))\|f\|_{H^k(\R^d)}^2 + C ,
\end{split}
\]
for a universal constant $C$. Noting that $\|f\|_{H^k(\R^d)}^2 = \sum_{|\beta|=k} \|\partial^\beta f\|_{L^2(\R^d)}^2 + \|f\|_{H^{k-1}(\R^d)}^2$ and  $\|f\|_{H^{k+s}_{(\gamma+2s)/2}(\R^d)}^2 = \sum_{|\beta| = k} [\partial^\beta f]_{H^s_{(\gamma+2s)/2}(\R^d)}^2 + \|f\|_{H^k_{(\gamma+2s)/2}(\R^d)}^2$, our inductive hypothesis \eqref{e:induction} implies
\begin{equation}\label{e:before-gronwall}
\begin{split}
\frac d {dt} \sum_{|\beta|=k}\|\partial^\beta f\|_{L^2(\R^d)}^2 &\leq - \frac{c_0} 2 \sum_{|\beta| = k}  [\partial^\beta f]_{H^s_{(\gamma+2s)/2}(\R^d)}^2   + C(1 + F_{k-1}^2(t))\sum_{|\beta|=k} \|\partial^\beta f\|_{L^2(\R^d)}^2 + C ,
\end{split}
\end{equation}
since $(\gamma+2s)/2<0$. Our second inductive hypothesis \eqref{e:induction2} implies $F_{k-1}^2$ is integrable, so with Gr\"onwall's inequality, we now have
\[
\sum_{|\beta|=k} \|\partial^\beta f\|_{L^2(\R^d)} \leq C \exp\left(C\int_0^t (1+F_{k-1}^2(t'))\dd t'\right) \leq C_k, \quad t\in [0,T_*),
\]
for some $C_k$ depending on universal constants and $k$. Using this bound in \eqref{e:before-gronwall} and integrating from $0$ to $T_*$, we find
\[
\frac {c_0} 2 \int_0^{T_*} \sum_{|\beta|=k} [\partial^\beta f]_{H^s_{(\gamma+2s)/2}(\R^d)}^2 \leq C_k,
\]
which allows us to close the induction, since $F_k(t) = \|f(t)\|_{H^{k+s}(\R^d)} \lesssim \|f\|_{H^k(\R^d)} + \sum_{|\beta|=k} [\partial^\beta f]_{H^s(\R^d)}$. 

We have shown that all Sobolev norms $\|f(t)\|_{H^k(\R^d)}$ are uniformly bounded in $t\in [0,T_*)$. By interpolation (Lemma \ref{l:interpolation}), the weighted $H^k_q(\R^d)$ norms of $f$ are also bounded on $[0,T_*)$, for all $k,q>0$.  In particular, we can apply short-time existence (Theorem \ref{t:short-time-existence}) with initial data $f(T_*,\cdot)$ to extend our solution to some time interval $[0,T']$ with $T'>T_*$, contradicting \eqref{e:contradiction-global}. This completes the proof.
\end{proof}

\section{Local existence}\label{s:short}


This section is devoted to the proof of local existence, Theorems \ref{t:short-time-existence-inhom} and \ref{t:short-time-existence}. The proof uses $L^2$-based energy estimates, but unlike the proof of Theorem \ref{t:main} above, there is no need to obtain an estimate that persists for large $t$, which leads to some simplifications. On the other hand, extra complications arise from the broader range of $\gamma$ and $s$, the dependence on $x$, and the need to commute polynomial weights of the form $\vv^q$. 

The ideas in this section owe a lot to short-time existence results for the Boltzmann equation such as \cite{amuxy2010regularizing, amuxy2011bounded, amuxy2013mild, morimoto2015polynomial, HST2020boltzmann} and especially \cite{henderson2022existence}, which addressed the case of $\gamma$ very close to $-d$. 

We begin with a commutator estimate for polynomially-growing weights:
\begin{lemma}\label{l:commutator}
If $m,q > \gamma+2s+d$, then
\[
\begin{split}
\| \vv^q Q(f,g) - Q(f,\vv^q g)\|_{L^2(\R^d)} \leq C \left( \|f\|_{L^\infty_m(\R^d)}^{1/2} N_{s,\gamma}^f(\vv^q g) + \|f\|_{L^\infty_m(\R^d)} \|g\|_{L^2_{q-2s}(\R^d)}\right).
\end{split}
\]
for a constant $C>0$ depending only on $d$, $\gamma$, $s$, $m$, and $q$.
\end{lemma}

\begin{proof}
It is clear that $\vv^q Q_2(f,g) - Q_2(f,\vv^q g) = 0$, giving 
\[
\begin{split}
\vv^q Q(f,g) - Q(f,\vv^q g) &= \vv^q Q_1(f,g) - Q_1(f,\vv^q g)\\
&= \int_{\R^d} K_f(v,w) g(v+w) [\vv^q - \langle v+w\rangle^q] \dd w\\
&= \int_{\R^d} K_f(v,w) g(v+w)\langle v+w\rangle^q [ \vv^q \langle v+w\rangle^{-q} - 1] \dd w\\
&= \int_{\R^d} K_f(v,w) [g(v+w) \langle v+w\rangle^q - g(v) \vv^q] [ \vv^q \langle v+w\rangle^{-q} - 1] \dd w\\
&\quad + g(v)\vv^q \int_{\R^d} K_f(v,w) [ \vv^q \langle v+w\rangle^{-q} - 1]\dd w\\
&= I_1 + I_2.
\end{split}
\]
This decomposition is similar to one that appears in the proof of \cite[Lemma 2.12]{silvestre2022nearequilibrium}. Taking the $L^2$ norm in $v$ of both terms $I_1$ and $I_2$, we first have, using Cauchy-Schwarz in $w$,
\[
\begin{split}
\int_{\R^d}I_1^2 \dd v &\leq \int_{\R^d} \left(\int_{\R^d}K_f(v,w)  |g(v+w)\langle v+w\rangle^q - g(v) \vv^q|^2 \dd w \right)\\
&\qquad\times \left(\int_{\R^d}  K_f(v,w) |\vv^q \langle v+w\rangle^{-q} - 1|^2 \dd w\right) \dd v.
\end{split}
\]
Since Lemmas \ref{l:convolution}(a) and \ref{l:q-power} imply
\[
\begin{split}
\vv^{2q}  \int_{\R^d} K_f(v,w)|\langle v+w\rangle^{-q} - \vv^{-q}|^2 \dd w &\approx \vv^{2q} [ f \ast |\cdot|^{\gamma+2s}] \int_{\R^d} \frac{|\langle v+w\rangle^{-q} - \vv^{-q}|^2 }{|w|^{d+2s}} \dd w \\
&\lesssim \|f\|_{L^\infty_m(\R^d)}\vv^{2q} \vv^{-2q-2s} \lesssim \|f\|_{L^\infty_m(\R^d)}, 
\end{split}
\]
we have
\[
\int_{\R^d} I_1^2 \dd v \lesssim \|f\|_{L^\infty_m(\R^d)}\int_{\R^d}\int_{\R^d} K_f(v,w)  |g(v+w)\langle v+w\rangle^q - g(v) \vv^q|^2 \dd w \dd v =  \|f\|_{L^\infty_m(\R^d)} N_{s,\gamma}^f\left(\langle \cdot\rangle^q g\right)^2.
\]
For $I_2$, we use Lemma \ref{l:q-power} to write
\[
|I_2| \leq |g(v)| \vv^{2q} [f\ast |\cdot|^{\gamma+2s}](v) (-\Delta)^s(\langle v\rangle^{-q}) \leq |g(v)| \vv^{q-2s} [f\ast|\cdot|^{\gamma+2s}](v) \lesssim \|f\|_{L^\infty_m} |g(v)| \vv^{q-2s},
\]
which gives 
\[
\int_{\R^d} I_2^2 \dd v \lesssim \|f\|_{L^\infty_m(\R^d)}^2 \|g\|_{L^2_{q-2s}(\R^d)}^2.
\]
Combining our estimates for $I_1$ and $I_2$, the proof is complete.
\end{proof}

The key lemma for local existence is the following energy estimate for a modified linear version of our equation. The purpose of the parameter $\sigma$ is to interpolate between the linear isotropic Boltzmann equation and the heat equation on $\R^{1+6}$, so that we can apply the method of continuity below in the proof of Lemma \ref{l:linear-existence}.

\begin{lemma}\label{l:energy-estimate}
Let $k\geq 2d+2$, $q \geq m > \gamma+2s+d$, $T>0$, and let $f\in L^\infty([0,T],H^k_m(\T^d\times\R^d))$ be fixed. For any $\sigma \in [0,1]$ and $R\in L^2([0,T], L^2_q(\T^d\times\R^d))$, let $g_\sigma$ be a solution of
\begin{equation}\label{e:sigma}
\partial_t g_\sigma +\sigma v\cdot\nabla_x g_\sigma = \sigma Q(f, g_\sigma) + (1-\sigma) \Delta_{x,v} g_\sigma + R,
\end{equation}
on the time interval $[0,T]$. Then the following estimate holds:
\[
\begin{split}
\|g_\sigma(t)\|_{H^k_q(\T^d \times \R^{d})} \leq &\left(\|g_\sigma(0)\|_{H^k_q(\T^d\times \R^d)} + C\int_0^t \|R(t')\|_{L^2_q(\T^d\times\R^d)}^2\dd t'\right)\\
&\qquad \times \exp\left( C\int_0^t (1+ \|f(t')\|_{H^k_m(\T^d\times\R^{d})}) \dd t' \right), \quad 0\leq t\leq T,
\end{split}
\]
where $C$ depends only on $k$, $q$, $m$, $d$, $\gamma$, and $s$. In particular, $C$ is independent of $\sigma$. 
\end{lemma}
\begin{proof}
For every multi-index $\beta = (\beta_1,\ldots,\beta_{2d})$ in $(x,v)$ variables with $|\beta| = \sum_i \beta_i = k$, we differentiate \eqref{e:sigma} by $\partial^\beta$, multiply by $\vv^{2q}\partial^\beta g_\sigma$, and integrate over $\T^d_x\times\R^{d}_v$ to obtain
\begin{equation}\label{e:energy-est}
\begin{split}
\frac 1 2 \frac d {dt} \int_{\T^d \times \R^{d}} &\vv^{2q} (\partial^\beta g_\sigma)^2 \dd v \dd x 
 + \sigma \int_{\T^d\times \R^d} \vv^{2q} \partial^\beta g_\sigma \partial^\beta[v\cdot \nabla_x g_\sigma] \dd v \dd x \\
&= \sigma \int_{\T^d \times \R^{d}} \vv^{2q} \partial^\beta g_\sigma \partial^\beta Q(f,g_\sigma) \dd v \dd x + (1-\sigma)\int_{\T^d \times \R^{d}} \vv^{2q} \partial^\beta g_\sigma \Delta_{x,v} \partial^\beta g_\sigma \dd v \dd x\\
&\quad + \int_{\T^d\times\R^d} \vv^{2q} \partial^\beta g_\sigma R \dd v \dd x.
\end{split}
\end{equation}
The last two terms are handled easily:
\[
\int_{\T^d\times\R^d} \vv^{2q} \partial^\beta g_\sigma R \dd v \dd x \leq \|g_\sigma\|_{H^k_q(\T^d\times\R^d)} \|R\|_{L^2_q(\T^d\times\R^d)}\lesssim \|g_{\sigma}\|_{H^k_q(\T^d\times\R^d)}^2 + \|R\|_{L^2_q(\T^d\times\R^d)}^2,
\]
and
\begin{equation}\label{e:last-term}
\begin{split}
\int_{\T^d\times\R^d} \vv^{2q} \partial^\beta g_\sigma \Delta_{x,v} \partial^\beta g_\sigma \dd v \dd x &= -\int_{\T^d\times\R^d} \vv^{2q} |\nabla_{x,v} \partial^\beta g_\sigma|^2 \dd v \dd x\\
&\quad - 2q \int_{\T^d\times\R^d} \partial^\beta g_\sigma \vv^{2q-2} v\cdot \nabla_v \partial^\beta g_\sigma \dd v \dd x\\
&\leq -q \int_{\T^d\times\R^d} \vv^{2q-2} v\cdot \nabla_v (\partial^\beta g_\sigma)^2 \dd v \dd x\\
&= C_{d,q} \int_{\T^d\times\R^d} \vv^{2q-2} (\partial^\beta g_{\sigma})^2 \dd v \dd x\\
&\leq C_{d,q} \|g_\sigma\|_{H^k_q(\T^d\times\R^d)}^2,
\end{split}
\end{equation}
since $|\beta|\leq k$. 

For the second term on the left in \eqref{e:energy-est}, we write $\partial^\beta = \partial^\beta_x\partial^\beta_v$, for multi-indices $\beta_x, \beta_v \in \N^d$. With the notation $\beta_v = (\beta_{v,1},\ldots,\beta_{v,d})$ and $e_i = (0,\ldots,1,\ldots,0)$, we have $\partial^\beta[v\cdot\nabla_x g_\sigma] = v\cdot \nabla_x\partial^\beta g_\sigma + \sum_{i=1}^d \beta_{v,i} \partial_{x_i} \partial^{\beta_x} \partial^{\beta_v - e_i} g_\sigma$, which implies
\[
\begin{split}
\left|\int_{\T^d\times \R^d} \vv^{2q} \partial^\beta g_\sigma \partial^\beta[v\cdot\nabla_x g_\sigma] \dd v \dd x \right| &= \left|\int_{\T^d\times \R^d} \vv^{2q} \partial^\beta g_\sigma v\cdot\nabla_x \partial^\beta g_\sigma \dd v \dd x\right.\\
&\quad \left.+ \sum_{i=1}^d \beta_{v,i} \int_{\T^d\times \R^d} \vv^{2q} \partial^\beta g_\sigma \partial_{x_i}\partial^{\beta_x}\partial^{\beta_v - e_i} g_\sigma \dd v \dd x\right|.
\end{split}
\]
The first term on the right is zero after integrating by parts in $x$. For the second term, Cauchy-Schwarz gives an upper bound of
\[
\lesssim \sum_{i=1}^d \|\vv^q \partial^\beta g_\sigma \|_{L^2(\T^d\times\R^d)} \|\vv^q \partial_{x_i}\partial^{\beta_x}\partial^{\beta_v - e_i} g_\sigma\|_{L^2(\T^d\times\R^d)} \lesssim d \|g_\sigma\|_{H^k_q(\T^d\times\R^d)}^2.
\]

It remains to estimate the first integral on the right in \eqref{e:energy-est}. As above, we use the Liebnitz rule $\partial Q(F,G) = Q(\partial F, G) + Q(F, \partial G)$ for $x$ or $v$ derivatives to write
\begin{equation*}
 \int_{\T^d \times \R^{d}} \vv^{2q} \partial^\beta g_\sigma \partial^\beta Q(f,g_\sigma) \dd v \dd x = \sum_{\beta' + \beta'' = \beta} \int_{\T^d\times \R^d} \vv^{2q} \partial^\beta g_\sigma Q(\partial^{\beta'} f, \partial^{\beta''} g_\sigma) \dd v \dd x.
\end{equation*}
For each pair $\beta', \beta''$, we write



\begin{equation}\label{e:J123}
\begin{split}
&\int_{\T^d \times\R^d} \vv^{2q}  \partial^\beta g_\sigma Q(\partial^{\beta'} f, \partial^{\beta''} g_\sigma) \dd v \dd x = J_1 + J_2 + J_3,
\end{split}
\end{equation}
with
\[
\begin{split}
J_1 &= 
\int_{\T^d \times\R^d} \vv^q \partial^\beta g_\sigma Q_1(\partial^{\beta'} f,\vv^q\partial^{\beta''} g_\sigma) \dd v \dd x,\\
J_2 &= 
 \int_{\T^d \times\R^d} \vv^q \partial^\beta g_\sigma Q_2(\partial^{\beta'} f,\vv^q\partial^{\beta''} g_\sigma) \dd v \dd x,\\
J_3 &= 
 \int_{\T^d \times\R^d} \vv^q \partial^\beta g_\sigma \left( \vv^q Q(\partial^{\beta'} f, \partial^{\beta''} g_\sigma) - Q(\partial^{\beta'} f, \vv^q \partial^{\beta''} g_\sigma\right) \dd v \dd x.
\end{split}
\]
The analysis of all three terms splits into cases depending on how the derivatives fall. 

We begin with the $Q_2$ terms because they are the simplest. Since $k\geq 2d+2$ and $|\beta'|+|\beta''|\leq k$, we must have either $|\beta'| \leq k - d - 1$ or $|\beta''|\leq k - d -2$. If $|\beta'|\leq k - d-1$, then Cauchy-Schwarz and the convolution estimate of Lemma \ref{l:convolution}(a) (since $m>\gamma +d$) yield
\[
\begin{split}
J_2 &\approx \int_{\T^d \times\R^d} \vv^q \partial^\beta g_\sigma [\partial^{\beta'} f \ast |\cdot|^\gamma] \vv^q \partial^{\beta''} g_\sigma \dd v \dd x \\
&\leq 
\|\vv^q \partial^\beta g_\sigma\|_{L^2(\T^d\times\R^d)} \| \partial^{\beta'} f\ast |\cdot|^\gamma\|_{L^\infty(\T^d\times\R^d)} \|\vv^q \partial^{\beta''} g_\sigma\|_{L^2(\T^d\times\R^d)}\\
&\leq \|\partial^{\beta'} f\|_{L^\infty_m(\T^d\times\R^d)} \|g_\sigma\|_{H^k_q(\T^d\times\R^d)}^2\\
&\leq 
\|\partial^{\beta'} f\|_{H^{d+1}_m(\T^d\times\R^d)} \|g_\sigma\|_{H^k_q(\T^d\times\R^d)}^2\\
&\leq \|f\|_{H^k_m(\T^d\times\R^d)}\|g_\sigma\|_{H^k_q(\T^d\times\R^d)}^2,
\end{split}
\]
using the Sobolev embedding $H^{d+1}(\T^d\times\R^d) \subset L^\infty(\T^d\times\R^d)$. 

If $|\beta''|\leq k- d - 2$, we proceed differently depending on $\gamma$. If $\gamma > -d/2$, we apply Lemma  \ref{l:convolution}(b) to $\partial^{\beta'} f$, since $m > d/2+\gamma$:
\[
\begin{split}
J_2
&\leq
C\int_{\T^d}\|\vv^q \partial^\beta g_\sigma\|_{L^2(\R^d)} \| \partial^{\beta'} f\ast |\cdot|^\gamma\|_{L^\infty(\R^d)} \|\vv^q \partial^{\beta''} g_\sigma\|_{L^2(\R^d)} \dd x\\
&\leq C \int_{\T^d} \|\vv^q \partial^\beta g_\sigma\|_{L^2(\R^d)} \| \partial^{\beta'} f\|_{L^2_m(\R^d)} \|\vv^q \partial^{\beta''} g_\sigma\|_{L^2(\R^d)} \dd x\\
&\leq C \|\vv^q \partial^\beta g_\sigma\|_{L^2(\T^d\times\R^d)} \|\partial^{\beta'} f\|_{L^2_m(\T^d\times\R^d)} \|\vv^q \partial^{\beta''} g_\sigma\|_{L^\infty_x L^2_v(\T^d\times\R^d)}\\
&\leq C \|g_\sigma\|_{H^k_q(\T^d\times\R^d)} \|\partial^{\beta'} f \|_{L^2_m(\T^d\times\R^d)} \|\vv^q \partial^{\beta''} g_\sigma\|_{H^{\lceil (d+1)/2\rceil}_x L^2_v(\T^d\times\R^d)}\\
&\leq C\|f\|_{H^k_m(\T^d\times\R^d)}\|g_\sigma\|_{H^k_q(\T^d\times\R^d)}^2.
\end{split}
\]
where we used the Sobolev embedding $H^{\lceil (d+1)/2\rceil}_x(\T^d)\subset L^\infty_x(\T^d)$. Here, $\lceil x \rceil$ denotes the smallest integer $\geq x$.  Note that $|\beta''| + \lceil (d+1)/2\rceil \leq k-d/2-1 < k$. 

If $\gamma \leq -d/2$, we apply the $L^2$ convolution estimate of Lemma \ref{l:young2} to $\partial^{\beta'} f$ with $\theta = m >\gamma + d$: 
\[
\begin{split}
J_2 
&\leq 
C\|\vv^{q} \partial^{\beta''} g_\sigma\|_{L^\infty(\T^d\times\R^d)} \int_{\T^d} \|\vv^q \partial^\beta g_\sigma\|_{L^2(\R^d)} \| \partial^{\beta'} f\ast |\cdot|^\gamma\|_{L^2(\R^d)} \dd x\\
&\leq C \|\partial^{\beta''} g_\sigma\|_{H^{d+1}_{q}(\T^d\times\R^d)} \int_{\T^d} \|\vv^q \partial^{\beta} g_\sigma\|_{L^2(\R^d)} \|\partial^{\beta'} f\|_{L^2_m(\R^d)} \dd x\\
&\leq C \|g_\sigma\|_{H^k_q(\T^d\times\R^d)}^2 \|\partial^{\beta'} f\|_{L^2_m(\T^d\times\R^d)},\\
\end{split}
\]
where we have applied the Sobolev embedding in $\T^d\times\R^d$ to $\partial^{\beta''} g_\sigma$. 

To address the terms $J_1$ and $J_3$ in \eqref{e:J123}, there are four cases:

\medskip

\noindent {\it Case 1: $\beta' = (0,\ldots,0)$, $\beta'' = \beta$.} When all derivatives fall on $g_\sigma$, we use the coercivity estimate of Lemma \ref{l:good-sign}, with exponent $m > \gamma+2s + d$:
\[
\begin{split}
J_1 &= \int_{\T^d\times\R^d} \vv^q \partial^\beta g_\sigma Q_1(f, \vv^q\partial^\beta g_\sigma)  \dd v \dd x\\
&\leq - \int_{\T^d} N_{s,\gamma}^f(\vv^q\partial^\beta g_\sigma)^2 \dd x + C\int_{\T^d} \|\vv^q \partial^\beta g_{\sigma}\|_{L^2(\R^d)}^2 \|f\|_{L^\infty_m(\R^d)} \dd x\\
&\leq - \int_{\T^d} N_{s,\gamma}^f(\vv^q\partial^\beta g_\sigma)^2 \dd x + C\|\vv^q \partial^\beta g_{\sigma}\|_{L^2(\T^d\times\R^d)}^2 \|f\|_{L^\infty_m(\T^d\times\R^d)}.
\end{split}
\]
For $J_3$, we use the commutator estimate of Lemma \ref{l:commutator}:
\[
\begin{split}
J_3 &= \int_{\T^d\times \R^d} \vv^{q} \partial^\beta g_\sigma [ \vv^q Q(f,\partial^\beta g_\sigma) - Q(f,\vv^q\partial^\beta g_\sigma) \dd v \dd x\\
&\leq C\|\vv^q \partial^\beta g_\sigma\|_{L^2(\T^d\times\R^d)} \left(\|f\|_{L^\infty_m(\T^d\times\R^d)}^{1/2} \int_{\T^d} N_{s,\gamma}^f(\vv^q \partial^{\beta} g)\dd x\right.\\
&\qquad \qquad \qquad \left. + \|f\|_{L^\infty_m(\T^d\times\R^d)} \|\partial^{\beta} g\|_{L^2_{q-2s}(\T^d\times\R^d)}\right)\\
&\leq \int_{\T^d} N_{s,\gamma}^f (\vv^q\partial^\beta g_\sigma)^2 \dd x + C\|f\|_{L^\infty_m(\T^d\times\R^d)}\|\partial^\beta g_\sigma\|_{L^2_q(\T^d\times\R^d)}^2,
\end{split}
\]
after using Young's inequality. Adding $J_1$ and $J_3$, the $N_{s,\gamma}^f(\vv^q \partial^\beta g_\sigma)$ terms cancel, and applying the Sobolev inequality to $f$ (since $k > d+1$), we have
\[
J_1+J_3 \leq C \|g_\sigma\|_{H^k_q(\T^d\times\R^d)}^2 \|f\|_{H^k_m(\T^d\times\R^d)}.
\]

\medskip

\noindent {\it Case 2: $|\beta'| = 1, |\beta''| \leq k-1$.} We consider the worst subcase $|\beta''| = k-1$ since the other subcases can be handled by the same method. 

In $J_1$, a naive estimate of $Q_1(\partial^{\beta'} f, \vv^q \partial^{\beta''} g_\sigma)$ would place $2s$ derivatives on $\partial^{\beta''} g_\sigma$, giving a term that cannot be bounded by the $H^k_q$ norm of $g_\sigma$ when $s>1/2$, since $|\beta''|+2s>k$. To get around this issue, we use a symmetrization method\footnote{A related symmetrization method was applied in \cite[Proposition 3.1(iv)]{HST2020boltzmann} and \cite[Proposition 2.9]{henderson2022existence} to solve the same issue in the context of the Boltzmann equation.} that exploits the fact that $\partial^{\beta} g_\sigma = \partial \partial^{\beta''} g_\sigma$ for some partial derivative $\partial$ in either $x$ or $v$. Let us assume $\partial = \partial_{v_i}$ for some $i$, since the case $\partial = \partial_{x_i}$ is simpler as $x$ derivatives commute with $\vv^q$. 

Integrating by parts in $v_i$, we have
\begin{equation}\label{e:J1-symmetry}
\begin{split}
\frac 1 {c_{d,\gamma,s}}J_1 &= \int_{\T^d\times\R^d} \vv^q \partial_{v_i} (\partial^{\beta''} g_\sigma) [f\ast |\cdot|^{\gamma+2s}] (-\Delta)^s \partial^{\beta''} g_\sigma \dd v \dd x\\
&= - \int_{\T^d\times\R^d} \vv^q \partial^{\beta''} g_\sigma \partial_{v_i} \left([f\ast |\cdot|^{\gamma+2s}] (-\Delta)^s \partial^{\beta''} g_\sigma \right)\dd v \dd x\\
&\quad  -\int_{\T^d\times\R^d} q\vv^{q-2} v_i \partial^{\beta''} g_\sigma [f\ast |\cdot|^{\gamma+2s}] (-\Delta)^s \partial^{\beta''} g_\sigma \dd v \dd x\\
&= : J_{1,1} + J_{1,2}.
\end{split}
\end{equation}
To bound $J_{1,1}$, we introduce the abbreviations $F = \partial^{\beta'} f \ast |\cdot|^{\gamma+2s}$ and $G = \vv^q \partial^{\beta''} g_\sigma$ and write
\[
\begin{split}
J_{1,1} &= -\int_{\T^d\times\R^d}  G \partial_{v_i} F (-\Delta)^s G \dd v \dd x - \int_{\T^d\times\R^d} G F (-\Delta)^s \partial_{v_i} G \dd v \dd x\\
&= -\int_{\T^d\times\R^d} (-\Delta)^s(G\partial_{v_i} F) G \dd v \dd x - \int_{\T^d\times\R^d} (-\Delta)^s(GF) \partial_{v_i} G \dd v \dd x.
\end{split}
\]
To both of these fractional Laplacians, we apply the formula
\[
 (-\Delta)^s (\zeta \eta) = \zeta(-\Delta)^s \eta + \eta(-\Delta)^s \zeta - E(\zeta,\eta),
 \]
 with
 \begin{equation*}
 E(\zeta,\eta) := c_{s,d} \int_{\R^d} \frac{(\zeta(v+w) - \zeta(v))(\eta(v+w)- \eta(v))}{|w|^{d+2s}} \dd w, 
 \end{equation*}
and obtain
\begin{equation}\label{e:six-terms}
\begin{split}
J_{1,1} &= -\int_{\T^d\times\R^d} G^2 (-\Delta)^s(\partial_{v_i} F) \dd v \dd x - \int_{\T^d\times\R^d} \partial_{v_i} F G (-\Delta)^s G \dd v \dd x + \int_{\T^d\times\R^d} G E(G, \partial_{v_i} F) \dd v\dd x\\
&\quad - \int_{\T^d\times\R^d} F \partial_{v_i} G (-\Delta)^s G \dd v \dd x - \int_{\T^d\times \R^d} G \partial_{v_i} G (-\Delta)^s F \dd v \dd x+\int_{\T^d\times\R^d} \partial_{v_i} G E(F,G) \dd v \dd x.
\end{split}
\end{equation}
Taking the terms in this expression one-by-one, for the first term we use
\[
(-\Delta)^s (\partial_{v_i} F) = (-\Delta)^s (\partial_{v_i}\partial^{\beta'} f \ast |\cdot|^{\gamma+2s}) \approx \partial_{v_i}\partial^{\beta'} f \ast |\cdot|^\gamma,
\] 
which follows from a calculation that places $(-\Delta)^s$ onto the convolution kernel, similar to the argument in \eqref{e:Q2der}. Combining this with Lemma \ref{l:convolution}(a) and Sobolev embedding in $\T^d\times\R^d$, we have
\[
\begin{split}
\int_{\T^d\times\R^d} G^2 (-\Delta)^s [\partial_{v_i} \partial^{\beta'} f \ast |\cdot|^{\gamma+2s}] \dd v\dd x &\lesssim \|G\|_{L^2(\T^d\times\R^d)}^2 \|\partial_{v_i}\partial^{\beta'} f\|_{L^\infty_m(\T^d\times\R^d)}\\
&\lesssim \|g_\sigma\|_{H^k_q(\T^d\times\R^d)}^2 \| f\|_{H^k_m(\T^d\times\R^d)},
\end{split}
\]
since $m> \gamma+d$ and $|\beta'|+1+(d+1) = d+3 \leq k$. For the second term in \eqref{e:six-terms}, we use Lemma \ref{l:good-sign} and inequality \eqref{e:Nbound}:
\[
\begin{split}
\int_{\T^d\times\R^d} \partial_{v_i} F G(-\Delta)^s G \dd v \dd x &= \int_{\T^d\times\R^d} G Q_1(\partial_{v_i}\partial^{\beta'} f, G) \dd v \dd x\\
&\leq \int_{\T^d}\left[\left| N_{s,\gamma}^{\partial_{v_i} \partial^{\beta'} f}(G)\right|^2  + C\|G\|_{L^2(\R^d)}^2 \|\partial_{v_i} \partial^{\beta'} f\|_{L^\infty_m(\R^d)}\right] \dd x\\
&\lesssim \int_{\T^d}  \|\partial_{v_i} \partial^{\beta'} f\|_{L^\infty_m(\R^d)} \left( \|G\|_{H^s(\R^d)}^2 + \|G\|_{L^2(\R^d)}^2\right) \dd x\\
&\lesssim \|\partial_{v_i} \partial^{\beta'} f\|_{L^\infty_m(\T^d\times\R^d)} \|G\|_{L^2_x H^s_v(\T^d\times\R^d)}^2\\
&\lesssim \|f\|_{H^k_m(\T^d\times\R^d)} \|g_\sigma\|_{H^k_q(\T^d\times\R^d)}^2,
\end{split}
\]
applying Sobolev embedding to $\partial_{v_i} \partial^{\beta'} f$ as above. For the third term in \eqref{e:six-terms}, we use the definition of $E(\zeta,\eta)$ to write
\begin{equation}\label{e:third-term}
\begin{split}
\int_{\T^d\times\R^d} G E(G,\partial_{v_i} F) \dd v \dd x &\lesssim \int_{\T^d\times \R^d} G \left(\int_{\R^d} \frac{|G(x,v+w) - G(x,v)|^2}{|w|^{d+2s}}\dd w\right)^{1/2}\\
&\qquad\qquad \times\left(\int_{\R^d} \frac{|\partial_{v_i} F(x,v+w) - \partial_{v_i}F(x,v)|^2}{|w|^{d+2s}}\dd w\right)^{1/2} \dd v \dd x.
\end{split}
\end{equation}
We estimate the $w$ integral involving $\partial_{v_i} F$ in two pieces. For some small $\eps>0$ such that $s+\eps<1$, we first have
\[
\begin{split}
\int_{B_1}  &\frac{|\partial_{v_i} F(x,v+w) - \partial_{v_i}F(x,v)|^2}{|w|^{d+2s}}\dd w\\
 &= \int_{B_1} |w|^{-d-2s} \left( \int_{\R^d} \partial_{v_i} \partial^{\beta'} f (v_*) [ |v+w-v_*|^{\gamma+2s} - |v-v_*|^{\gamma+2s}] \dd v_*\right)^2 \dd w\\
&\leq \int_{B_1} |w|^{-d-2s} \left( \int_{\R^d} \partial_{v_i} \partial^{\beta'} f (v_*) |v-v_*|^{\gamma+s-\eps} |w|^{s+\eps} \dd v_*\right)^2 \dd w\\
&\leq \int_{B_1} |w|^{-d+2\eps} \dd w \left([\partial_{v_i} \partial^{\beta'} f \ast |\cdot|^{\gamma+s-\eps}](v)\right)^2\\
&\lesssim \|\partial_{v_i} \partial^{\beta'}f\|_{L^\infty_m(\T^d\times\R^d)}^2,
\end{split}
\]
by Lemma \ref{l:convolution}(a). In the third line, we used the H\"older estimate $[|\cdot|^{\gamma+2s}]_{C^{s+\eps}(v)} \lesssim |v|^{\gamma+s-\eps}$, which follows from Lemma \ref{l:q-power}. Next, 
\[
\begin{split}
&\int_{\R^d\setminus B_1}  \frac{|\partial_{v_i} F(x,v+w) - \partial_{v_i}F(x,v)|^2}{|w|^{d+2s}}\dd w\\
 &\quad\leq \|\partial_{v_i}\partial^{\beta'} f\|_{L^\infty_m(\R^d)}^2 \int_{\R^d\setminus B_1} |w|^{-d-2s} \left( \int_{\R^d} \langle v_*\rangle^{-m} [ |v+w-v_*|^{\gamma+2s} - |v-v_*|^{\gamma+2s}] \dd v_*\right)^2 \dd w\\
 &\quad \lesssim \|\partial_{v_i}\partial^{\beta'} f\|_{L^\infty_m(\R^d)}^2\int_{\R^d\setminus B_1} |w|^{-d-2s} \dd w\\
 &\quad \lesssim \|\partial_{v_i}\partial^{\beta'} f\|_{L^\infty_m(\R^d)}^2,
\end{split}
\]
where in the third line, we used Lemma \ref{l:convolution}(a) with $f(v) = \langle v\rangle^{-m}$. Returning to \eqref{e:third-term}, we have
\[
\begin{split}
\int_{\T^d\times\R^d} G E(G,\partial_{v_i} F) \dd v \dd x &\leq \int_{\T^d} \|G\|_{L^2(\R^d)} \|G\|_{H^s(\R^d)} \|\partial_{v_i} \partial^{\beta'} f\|_{L^\infty_m(\R^d)} \dd x\\
&\leq \|G\|_{L^2(\T^d\times\R^d)} \|G\|_{L^2_x H^s_v(\T^d\times\R^d)} \|\partial_{v_i} \partial^{\beta'} f\|_{L^\infty_m(\T^d\times\R^d)}\\
&\leq \|g_\sigma\|_{H^k_q(\T^d\times\R^d)}^2 \|f\|_{H^k_m(\T^d\times\R^d)},
\end{split}
\]
using Sobolev embedding as above, and the inequality $\|G\|_{L^2_x H^s_v(\T^d\times\R^d)} \lesssim \|g_\sigma\|_{H^k_q(\T^d\times\R^d)}$, which can be proven in a straightforward way using interpolation. 

The fourth term in \eqref{e:six-terms} is equal to $-\frac 1 {c_{d,\gamma,s}} J_1$, so we absorb it into the left-hand side of \eqref{e:J1-symmetry}. The fifth term is bounded in a similar manner to the first term: since $(-\Delta)^s F \approx \partial^{\beta'} f \ast |\cdot|^\gamma$,
\[
\begin{split}
\int_{\T^d\times\R^d} G \partial_{v_i} G (-\Delta)^s F \dd v\dd x &\lesssim \|G\|_{L^2(\T^d\times\R^d)} \|\partial_{v_i} G\|_{L^2(\T^d\times\R^d)} \|\partial^{\beta'} f\|_{L^\infty_m(\T^d\times\R^d)}\\
&\lesssim \|g\|_{H^k_q(\T^d\times\R^d)}^2 \| f\|_{H^k_m(\T^d\times\R^d)},
\end{split}
\]
by Lemma \ref{l:convolution}(a) and Sobolev embedding. Finally, the sixth term in \eqref{e:six-terms} is bounded by exactly the same argument as the third term. 

Next, we note that the term $J_{1,2}$ in \eqref{e:J1-symmetry} can be bounded by the same method as $J_{1,1}$ (but simpler) because it has fewer derivatives and a smaller exponent of $\vv$. 
Returning to \eqref{e:J1-symmetry}, we now have
\[
2J_1 \leq C \|g_\sigma\|_{H^k_q(\T^d\times\R^d)}^2 \|f\|_{H^k_m(\T^d\times\R^d)},
\]
as desired.

For the commutator term, Lemma \ref{l:commutator} implies
\[
\begin{split}
J_3
 &\leq  \|\vv^q \partial^\beta g_\sigma\|_{L^2(\T^d\times\R^d)} \|\vv^q Q(\partial^{\beta'} f,\partial^{\beta''} g_\sigma) - Q(\partial^{\beta'} f,\vv^q \partial^{\beta''}g_\sigma)\|_{L^2(\T^d\times \R^d)}\\
&\lesssim  \|\vv^q \partial^\beta g_\sigma\|_{L^2(\T^d\times\R^d)} \left( \|\partial^{\beta'} f\|_{L^\infty_m(\T^d\times\R^d)}^{1/2}\int_{\T^d} N_{s,\gamma}^f(\vv^q \partial^{\beta''} g_\sigma) \dd x\right.\\
&\qquad\qquad \left.+ \|\partial^{\beta'} f\|_{L^\infty_m(\T^d\times\R^d)} \|\partial^{\beta''} g_\sigma\|_{L^2(\T^d\times\R^d)}\right).
\end{split}
\]
Using the upper bound \eqref{e:Nbound} for $N_{s,\gamma}^f(\vv^q\partial^{\beta''} g_\sigma)$, we obtain
\[
\begin{split}
J_3 &\lesssim \|\vv^q \partial^\beta g_\sigma\|_{L^2(\T^d\times\R^d)} \|\partial^{\beta'} f\|_{L^\infty_m(\T^d\times\R^d)}\left( \|\vv^q \partial^{\beta''} g_\sigma\|_{H^s(\R^d)} + \|\vv^q \partial^{\beta''} g_\sigma\|_{L^2(\T^d\times\R^d)}\right)\\
&\lesssim \|\vv^q \partial^\beta g_\sigma\|_{L^2(\T^d\times\R^d)} \|\partial^{\beta'} f\|_{H^{d+1}_m(\T^d\times\R^d)}\left( \|\vv^q \partial^{\beta''} g_\sigma\|_{H^1(\R^d)} + \|\vv^q \partial^{\beta''} g_\sigma\|_{L^2(\T^d\times\R^d)}\right)\\
&\lesssim  \|g_\sigma\|_{H^k_q(\T^d\times\R^d)}^2  \|f\|_{H^k_m(\T^d\times\R^d)},
\end{split}
\]
after using the standard interpolation inequality $\|\cdot\|_{H^s}^2\lesssim \|\cdot\|_{H^1}^2 + \|\cdot\|_{L^2}^2$ and Sobolev embedding for $\partial^{\beta'}f$ (since $|\beta'|+d+1 = d+2 < k$).

\medskip

\noindent {\it Case 3: $2 \leq  |\beta'| \leq k - d - 1$.} 
In this case, there is room to apply $2s$ derivatives to $\partial^{\beta''} g_\sigma$, so we use the inequality 
\[
\|Q_1(F, G)\|_{L^2(\R^d)} \approx \| [F\ast|\cdot|^{\gamma+2s}] (-\Delta)^s G\|_{L^2(\R^d)} \leq C \|F\|_{L^\infty_m(\R^d)} \|G\|_{H^{2s}(\R^d)},
\]
which follows from Lemma \ref{l:convolution}(a) since $m> \gamma+2s+d$. This gives
\[
\begin{split}
J_1 &\leq C\| \vv^q \partial^\beta g_\sigma\|_{L^2(\T^d\times\R^d)} \left(\int_{\T^d} \|\partial^{\beta'} f\|_{L^\infty_m(\R^d)}^2 \| \vv^q \partial^{\beta''} g_\sigma\|_{H^{2s}(\R^d)}^2 \dd x    \right)^{1/2}\\
&\leq C\| \vv^q \partial^\beta g_\sigma\|_{L^2(\T^d\times\R^d)} \|\partial^{\beta'} f \|_{L^\infty_m(\T^d\times\R^d)} \left(\int_{\T^d} \left(\|\vv^q \partial^{\beta''} g_\sigma\|_{H^{2}(\R^d)}^2 + \|\vv^q\partial^{\beta''} g_{\sigma}\|_{L^2(\R^d)}^2\right) \dd x    \right)^{1/2}\\
&\leq C\| \vv^q \partial^\beta g_\sigma\|_{L^2(\T^d\times\R^d)}^2 \|\partial^{\beta'} f \|_{H^{d+1}_m(\T^d\times\R^d)},
\end{split}
\]
by Sobolev embedding and the interpolation inequality $\|\cdot\|_{H^{2s}}^2\lesssim \|\cdot\|_{H^2}^2 + \|\cdot\|_{L^2}^2$. The commutator term is handled exactly as in Case 2, since $|\beta'| + d+1 \leq k$, and we obtain
\[
J_1 + J_3 \leq C\|g_\sigma\|_{H^k_q(\T^d\times\R^d)}^2 \|f\|_{H^k_m(\T^d\times\R^d)}.
\]

\medskip

\noindent {\it Case 4: $k-d \leq |\beta'| \leq k$.} 
By our condition on $k$, we must have $|\beta''| \leq k - d - 3$ in this case. 
As in the analysis of $J_2$ above, we split this case into sub-cases based on $\gamma+2s$. 

Starting with $J_1$, if $\gamma+2s > -d/2$, we apply Lemma \ref{l:convolution}(b), since $m>\gamma+2s+d > \gamma+2s+d/2$:
\[
\begin{split}
J_1 &\approx \int_{\T^d\times\R^d} \vv^q \partial^\beta g_\sigma [\partial^{\beta'} f \ast |\cdot|^{\gamma+2s}] (-\Delta)^s (\vv^q\partial^{\beta''} g_\sigma) \dd v \dd x\\
&\leq \int_{\T^d}  \|\vv^q \partial^\beta g_\sigma\|_{L^2(\R^d)} \|\partial^{\beta'} f\ast |\cdot|^{\gamma+2s}\|_{L^\infty(\R^d)} \|(-\Delta)^s (\vv^q \partial^{\beta''} g_\sigma)\|_{L^2(\R^d)}\dd x \\
&\leq C\int_{\T^d} \|\vv^q \partial^\beta g_\sigma\|_{L^2(\R^d)} \|\partial^{\beta'} f\|_{L^2_m(\R^d)} \|\vv^q \partial^{\beta''} g_\sigma\|_{H^{2s}(\R^d)} \dd x\\
&\leq C\|\vv^q \partial^\beta g_\sigma\|_{L^2(\T^d\times\R^d)} \|\partial^{\beta'}f \|_{L^2_m(\T^d\times\R^d)} \|\vv^q \partial^{\beta''} g_\sigma\|_{L^\infty_x H^{2s}_v(\T^d\times\R^d)}\\
&\leq C \|g_\sigma\|_{H^k_q(\T^d\times\R^d)} \|f\|_{H^k_m(\T^d\times\R^d)},
\end{split}
\]
where we used the Sobolev embedding $H^{\lceil (d+1)/2\rceil}_x(\T^d)\subset L^\infty_x(\T^d)$ and $|\beta''|+2s+\lceil (d+1)/2\rceil \leq k - d/2 - 2 + 2s < k$. 

If $\gamma+2s\leq -d/2$, we apply Lemma \ref{l:young2} to $\partial^{\beta'} f$ and the Sobolev embedding $H^{d+1}(\T^d\times\R^d) \subset L^\infty(\R^d)$ to $\partial^{\beta''} g_\sigma$:
\[
\begin{split}
J_1 
&\lesssim \|(-\Delta)^s (\vv^q \partial^{\beta''} g_\sigma)\|_{L^\infty(\T^d\times \R^d)}\int_{\T^d}  \|\vv^q \partial^\beta g_\sigma\|_{L^2(\R^d)} \|\partial^{\beta'} f\ast |\cdot|^{\gamma+2s}\|_{L^2(\R^d)} \dd x \\
&\leq C\|(-\Delta)^s (\vv^q \partial^{\beta''} g_\sigma)\|_{H^{d+1}(\T^d\times \R^d)} \int_{\T^d} \|\vv^q \partial^\beta g_\sigma\|_{L^2(\R^d)} \|\partial^{\beta'} f\|_{L^2_m(\R^d)}  \dd x\\
&\leq C\|\vv^q \partial^{\beta''} g_\sigma\|_{H^{2s+d+1}(\T^d\times\R^d)} \|\vv^q \partial^{\beta} g_\sigma\|_{L^2(\T^d\times\R^d)} \|\partial^{\beta'}f \|_{L^2_m(\T^d\times\R^d)} \\
&\leq C \|g_\sigma\|_{H^k_q(\T^d\times\R^d)}^2 \|f\|_{H^k_m(\T^d\times\R^d)},
\end{split}
\]
since $|\beta''| + d+1 +2s \leq k-2+2s< k$.

For the commutator term, there are again two sub-sases. If $\gamma+2s \leq -d/2$, we repeat a simple calculation from the proof of Lemma \ref{l:commutator} to write
\begin{equation}\label{e:commutator-step}
\begin{split}
| \vv^q Q(\partial^{\beta'} f,\partial^{\beta''} g_\sigma) - &Q(\partial^{\beta'} f,\vv^q \partial^{\beta''}g_\sigma)|\\
 &\approx  [\partial^{\beta'}f \ast |\cdot|^{\gamma+2s}] \int_{\R^d}\partial^{\beta''}g_\sigma(t,x,v+w)\frac{\vv^q - \langle v+w\rangle^{q}}{|w|^{d+2s}} \dd w.
 \end{split}
 \end{equation}
 By H\"older's inequality and Lemma \ref{l:q-power}, we have
 \[
 \begin{split}
| \vv^q Q(\partial^{\beta'} f,\partial^{\beta''} g_\sigma) - &Q(\partial^{\beta'} f,\vv^q \partial^{\beta''}g_\sigma)|  \\
& \lesssim
\|\partial^{\beta''} g_\sigma\|_{L^\infty_{q}(\R^d)} [\partial^{\beta'} f  \ast |\cdot|^{\gamma+2s}] \vv^q \int_{\R^d}\frac{\langle v+w\rangle^{-q} - \langle v\rangle^{-q} }{|w|^{d+2s}} \dd w \\
&\lesssim
\|\partial^{\beta''} g_\sigma\|_{L^\infty_{q}(\R^d)}\vv^{-2s} [\partial^{\beta'} f \ast |\cdot|^{\gamma+2s}].
\end{split}
\]
Taking the $L^2$ norm in $v$ and applying Lemma \ref{l:young2} with $\theta = m > \gamma+2s+d$, we have, for each fixed $x\in \T^d$, 
\[
\|\vv^q Q(\partial^{\beta'} f,\partial^{\beta''} g_\sigma) - Q(\partial^{\beta'} f,\vv^q \partial^{\beta''}g_\sigma)\|_{L^2(\R^d)} \leq C \|\partial^{\beta''} g_\sigma\|_{L^\infty_{q}(\R^d)}\|\partial^{\beta'} f\|_{L^2_m(\R^d)},
\]
and as a result,
\[
\begin{split}
J_3 &\leq C\|\vv^q\partial^\beta g_\sigma\|_{L^2(\T^d\times\R^d)}\|\partial^{\beta''} g_\sigma\|_{L^\infty_q(\T^d\times\R^d)} \|\partial^{\beta'} f\|_{L^2_m(\T^d\times\R^d)}\\
&\leq C \|g_\sigma\|_{H^k_q(\T^d\times\R^d)}^2 \|f\|_{H^k_m(\T^d\times\R^d)},
\end{split}
\]
by Sobolev embedding, since $|\beta''|+d+1 \leq k-2 < k$. 

If $\gamma+2s>-d/2$,  we have \eqref{e:commutator-step} as above. We then write, omitting the dependence of $g_\sigma$ on $t$ and $x$, 
\[
\begin{split}
\int_{\R^d} \partial^{\beta''} g_\sigma &(v+w) \frac{\vv^q - \langle v+w\rangle^q}{|w|^{d+2s}} \dd w\\
 & = \vv^q \int_{\R^d} \partial^{\beta''} g_\sigma(v+w) \langle v+w\rangle^{q} \frac{ \langle v+w\rangle^{-q} - \vv^{-q}}{|w|^{d+2s}} \dd w \\
&= \vv^q \int_{\R^d} (\partial^{\beta''}g_\sigma(v+w) \langle v+w\rangle^q - \partial^{\beta''} g_\sigma(v) \vv^q)  \frac{\langle v+w\rangle^{-q} - \vv^{-q}}{|w|^{d+2s}} \dd w\\
&\quad + \vv^{2q} \partial^{\beta''} g_\sigma(v) \int_{\R^d} \frac{\langle v+w\rangle^{-q} - \vv^{-q}}{|w|^{d+2s}} \dd w\\
&\lesssim \vv^{-s} \left(\int_{\R^d} \frac{|\partial^{\beta''} g_\sigma(v+w)\langle v+w\rangle^q - \partial^{\beta''} g_\sigma(v)\vv^q|^2} {|w|^{d+2s}} \dd w \right)^{1/2}+ \vv^{q-2s} \partial^{\beta''} g_\sigma(v),
\end{split}
\]
using Lemma \ref{l:q-power} again. Using this in \eqref{e:commutator-step} and integrating against $\vv^q\partial^\beta g_\sigma$, we have, using Lemma \ref{l:convolution}(b),
\[
\begin{split}
J_3 &\lesssim \int_{\T^d} \|\vv^q\partial^\beta g_\sigma\|_{L^2_v(\R^d)} \|\partial^{\beta'} f \ast |\cdot|^{\gamma+2s}\|_{L^\infty_v(\R^d)} \left( \|\vv^q \partial^{\beta''} g_\sigma\|_{H^s_v(\R^d)} + \|\vv^{q-2s}\partial^{\beta''}g_\sigma\|_{L^2_v(\R^d)}\right)\\
&\lesssim \int_{\T^d}  \|\vv^q\partial^\beta g_\sigma\|_{L^2_v(\R^d)} \|\partial^{\beta'} f\|_{L^2_m(\R^d)}  \left( \|\vv^q \partial^{\beta''} g_\sigma\|_{H^s_v(\R^d)} + \|\vv^{q-2s}\partial^{\beta''}g_\sigma\|_{L^2_v(\R^d)}\right) \dd x\\
&\lesssim \|\vv^q \partial^\beta g_\sigma\|_{L^2(\T^d\times\R^d)} \|\vv^m \partial^{\beta'} f\|_{L^2(\T^d\times\R^d)} \|\vv^q g_\sigma\|_{L^\infty_x H^s_v(\T^d\times\R^d)}\\
&\lesssim \|g_\sigma\|_{H^k_q(\T^d\times\R^d)}^2 \|f\|_{H^k_m(\T^d\times\R^d)},
\end{split}
\]
by the Sobolev embedding $H^{\lceil (d+1)/2 \rceil}_x(\T^d) \subset L^\infty_x(\T^d)$, since $|\beta''| +  2s + \lceil (d+1)/2\rceil< k$.


Returning to \eqref{e:energy-est} and adding up all multi-indices $\beta$ with $|\beta|\leq k$, we have shown
\[
\frac 1 2 \frac d {dt}\|g_\sigma\|_{H^k_q(\T^d\times\R^d)}^2 \leq C\left( (\|f\|_{H^k_m(\T^d\times\R^d)}+1) \|g_\sigma\|_{H^k_q(\T^d\times\R^d)}^2 +  \|R\|_{L^2_q(\T^d\times\R^d)}^2\right),
\]
with $C$ as in the statement of the lemma. Gr\"onwall's inequality completes the proof.
\end{proof}

The following (somewhat crude) estimates on the collision operator will be helpful in our proof of local existence:

\begin{lemma}\label{l:simple}
The estimates
\[
\begin{split}
\|Q(F,G)\|_{L^2_q(\T^d\times\R^d)} &\leq C\|F\|_{H^{d+1}_m(\T^d\times\R^d)}\|G\|_{H^2_q(\T^d\times\R^d)},\\
\|Q(F,G)\|_{L^\infty(\T^d\times\R^d)} &\leq C\|F\|_{H^{d+1}_m(\T^d\times\R^d)} \|G\|_{H^{d+3}(\T^d\times\R^d)},
\end{split}
\]
hold, for any $q, m> \gamma + 2s + d$ and any $F, G:\T^d\times\R^d\to \R$ such that the right-hand sides are finite. The constant $C>0$ depends only on $d$, $\gamma$, $s$, $q$, and $m$. 
\end{lemma}

\begin{proof}
Writing 
\[
Q(F,G) \approx [F\ast |\cdot|^{\gamma+2s}] (-\Delta)^s G + [F\ast |\cdot|^\gamma] G,
\]
the conclusion follows from Lemma \ref{l:convolution}(a), the Sobolev embedding $H^{d+1}(\T^d\times\R^d)\subset L^\infty(\T^d\times\R^d)$, and standard interpolation estimates for (unweighted) Sobolev norms.
\end{proof}

Next, we prove existence for the linear isotropic Boltzmann equation:

\begin{lemma}\label{l:linear-existence}
Let $T>0$, $m>\gamma+2s+d$, and $q> \max\{\gamma+2s+d,1\}$ be arbitrary. For any $f\in L^\infty([0,T], H^{2d+2}_m(\T^d\times\R^d))$ and $g_{\rm in}\in H^{2d+2}_q(\T^d\times\R^d)$ with $f, g_{\rm in} \geq 0$, there exists a solution 
\[
g\in L^\infty([0,T], H^{2d+2}_q(\T^d\times\R^d))\cap W^{1,\infty}([0,T],L^2_{q-1}(\T^d\times\R^d))
\] 
to the initial value problem
\begin{equation}\label{e:lin-iso}
\begin{cases}\partial_t g + v\cdot\nabla_x g = Q(f,g),\\
g(0,x,v) = g_{\rm in}(x,v).
\end{cases}
\end{equation}
Furthermore, $g\geq 0$. 
\end{lemma}
\begin{proof}
For $T>0$, define the Banach spaces
\[
\begin{split}
X_T &:= L^\infty([0,T], H^{2d+2}_q(\T^d\times\R^d)) \cap W^{1,\infty}([0,T], L^2_{q-1}(\T^d\times\R^d)).\\
Y_T &: = L^2([0,T], L^2_{q-1}(\T^d\times\R^d)) \times H^{2d+2}_q(\T^d\times\R^d).
\end{split}
\]
For fixed $\sigma\in [0,1]$ and $f\in L^\infty([0,T],H^{2d+2}_m(\T^d\times\R^d))$, define the linear operator 
\[
\mathcal L_\sigma: X_T \to Y_T,
\] 
by 
\[
\mathcal L_\sigma(g) = \left(\partial_t g + \sigma v\cdot \nabla_x g - \sigma Q(f,g) - (1-\sigma)\Delta_{x,v} g, \,g\Big|_{t=0}\right).
\]
To verify that $\mathcal L_\sigma(g)$ is indeed an element of $Y_T$, we note the following: for $g\in X_T$, we have from Lemma \ref{l:simple} and the fact that $H^{2d+2}(\T^d\times\R^d)$ embeds in $C^2(\T^d\times\R^d)$, the inequality
\begin{equation}\label{e:embed}
\begin{split}
\|\sigma v\cdot \nabla_x g - &\sigma Q(f,g) - (1-\sigma)\Delta_{x,v} g\|_{L^2_{q-1}(\T^d\times\R^d)}\\
& \leq C\left(\|f\|_{H^{2d+2}_m(\T^d\times\R^d)} + 1\right)\|g\|_{H^{2d+2}_q(\T^d\times\R^d)}, \quad 0\leq t\leq T.
\end{split}
\end{equation}
The loss of one moment comes from the term $\sigma v\cdot \nabla_x g$. We also have $\partial_t g\in L^2([0,T], L^2_{q-1}(\T^d\times \R^d))$ as a result of the $W^{1,\infty}$ norm included in the definition of $X_T$. Finally, the function $g$ is continuous as $t\to 0$ from the time regularity included in the definition of $X_T$, and since $g(t)\in H^{2d+2}_q(\T^d\times\R^d)$ uniformly in $t$, we indeed have $g\Big|_{t=0} \in H^{2d+2}_q(\T^d\times\R^d)$. 


From Lemma \ref{l:energy-estimate}, we have for any $g\in X_T$,
\begin{equation}\label{e:method1}
\|g\|_{L^\infty([0,T], H^{2d+2}_q(\T^d\times\R^d))} \leq \|\mathcal L_\sigma(g)\|_{Y_T} \exp\left(C\int_0^T(1+\|f(t')\|_{H^{2d+2}_m(\T^d\times\R^d)}\dd t'\right).
\end{equation}
Next, we use the definition of $\mathcal L_\sigma$ and \eqref{e:embed} to write
\[
\begin{split}
\|\partial_t g\|_{L^\infty([0,T],L^2_{q-1}(\T^d\times\R^d))} &\leq \|\partial_t g + \sigma v\cdot \nabla_x g - \sigma Q(f,g) - (1-\sigma)\Delta_{x,v} g\|_{L^\infty([0,T],L^2_{q-1}(\T^d\times\R^d))}\\
&\quad  + \|v\cdot\nabla_x g + Q(f,g) + \Delta_{x,v} g\|_{L^\infty([0,T],L^2_{q-1}(\T^d\times\R^d))}\\
&\leq \|\mathcal L_\sigma(g)\|_{Y_T} + C(\|f\|_{H^{d+1}_m(\T^d\times\R^d)}\|g\|_{L^\infty([0,T],H^{2d+2}_q(\T^d\times\R^d))}\\
&\leq C \|\mathcal L_\sigma(g)\|_{Y_T},
\end{split}
\]
using \eqref{e:method1} in the last line, with $C>0$ depending on $f$. Combining the last two estimates, we have
\[
\|g \|_{X_T} \leq C\|\mathcal L_\sigma(g)\|_{Y_T.}
\]

Note that $\mathcal L_0$ is a surjective map from $X_T$ to $Y_T$, as this fact corresponds to the solvability of the heat equation in $\R^{1+6}$ with a source term $R\in L^2([0,T], L^2_{q-1}(\T^d\times\R^d))$. 
Therefore, the method of continuity \cite[Theorem 5.2]{gilbargtrudinger} implies $\mathcal L_1$ is surjective from $X_T$ to $Y_T$, which means there is some $g\in X_T$ with $\mathcal L_1(g) = (0, g_{\rm in})$. This function $g$ solves \eqref{e:lin-iso}, as desired.

The nonnegativity of $g$ now follows from classical maximum principle arguments, using the monotonicity of the operator $Q(f,g)$, i.e. $Q(f,g) \geq 0$ at a nonnegative global minimum of $g$ in $v$. We omit the details. 
\end{proof}

Now we are ready to prove local existence for the nonlinear equation:

\begin{proof}[Proof of Theorem \ref{t:short-time-existence-inhom}]

Let $C$ be the constant from Lemma \ref{l:energy-estimate}, and define
\[
T = \frac{ \log 2}{2C \|f_{\rm in}\|_{H^{2d+2}_q(\T^d\times\R^d)}}.
\]

Define $f_0 (t,x,v)= f_{\rm in}(x,v)$, and for each $n$, let $f_{n+1}$ solve the initial value problem
\begin{equation}\label{e:n-iteration}
\begin{cases}
\partial_t f_{n+1} + v\cdot \nabla_x f_{n+1} = Q(f_n, f_{n+1}),\\
f(0,\cdot) = f_{\rm in},
\end{cases}
\end{equation}
on the time interval $[0,T]$. The solution $f_{n+1}$ exists and is nonnegative by Lemma \ref{l:linear-existence}. By induction, assume 
\begin{equation}\label{e:induction3}
\|f_n\|_{L^\infty([0,T],H^{2d+2}_q(\T^d\times\R^d))} \leq 2\|f_{\rm in}\|_{H^{2d+2}_q(\T^d\times\R^d)}.
\end{equation}
From Lemma \ref{l:energy-estimate} with $\sigma = 1$,  $R=0$, and $m=q$, we have
\[
\|f_{n+1}(t)\|_{H^{2d+2}_q(\T^d\times\R^d)} \leq \|f_{\rm in}\|_{H^{2d+2}_q(\T^d\times\R^d)} \exp\left(C \int_0^t (\|f_n(t')\|_{H^{2d+2}_q(\T^d\times\R^d)}+1) \dd t'\right), \quad 0\leq t\leq T.
\]
By our choice of $T$, and our inductive hypothesis on $f_n$, this implies
\[
\|f_{n+1}(t)\|_{H^{2d+2}_q(\T^d\times\R^d)} \leq \|f_{\rm in}\|_{H^{2d+2}_q(\T^d\times\R^d)}\exp(2C \|f_{\rm in}\|_{H^{2d+2}_q(\T^d\times\R^d)} T) \leq 2\|f_{\rm in}\|_{H^{2d+2}_q(\T^d\times\R^d)},
\]
and we have shown that \eqref{e:induction3} holds for all $n$. 

To prove regularity in $t$, we differentiate the equation $\partial_t f_{n+1} + v\cdot\nabla_x f_{n+1} = Q(f_n, f_{n+1})$ by some partial derivative $\partial^\beta$ in $(x,v)$ variables with $|\beta|\leq d+1$. Using Lemma \ref{l:simple} and the triangle inequality, it is straightforward to show 
\[
\|\partial^\beta \partial_t f_{n+1}(t)\|_{L^2_{q-1}(\T^d\times\R^d)} \leq C\left( \|f_{n+1}\|_{H^{|\beta|+1}_{q}(\T^d\times\R^d)} + \|f_n\|_{H^{|\beta|+d+1}_q(\T^d\times\R^d)} \|f_{n+1}\|_{H^{2+|\beta|}_{q-1}(\T^d\times\R^d)}\right),
\] 
uniformly in $t$. Since $|\beta|\leq d+1$, we conclude $f_n$ is bounded in $W^{1,\infty}([0,T],H^{d+1}_{q-1}(\T^d\times\R^d))$, uniformly in $n$. 

Therefore, a subsequence of $f_n$ converges weak-$\ast$ in $L^\infty([0,T],H^{2d+2}_q(\T^d\times\R^d))$, strongly in $L^\infty([0,T], H^{d+1}_{q-1}(\T^d\times\R^d))$, and pointwise a.e. to a limit $f \in L^\infty([0,T],H^{2d+2}_q(\T^d\times\R^d))$. To take the limit in the equation, we use the distributional form: for any smooth $\varphi$ with compact support in $(0,T)\times\T^d\times\R^d$, we have
\begin{equation}\label{e:fn-weak}
\begin{split}
-\int_{[0,T]\times\T^d\times\R^d} f_{n+1} & (\partial_t + v\cdot \nabla_x) \varphi  \dd v \dd x \dd t\\
 &= \int_{[0,T]\times\T^d\times\R^d}\left( \varphi  Q(f_n,f_{n}) + \varphi Q(f_n, f_{n+1} - f_n)\right) \dd v \dd x \dd t\\
&= \int_{[0,T]\times\T^d\times\R^d} \int_{\R^d\times\R^d} B(v,v_*,w) f_n (f_n)_*  [\varphi_*' + \varphi' - \varphi_* - \varphi] \dd w \dd v_* \dd v \dd x \dd t\\
&\quad +  \int_{[0,T]\times\T^d\times\R^d}\varphi Q(f_n, f_{n+1} - f_n) \dd v \dd x \dd t,
\end{split}
\end{equation}
by \eqref{e:weak2}. For the first term in this right-hand side, since $\varphi$ is smooth, a second-order Taylor expansion shows that $B(v,v_*,w) [\varphi_*' + \varphi' - \varphi_* - \varphi]\approx |v-v_*+w|^{\gamma+2s} |w|^{-d-2s} |w|^2$ is integrable near $w=0$. The convergence of this term then follows from the pointwise convergence of $f_n\to f$. 
The second term on the right in \eqref{e:fn-weak} converges to zero because of Lemma \ref{l:simple}, the uniform bound \eqref{e:induction3}, and the strong convergence of $f_n$ in $L^\infty([0,T],H^{d+1}_{q-1}(\T^d\times\R^d))$.
The left-hand side of \eqref{e:fn-weak} converges as a result of the weak-$\ast$ convergence of $f_n$ in $L^\infty([0,T],H^{2d+2}_q(\T^d\times\R^d))$.  In the limit, we obtain (after using \eqref{e:weak2} again)
\begin{equation}\label{e:f-weak}
-\int_{[0,T]\times\T^d\times\R^d} f (\partial_t + v\cdot \nabla_x) \varphi \dd v \dd x \dd t = \int_{[0,T]\times\T^d\times\R^d} \varphi Q(f,f) \dd v \dd x \dd t.
\end{equation}
We claim that $f$ is continuous in all three variables on $[0,T]\times\T^d\times\R^d$. Indeed, the uniform bound for $f_n$ in $W^{1,\infty}([0,T],H^{d+1}_{q-1}(\T^d\times\R^d))$, combined with Sobolev embedding, imply that $f_n$ are uniformly H\"older continuous in $(t,x,v)$, so the subsequential limit $f$ is also continuous. Sending $t\to 0$, this implies $f(0,x,v) = f_{\rm in}$. Also, since $f$ is continuous in $t$ and satisfies \eqref{e:f-weak} for any smooth $\varphi$, a standard argument implies $f$ is differentiable in $t$, and therefore we can integrate by parts on the left side of \eqref{e:f-weak}.


We have shown that $f$ solves the nonlinear equation \eqref{e:main} pointwise. This solution $f$ is naturally entitled to the energy estimate of Lemma \ref{l:energy-estimate} with $\sigma = 1$, $f = g_\sigma$, and $R=0$, and the same argument that was applied above to $\partial_t f_{n+1}$ implies $f \in W^{1,\infty}([0,T],H^{d+1}_{q-1}(\T^d\times\R^d))$. 

Next, we prove uniqueness. If $f$ and $g$ are two solutions in $L^\infty([0,T],H^{2d+2}_q(\T^d\times\R^d))$ with the same initial data $f_{\rm in}$, then $h = f-g$ solves 
\[
\partial_t h + v\cdot \nabla_x h = Q( h, f) + Q(g, h),
\]
with initial data $h_{\rm in} \equiv 0$. By Lemma \ref{l:energy-estimate} with $R= Q(h,f)$, we have for any $t\in [0,T]$,
\[
\begin{split}
\|h(t)\|_{H^{2d+2}_q(\T^d\times\R^d)} &\leq C\int_0^t\|Q(h,f)(t')\|_{L^2_q(\T^d\times\R^d)} \dd t' \exp\left(C T\|g\|_{L^\infty([0,T], H^{2d+2}_q(\T^d\times\R^d))}\right)\\
&\leq C\int_0^t \|h\|_{H^{d+1}_q(\T^d\times\R^d)} \|f\|_{H^2_q(\T^d\times\R^d)}\dd t' \exp\left(C T\|g\|_{L^\infty([0,T], H^{2d+2}_q(\T^d\times\R^d))}\right),
\end{split}
\]
by Lemma \ref{l:simple}. Using our assumption that $f$ and $g$ are bounded in $L^\infty([0,T],H^{2d+2}_q(\T^d\times\R^d))$, together with interpolation and Gr\"onwall's inequality, we conclude $h\equiv 0$ on $[0,T]\times\T^d\times\R^d$. 


Now, let $q_0$ denote the value of $q$ used above, such that $f_{\rm in} \in H^{2d+2}_{q_0}(\T^d\times\R^d)$. The foregoing proof provides a time of existence $T>0$ depending on the $H^{2d+2}_{q_0}$ norm of the initial data.  For initial data with more regularity and decay, we would like to propagate all weighted Sobolev norms that are finite at $t=0$ to this same time interval $[0,T]$. Suppose that the $H^k_{q'}(\T^d\times\R^d)$ norm of $f_{\rm in}$ is finite, for some $k\geq 2d+2$ and $q'\geq q_0$. We prove by induction that
\begin{equation}\label{e:induction-local}
\| f(t)\|_{H^j_{q'}(\T^d\times\R^d)} \leq C_{j,q',T}, \quad t\in [0,T],
\end{equation}
for all $2d+2 \leq j \leq k$, for a family of constants $C_{j,q',T}$. The base case $j=2d+2$ follows immediately from the energy estimates of Lemma \ref{l:energy-estimate} with $\sigma = 1$, $g_\sigma = f$, $R=0$, $m=q_0$, and $q=q'$. Assuming \eqref{e:induction-local} holds for some $j$, we differentiate the equation \eqref{e:main} by a single partial derivative in $v$, giving
\[
\partial_t (\partial_{v_i} f) + v\cdot \nabla_x (\partial_{v_i} f) = Q(f,\partial_{v_i} f) + R,
\]
with 
\[
R = Q(\partial_{v_i} f, f) - \partial_{x_i} f = c_1 [\partial_{v_i} f \ast |\cdot|^{\gamma+2s}] (-\Delta)^s f + c_2[\partial_{v_i} f\ast |\cdot|^\gamma] f - \partial_{x_i} f.
\]
for constants $c_1, c_2$ defined above, depending only on $d$, $\gamma$, and $s$. For $\theta > d/2 + \gamma+2s$, we write
\[
\begin{split}
\|R(t)\|_{L^2_{q'}(\T^d\times\R^d)} &\lesssim \|(-\Delta)^s f\|_{L^\infty_{q'+\theta}(\T^d\times\R^d)} \|\partial_{v_i}f\ast|\cdot|^{\gamma+2s}\|_{L^2_{-\theta}(\T^d\times\R^d)}\\
&\quad + \|\partial_{v_i} f \ast |\cdot|^\gamma\|_{L^\infty(\T^d\times\R^d)} \|f\|_{L^2_{q'}(\T^d\times\R^d)} + \|\partial_{x_i} f\|_{L^2_{q'}(\T^d\times\R^d)}.
\end{split}
\]
To bound this expression, we apply Sobolev embedding to $(-\Delta)^s f$, Lemma \ref{l:young} to $\partial_{v_i}f \ast|\cdot|^{\gamma+2s}$ (since $q'\geq q_0>d/2+\gamma+2s$), and Lemma \ref{l:convolution}(a) to $\partial_{v_i} f \ast |\cdot|^\gamma$. Overall, we obtain
\[
\begin{split}
\|R(t)\|_{L^2_{q'}(\T^d\times\R^d)} &\lesssim \|(-\Delta)^s f\|_{H^{d+1}_{q'+\theta}(\T^d\times\R^d)} \|\partial_{v_i} f\|_{L^2_{q'}(\T^d\times\R^d)}\\
&\quad + \|\partial_{v_i} f\|_{L^\infty_{q'}(\T^d\times\R^d)} \|f\|_{L^2_{q'}(\T^d\times\R^d)} + \|\partial_{x_i} f\|_{L^2_{q'}(\T^d\times\R^d)}\\
&\lesssim \|f\|_{H^{2d+2}_{q'}(\T^d\times\R^d)}^2 + \|f\|_{H^{2d+2}_{q'}(\T^d\times\R^d)},
\end{split}
\]
after another Sobolev embedding applied to $\partial_{v_i} f$. We have used 
the fact that $d+1+2s < 2d+2$. 
We conclude $R$ is bounded in $L^\infty([0,T],L^2_{q'}(\T^d\times\R^d))$, so in particular it is in $L^2([0,T],L^2_{q'}(\T^d\times\R^d))$. 

We apply Lemma \ref{l:energy-estimate} to $\partial_{v_i} f$, with $j$ playing the role of $k$, which yields
\[
\begin{split}
\|\partial_{v_i} f(t)\|_{H^j_{q'}(\T^d\times\R^d)} &\leq \left(\|\partial_{v_j}f_{\rm in}\|_{H^j_{q'}(\T^d\times \R^d)} + C\int_0^t \|R(t')\|_{L^2_{q'}(\T^d\times\R^d)}^2\dd t'\right)\\
&\qquad \times \exp\left( C\int_0^t (1+ \|f(t')\|_{H^j_m(\T^d\times\R^{d})}) \dd t' \right).
\end{split}
\]
The inductive hypothesis \eqref{e:induction-local} and the assumption $\partial_{v_i} f_{\rm in} \in H^j_{q'}(\T^d\times\R^d)$ implies $\|\partial_{v_j} f(t)\|_{H^j_{q'}(\T^d\times\R^d)}$ is uniformly bounded on $[0,T]$. The same argument (but slightly simpler since the remainder $R$ has one fewer term) applies to first-order partial derivatives of $f$ in $x$. Therefore, the inequality \eqref{e:induction-local} holds for $j+1$, and the proof is complete.
\end{proof}

Finally, we address the spatially homogeneous case, Theorem \ref{t:short-time-existence}. On the one hand, existence for the homogeneous equation is a special case of the inhomogeneous problem. However, by repeating the argument above and using Sobolev embeddings in $\R^d$ instead of $\T^d\times\R^d$, we can improve the regularity requirement on $f_{\rm in}$ to $H^{d+3}_q(\R^d)$. The case analysis in the proof of Lemma \ref{l:energy-estimate} breaks down in a similar way, with the Sobolev exponent $\lceil (d+1)/2\rceil$ replacing $d + 1$. Other than this, the only changes to the proof are (i) fewer applications of H\"older's inequality needed because there is no $x$ dependence, (ii) the second term on the left in \eqref{e:energy-est} does not appear, and (iii) there is no moment loss in the proof of Lemma \ref{l:linear-existence}. We omit the details because the proof is strictly simpler than in the inhomogeneous case.

\appendix

\section{Technical Lemmas}\label{s:a}

First, we collect two standard upper bounds for convolutions with power functions $|\cdot|^\mu$. We omit the proof, which is elementary.

\begin{lemma}\label{l:convolution}
Let $f:\R^d\to \R$ and $-d < \mu < 0$. 
\begin{enumerate}
\item[(a)] If $f\in L^\infty_q(\R^d)$ for some $q> d+\mu$, then 
\[
\|f\ast |\cdot|^\mu\|_{L^\infty(\R^d)} \leq C\|f\|_{L^\infty_q(\R^d)},
\]
for a constant $C>0$ depending only on $d$, $\mu$, and $q$. 
\item[(b)] If $f\in L^2_q(\R^d)$ for some $q> d/2 + \mu$, then
\[
\|f\ast |\cdot|^\mu\|_{L^\infty(\R^d)} \leq C\|f\|_{L^2_q(\R^d)},
\]
for a constant $C>0$ depending only on $d$, $\mu$, and $q$. 
\end{enumerate}
\end{lemma}

%
%
%

The following is a weighted version of Young's inequality for convolutions. The proof is the same as  \cite[Lemma 4.2]{henderson2022existence}, which addressed the case $d=3$. 
\begin{lemma}\label{l:young}
If $-d < \mu < 0$, $m> d/2 + \mu$, and $\ell > d/2 + \mu + (d/2-m)_+$, then for any $g\in L^2_m(\R^d)$, there holds
\[
\left\| g\ast |\cdot|^\mu \right\|_{L^2_{-\ell}(\R^d)} \leq C \|g\|_{L^2_m(\R^d)},
\]
for a constant $C>0$ depending only on $d$, $\mu$, $m$, and $\ell$.
%
%
\end{lemma}

Next, we have an estimate similar to the previous lemma, that is somewhat sharper in terms of weights:

\begin{lemma}\label{l:young2}
For $-d < \mu \leq -d/2$, $\theta > \mu+d$, and $g\in L^2_\theta(\R^d)$, one has
\[
\|g\ast |\cdot|^\mu\|_{L^2(\R^d)} \leq C \|g\|_{L^2_\theta(\R^d)}.
\]
\end{lemma}
\begin{proof}
%
%
Dividing the convolution integral into regions with $|w|\geq |v|/2$ and $|w|< |v|/2$, we first have
\[
\begin{split}
\int_{\{|w|\geq |v|/2\}} g(w) |v-w|^\mu \dd w \leq \vv^{-\theta} \int_{\{|w|\geq |v|/2\}} \langle w\rangle^\theta g(w) |v-w|^\mu \dd w \leq \vv^{-\theta} [\langle \cdot\rangle^\theta g\ast |\cdot|^\mu](v),
\end{split}
\]
so that, by Lemma \ref{l:young} with $m=0$ and $\ell = \theta$,
\[
\int_{\R^d} \left( \int_{\{|w|\geq |v|/2\}} g(w) |v-w|^\mu \dd w\right)^2 \dd v \leq \int_{\R^d} \vv^{-2\theta} [\langle \cdot \rangle^\theta g\ast |\cdot|^\mu]^2 \dd v \leq \|\vv^\theta g\|_{L^2(\R^d)}^2,
\]
where we could apply Lemma \ref{l:young} because $\mu < -d/2$ and $\theta > \mu+d$. For the integral over $\{|w|< |v|/2\}$, we use $|v-w|\geq |v|/2$ to write
\[
\begin{split}
\int_{\{|w|< |v|/2\}} g(w) |v-w|^\mu \dd w &\leq  \int_{\{|w|< |v|/2\}}  g(w) |w|^\theta |w|^{-\theta} |v-w|^\mu \dd w \\
&\leq \|g |w|^\theta\|_{L^2(\R^d)} \left( \int_{\{|w|< |v|/2\}} |w|^{-2\theta} |v-w|^{2\mu} \dd w\right)^{1/2}\\
&\leq  \|g\|_{L^2_\theta(\R^d)} \vv^\mu \left(\int_{\{|w|< |v|/2\}} |w|^{-2\theta} \dd w \right)^{1/2}\\
&\leq  \|g\|_{L^2_\theta(\R^d)} \vv^{\mu -\theta+d/2},
\end{split}
\]
and therefore
\[
\int_{\R^d} \left( \int_{\{|w|< |v|/2\}} g(w) |v-w|^\mu \dd w\right)^2 \dd v \leq \|g\|_{L^2_\theta(\R^d)} \int_{\R^d} \vv^{2(\mu - \theta) + d} \dd v \lesssim \|g\|_{L^2_\theta(\R^d)},
\]
since $\theta > \mu + d$. Combining our estimates for $\{|w|\geq |v|/2\}$ and $\{|w|< |v|/2\}$ implies the conclusion of the lemma.
\end{proof}

Finally, we collect three estimates about the regularity of the functions $\vv^{-q}$ and $|v|^{-q}$. We omit the proofs, which are elementary.
\begin{lemma}\label{l:q-power}
For $q>0$ and $s\in (0,1)$,
\[
\begin{split}
(-\Delta)^s \vv^{-q} &\leq C \vv^{-q-2s},\\
\int_{\R^d} \frac{(\langle v+w\rangle^{-q} - \vv^{-q})^2}{|w|^{d+2s}} \dd w &\leq C \vv^{-2q-2s},\\
[|v|^{-q}]_{C^{s}(z)} &\leq C |z|^{-q-s}, \quad z\in \R^d,
\end{split}
\]
where the constant $C$ depends on $d$, $s$, and $q$. In the last inequality, we use the notation
\[
[g]_{C^{s}(z)} = \sup_{v\in \R^d} \frac{|g(v) -g(z)|}{|v-z|^{s}}.
\]
\end{lemma}

%
%
%

\bibliographystyle{plain}
\bibliography{isotropic}

\end{document}